\documentclass{amsart}
\usepackage{amsfonts,amssymb}

\newtheorem{theorem}{Theorem}
\newtheorem{lemma}{Lemma}
\newtheorem{corollary}{Corollary}

\newcommand{\integers}{{\mathbb Z}}
\newcommand{\realnos}{{\mathbb R}}

\def\ov{\overline}

\def\Beta{{\rm B}}
\def\Nu{{\rm N}}
\def\Mu{{\rm M}}
\def\Tau{{\rm T}}
\def\Eta{{\rm H}}

\def\Kappa{{\rm K}}
\def\Iota{{\rm I}}

\begin{document}

\title{Closed Flat Riemannian 4-Manifolds}

\author{Thomas P. Lambert, John G. Ratcliffe, and Steven T. Tschantz}

\address{Department of Mathematics, Vanderbilt University, Nashville, TN 37240
\vspace{.1in}}

\email{j.g.ratcliffe@vanderbilt.edu}


\date{}


\begin{abstract}
In this paper we describe the classification of all the geometric fibrations of a closed flat Riemannian 4-manifold over a 1-orbifold. 
\end{abstract}

\maketitle

\section{Introduction} 
In this paper, 
a {\it closed flat $n$-manifold} is a compact, connected, Riemannian $n$-manifold 
of constant zero sectional curvature. 
There are only finitely many homeomorphism classes of  closed flat $n$-manifolds 
for each dimension $n$. 
In dimensions $1,2,3,4$, there are $1,2,10, 74$ homeomorphism equivalence classes 
of closed flat manifolds, respectively. 

The classification of closed flat 4-manifolds 
was first achieved by Brown et al.  in their 
1978 computer assisted classification ~\cite{B-Z} 
of all the isomorphism classes of 4-dimensional crystallographic groups. 
Brown et al. ~found 74 homeomorphism equivalence classes of closed flat 4-manifolds, 
with 27 classes of orientable manifolds and 47 classes of nonorientable manifolds.  
No other topological information about these manifolds was given by Brown et al. ~in \cite{B-Z}. 

Earlier Levine had published in his 1970 PhD thesis \cite{Levine} a computer assisted classification  
of 75 homeomorphism equivalence classes of closed flat 4-manifolds, 
with 27 classes of orientable manifolds and 48 classes of nonorientable manifolds. 
Levine's thesis is flawed by a duplication in the list of nonorientable manifolds. 
Levine's thesis does have an advantage over \cite{B-Z} 
in that it lists the first homology group of each manifold, 
which is a useful topological invariant.   

In his 1995 paper \cite{H}, Hillman described an algebraic topological classification 
of closed flat 4-manifolds based on his proof  
that every such manifold is either a fiber bundle over the circle or a 
union of two twisted $I$-bundles intersecting along their boundaries. 
The latter case is equivalent to an orbifold fibration of the manifold over a closed interval. 
Hillman found 74 homeomorphism equivalence classes of closed flat 4-manifolds. 
Hillman made no attempt in \cite{H} to make his classification correspond to the previous classifications 
of Brown et al. ~and Levine. 

In his 2007 PhD thesis \cite{Lambert}, 
Lambert found the correspondence between 
the Brown et al., Levine, and Hillman classifications and the duplication 
in Levine's classification.  
Today one can use the computer program CARAT \cite{Carat} to identify 
a closed flat 4-manifold from an affine representation. 
However, CARAT does not tell you much about the geometry or topology of an individual manifold. 

In this paper we describe our classification of all the geometric orbifold fibrations of 
a closed flat 4-manifold over a 1-orbifold (a circle or a closed interval). 
We prove that there are 214 affine equivalence classes of geometric orbifold fibrations of closed flat 4-manifolds over a 1-orbifold, 
with 98 classes over a circle and 116 classes over a closed interval.  We also identify the total spaces of the fibrations. 
All 98 classes of fiberings over a circle were found by Hillman in \cite{H} but they were not all sorted by their total spaces. 
A closed flat 4-manifold fibers over a circle if and only if the first Betti number of the manifold is positive. 

As part of his classification in \cite{H}, Hillman also described 8 affine equivalence classes of geometric orbifold fibrations of closed flat 4-manifolds, with 
first Betti number 0, over a closed interval.  We show that, in fact, there are 12 such equivalence classes. 
This implies that Hillman's argument for the classification of closed flat 4-manifolds, with first Betti number 0, in \cite{H} has a gap, 
which is filled by our classification.  See Example 5 in this regard.

Closed flat $n$-manifolds occur naturally in the theory of hyperbolic manifolds as horospherical cross-sections 
of the cusps of hyperbolic $(n+1)$-manifolds of finite volume. 
We first encountered closed flat 4-manifolds in our classification \cite{R-T5} of certain hyperbolic 5-manifolds of finite volume. 
Initially we had a hard time to identity the closed flat 4-manifolds that appeared as horospherical cross-sections of the cusps of our 5-manifolds. 
What we observed is that these 4-manifolds were fiber bundles over a circle, and with this extra information, we were able to identify them. 
The main idea of this paper is that a geometric fibration of a closed flat $4$-manifold over a 1-orbifold carries enough information 
to identify the manifold. 
It is the purpose of this paper to explain how to use this information to identify a closed flat 4-manifold. 
It is our desire that this paper will put the topology and geometry of closed flat 4-manifolds in good order. 

\tableofcontents

\section{Geometric Fibrations of Closed Flat Manifolds} 

Let $E^n$ be Euclidean $n$-space. 
A map $\phi:E^n\to E^n$ is an isometry of $E^n$ 
if and only if there is an $a\in E^n$ and an $A\in {\rm O}(n)$ such that 
$\phi(x) = a + Ax$ for each $x$ in $E^n$. 
We shall write $\phi = a+ A$. 
In particular, every translation $\tau = a + I$ is an isometry of $E^n$.

An $n$-dimensional {\it crystallographic group} ($n$-space group) is a discrete group 
of isometries $\Gamma$ of $E^n$ such that the orbit space $E^n/\Gamma$ is compact. 
If $\Gamma$ is a torsion-free $n$-space group, 
then the orbit space $E^n/\Gamma$ is called an $n$-dimensional {\it Euclidean space-form}. 
A space-form $E^n/\Gamma$ is a closed flat $n$-manifold. 
Conversely, if $M$ is a closed flat $n$-manifold, then there is a torsion-free 
$n$-space group such that $M$ is isometric to the space-form $E^n/\Gamma$. 
The homeomorphism classification of closed flat $n$-manifolds 
is equivalent to the isomorphism classification of torsion-free $n$-space groups.

Let $\Gamma$ be an $n$-space group. 
Define $\eta:\Gamma \to {\rm O}(n)$ by $\eta(a+A) = A$. 
Then $\eta$ is a homomorphism whose kernel is the group 
of translations in $\Gamma$. 
The image of $\eta$ is a finite group called the 
{\it point group} of $\Gamma$.  If $\Gamma$ is torsion-free, the 
point group of $\Gamma$ is also called the {\it holonomy group} of $\Gamma$. 

Let $\Eta$ be a subgroup of an $n$-space group $\Gamma$. 
Define the {\it span} of $\Eta$ by the formula
$${\rm Span}(\Eta) = {\rm Span}\{a\in E^n:a+I\in \Eta\}.$$
Note that ${\rm Span}(\Eta)$ is a vector subspace $V$ of $E^n$.  
Let $V^\perp$ denote the orthogonal complement of $V$ in $E^n$. 

\begin{theorem}  {\rm (Theorem 2 \cite{R-T})\ } 
Let ${\rm N}$ be a normal subgroup of an $n$-space group $\Gamma$, 
and let $V = {\rm Span}(\Nu)$. 
\begin{enumerate}
\item If $b+B\in\Gamma$, then $BV=V$. 
\item If $a+A\in \Nu$,  then $a\in V$ and $ V^\perp\subseteq{\rm Fix}(A)$. 
\item The group $\Nu$ acts effectively on each coset $V+x$ of $V$ in $E^n$ 
as a space group of isometries of $V+x$. 
\end{enumerate}
\end{theorem}

Let $\Gamma$ be an $n$-space group. 
The {\it dimension} of $\Gamma$ is $n$. 
If $\Nu$ is a normal subgroup of $\Gamma$, 
then $\Nu$ is a $m$-space group with $m= \mathrm{dim}(\mathrm{Span}(\Nu))$ 
by Theorem 1(3). 

\vspace{.15in}
\noindent{\bf Definition:}
Let $\Nu$ be a normal subgroup $\Nu$ of an $n$-space group $\Gamma$, and let $V = {\rm Span}(\Nu)$.  
Then $\Nu$ is said to be a {\it complete normal subgroup} of $\Gamma$ if 
$$\Nu= \{a+A\in \Gamma: a\in V\ \hbox{and}\ V^\perp\subseteq{\rm Fix}(A)\}.$$
\noindent{\bf Remark 1.}
If $\Nu$ is a normal subgroup of an $n$-space group $\Gamma$, 
then $\Nu$ is contained in a unique complete normal subgroup $\hat{\Nu}$ 
of $\Gamma$ such that $\Nu$ has finite index in $\hat{\Nu}$. 
The group $\hat{\Nu}$ is called the {\it completion} of $\Nu$ in $\Gamma$.  

\begin{lemma} {\rm (Lemma 1 \cite{R-T})\ }
Let $\Nu$ be a complete normal subgroup of an $n$-space group $\Gamma$, 
and let $V={\rm Span}(\Nu)$. 
Then $\Gamma/\Nu$ acts effectively as a space group of isometries of $E^n/V$ 
by the formula
$({\rm N}(b+B))(V+x) = V+ b+Bx.$
\end{lemma}

\noindent{\bf Remark 2.} A normal subgroup $\Nu$ of a space group $\Gamma$ is complete 
precisely when $\Gamma/\Nu$ is a space group by Theorem 5 of \cite{R-T}.

\vspace{.15in}

A {\it flat $n$-orbifold} is a $(E^n,{\rm Isom}(E^n))$-orbifold 
as defined in \S 13.2 of Ratcliffe \cite{R}. 
A connected flat $n$-orbifold has a natural inner metric space structure. 
If $\Gamma$ is a discrete group of isometries of $E^n$, 
then its orbit space $E^n/\Gamma = \{\Gamma x: x\in E^n\}$ 
is a  connected, complete, flat $n$-orbifold, 
and conversely if $M$ is a connected, complete, flat $n$-orbifold, 
then there is a discrete group $\Gamma$ of isometries of $E^n$  
such that $M$ is isometric to $E^n/\Gamma$ by Theorem 13.3.10 of \cite{R}. 

\vspace{.15in}
\noindent{\bf Definition:}
A flat $n$-orbifold $M$ {\it geometrically fibers} over a flat $m$-orbifold $B$, 
with {\it generic fiber} a flat $(n-m)$-orbifold $F$, if there is a surjective map $\eta: M \to B$, 
called the {\it fibration projection},  
such that for each point $y$ of $B$,  
there is an open metric ball $B(y,r)$ of radius $r > 0$ centered at $y$ in $B$
such that $\eta$ is isometrically equivalent on $\eta^{-1}(B(y,r))$
to the natural projection $(F\times B_y)/G_y \to B_y/G_y$, 
where $G_y$ is a finite group acting diagonally on $F\times B_y$, isometrically on $F$,  
and effectively and orthogonally on an open metric ball $B_y$ in $E^m$ of radius $r$. 
This implies that the fiber $\eta^{-1}(y)$ is isometric to $F/G_y$. 
The fiber $\eta^{-1}(y)$ is said to be {\it generic} if $G_y = \{1\}$  
or {\it singular} if $G_y$ is nontrivial. 

\vspace{.15in}
In the above definition, if $M$ is a flat $n$-manifold, 
then the generic fiber $F$ must be a flat $(n-m)$-manifold and the group $G_y$ must act freely on  $F\times B_y$, and so $G_y$ acts freely 
on $F$ for each point $y$ of $B$.  Consequently, each fiber $\eta^{-1}(y)$ is a flat $(n-m)$-manifold. 

If the {\it base} $B$ is a flat manifold, then the fibration projection $\eta:M \to B$ is a fiber bundle projection. 
We call a geometric fibration with flat manifold base a {\it flat fibering}. 

If $M$ is a flat $n$-manifold and the base $B$ is the closed interval $\Iota = [0,1]$. 
Then $\eta^{-1}([0,1/2])$ and $\eta^{-1}([1/2,1])$ are twisted $\Iota$-bundles over 
the singular fibers $\eta^{-1}(0)$ and $\eta^{-1}(1)$, respectively, 
and $M$ is the union of these two twisted $\Iota$-bundles over their common boundary $\eta^{-1}(1/2)$. 

\begin{theorem} {\rm (Theorem 4 \cite{R-T})} 
Let ${\rm N}$ be a complete normal subgroup of an $n$-space group $\Gamma$, 
and let $V = {\rm Span}({\rm N})$.  
Then the flat orbifold $E^n/\Gamma$ geometrically fibers over the flat orbifold 
$(E^n/V)/(\Gamma/{\rm N})$ with generic fiber the flat orbifold $V/{\rm N}$ and 
fibration projection $\eta_V: E^n/\Gamma \to (E^n/V)/(\Gamma/{\rm N})$ defined by the formula 
$\eta_V(\Gamma x) = (\Gamma/\Nu)(V+x).$
\end{theorem}

Let $\Nu$ be a complete normal subgroup of an $n$-space group $\Gamma$.  
We call the map $\eta_V:E^n/\Gamma\to (E^n/V)/(\Gamma/\Nu)$ 
defined in Theorem 2, the {\it fibration projection determined by} $\Nu$.

Let $M_i$ be a connected complete flat $n$-orbifold for $i=1,2$. 
Suppose $M_i$ geometrically fibers over a flat $m$-orbifold $B_i$ 
with fibration projection 
$\eta_i:M_i\to B_i$ for $i=1,2$. 
The fibration projections $\eta_1$ and $\eta_2$ are said to be 
{\it isometrically equivalent} if there are isometries $\alpha:M_1\to M_2$ 
and $\beta:B_1\to B_2$ such that $\beta\eta_1=\eta_2\alpha$.

\begin{theorem} {\rm (Theorem 8 \cite{R-T})} 
Let $M$ be a compact connected flat $n$-orbifold. 
If $M$ geometrically fibers over a flat $m$-orbifold $B$ 
with fibration projection $\eta:M\to B$,  
then there exists an $n$-space group $\Gamma$ 
and a complete normal subgroup $\Nu$ of $\Gamma$   
such that $\eta$ is isometrically equivalent to the fibration 
projection $\eta_V: E^n/\Gamma\to (E^n/V)/(\Gamma/\Nu)$ determined by $\Nu$. 
\end{theorem} 

An {\it affinity} $\alpha$ of $E^n$ is a map $\alpha: E^n \to E^n$ of the form $\alpha = a+A$ with $a \in E^n$ and $A\in \mathrm{GL}(n,\integers)$. 
The set of all affinities of $E^n$ forms a group $\mathrm{Aff}(E^n)$ that contains the group of all isometries of $E^n$ as a subgroup. 

There exist covering projections $\xi_i: E^n \to M_i$ for $i = 1, 2$ that are local isometries. 
A map $\alpha: M_1 \to M_2$ is an {\it affinity} if there is an affinity $\tilde\alpha$ of $E^n$ 
such that $\alpha\xi_1 = \xi_2\tilde\alpha$. 

The fibration projections $\eta_1$ and $\eta_2$ are said to be 
{\it affinely equivalent} if there are affinities $\alpha:M_1\to M_2$ 
and $\beta:B_1\to B_2$ such that $\beta\eta_1=\eta_2\alpha$. 

\begin{theorem} {\rm (Theorem 10 \cite{R-T})} 
Let $\Nu_i$ be a complete normal subgroup of an $n$-space group $\Gamma_i$ for $i=1,2$, 
and let $\eta_i:E^n/\Gamma_i \to (E^n/V_i)/(\Gamma_i/\Nu_i)$ 
be the fibration projections determined by $\Nu_i$ for $i=1,2$. 
Then the following are equivalent:
\begin{enumerate}
\item The fibration projections $\eta_1$ and $\eta_2$ are affinely equivalent. 
\item There is an affinity $\phi$ of $E^n$ such that 
$\phi\Gamma_1\phi^{-1} = \Gamma_2$ and $\phi\Nu_1\phi^{-1} = \Nu_2$. 
\item There is an isomorphism $\psi:\Gamma_1\to \Gamma_2$ such that $\psi(\Nu_1) = \Nu_2$.
\end{enumerate}
\end{theorem}

An {\it affine $n$-space group} is a group $\Gamma$ of affinities of $E^n$ that is conjugate to an $n$-space group in $\mathrm{Aff}(E^n)$. 
As is customary in the theory of crystallographic groups, we shall in practice work only with affine $n$-space groups 
whose point groups are subgroups of $\mathrm{GL}(n,\integers)$.

\section{The Generalized Calabi Construction}  

Let $\Nu$ be a complete normal subgroup of an $n$-space group $\Gamma$, 
let $V = \mathrm{Span}(\Nu)$, and let $V^\perp$ be the orthogonal complement of $V$ in $E^n$.  
Euclidean $n$-space $E^n$ decomposes as the Cartesian product $E^n = V \times V^\perp$. 
Let $b+B\in\Gamma$ and let $x\in E^n$.  Write $b= c+d$ with $c\in V$ and $d\in V^\perp$. 
Write $x = v+w$ with $v\in V$ and $w\in V^\perp$. Then 
$$(b+B)x = b+Bx = c+d + Bv + Bw = (c+Bv) + (d+Bw).$$
Hence the action of $\Gamma$ on $E^n$ corresponds to the diagonal action of $\Gamma$ 
on $V\times V^\perp$ defined by the formula
$$(b+B)(v,w) = (c+Bv,d+Bw).$$
Here $\Gamma$ acts on both $V$ and $V^\perp$ via isometries. 
The kernel of the corresponding homomorphism from $\Gamma$ to $\mathrm{Isom}(V)$ 
is the group
$$\Kappa = \{b+B\in\Gamma: b \in V^\perp\ \hbox{and}\ V \subseteq \mathrm{Fix}(B)\}.$$
We call $\Kappa$ the {\it kernel of the action} of $\Gamma$ on $V$. 
The group $\Kappa$ is a normal subgroup of $\Gamma$. 
The action of $\Gamma$ on $V$ induces an effective action of $\Gamma/\Kappa$ on $V$ 
via isometries.  
Note that $\Nu\cap\Kappa = \{I\}$,  
and each element of $\Nu$ commutes with each element of $\Kappa$.   
Hence $\Nu\Kappa$ is a normal subgroup of $\Gamma$,  
and $\Nu\Kappa$ is the direct product of $\Nu$ and $\Kappa$. 

The action of $\Nu$ on $V^\perp$ is trivial and the action of $\Kappa$ on $V$ is trivial. 
Hence $E^n/\Nu\Kappa$ decomposes as the Cartesian product 
$E^n/\Nu\Kappa = V/\Nu \times V^\perp/\Kappa.$

The action of $\Gamma/\Nu\Kappa$ on $E^n/\Nu\Kappa$ corresponds 
to the diagonal action of $\Gamma/\Nu\Kappa$ on $V/\Nu \times V^\perp/\Kappa$ via isometries 
defined by the formula
$$(\Nu\Kappa(b+B))(\Nu v,\Kappa w) = (\Nu(c+Bv),\Kappa(d+Bw)).$$
The generalized Calabi construction is the following theorem.
\begin{theorem} {\rm (Theorem 3 \cite{R-TII})} 
Let $\Nu$ be a complete normal subgroup of an $n$-space group $\Gamma$, 
and let $\Kappa$ be the kernel of the action of $\Gamma$ on $V= \mathrm{Span}(\Nu)$. 
Then the map
$$\chi: E^n/\Gamma \to (V/\Nu\times V^\perp/\Kappa)/(\Gamma/\Nu\Kappa)$$
defined by $\chi(\Gamma x) = (\Gamma/\Nu\Kappa)(\Nu v, \Kappa w)$, 
with  $x = v + w$ and $v\in V$ and $w\in V^\perp$, is an isometry. 
\end{theorem}

We call $\Gamma/\Nu\Kappa$ the {\it structure group} 
of the geometric fibration of $E^n/\Gamma$ 
determined by the complete normal subgroup $\Nu$ of $\Gamma$. 

The group $\Nu$ is the kernel of the action of $\Gamma$ on $V^\perp$, and so we have 
an effective action of $\Gamma/\Nu$ on $V^\perp$. 
The natural projection from $V/\Nu\times V^\perp/\Kappa$ to $V^\perp/\Kappa$ 
induces a continuous surjection
$$\pi^\perp: (V/\Nu\times V^\perp/\Kappa)/(\Gamma/\Nu\Kappa) \to V^\perp/(\Gamma/\Nu).$$
Orthogonal projection from $E^n$ to $V^\perp$ induces an isometry from $E^n/V$ to $V^\perp$ 
which in turn induces an isometry 
$$\psi^\perp: (E^n/V)/(\Gamma/\Nu) \to V^\perp/(\Gamma/\Nu).$$
\begin{theorem} {\rm (Theorem 4 \cite{R-T})} 
The following diagram commutes
\[\begin{array}{ccc}
E^n/\Gamma & {\buildrel \chi\over\longrightarrow} &
(V/\Nu\times V^\perp/\Kappa)/(\Gamma/\Nu\Kappa) \\
\eta_V \downarrow \  & & \downarrow\pi^\perp \\
(E^n/V)/(\Gamma/\Nu) & {\buildrel \psi^\perp\over\longrightarrow}  & V^\perp/(\Gamma/\Nu). 
\end{array}\] 
\end{theorem}

If $\mathrm{Span}(\Kappa) = V^\perp$, then $\Kappa$ is a complete normal subgroup 
of $\Gamma$ called the {\it orthogonal dual} of $\Nu$ in $\Gamma$, 
and we write $\Kappa = \Nu^\perp$. 
The group $\Nu$ has an orthogonal dual in $\Gamma$ if and only if the structure group 
$\Gamma/\Nu\Kappa$ is finite by Theorem 5 of \cite{R-TII}.

Suppose $\Kappa = \Nu^\perp$.  
The natural projection from $V/\Nu\times V^\perp/\Kappa$ to $V/\Nu$ 
induces a continuous surjection
$$\pi: (V/\Nu\times V^\perp/\Kappa)/(\Gamma/\Nu\Kappa) \to V/(\Gamma/\Kappa).$$
Orthogonal projection from $E^n$ to $V$ induces an isometry from $E^n/V^\perp$ to $V$ 
which in turn induces an isometry 
$$\psi: (E^n/V^\perp)/(\Gamma/\Kappa) \to V/(\Gamma/\Kappa).$$
The next corollary follows from Theorem 6 by reversing the roles of $\Nu$ and $\Kappa$. 
\begin{corollary} 
The following diagram commutes
\[\begin{array}{ccc}
E^n/\Gamma & {\buildrel \chi\over\longrightarrow} &
(V/\Nu\times V^\perp/\Kappa)/(\Gamma/\Nu\Kappa) \\
\eta_{V ^\perp}\downarrow\ \ \  & & \downarrow\pi \\
(E^n/V^\perp)/(\Gamma/\Kappa) & {\buildrel \psi\over\longrightarrow}  & V/(\Gamma/\Kappa). 
\end{array}\] 
\end{corollary}

Theorem 6 says that the fibration projection $\eta_V: E^n/\Gamma \to (E^n/V)/(\Gamma/\Nu)$  
is equivalent to the projection 
$\pi^\perp: (V/\Nu\times V^\perp/\Kappa)/(\Gamma/\Nu\Kappa) \to V^\perp/(\Gamma/\Nu)$ 
induced by the projection on the second factor of $V/\Nu\times V^\perp/\Kappa$;   
while Corollary 1 says that the orthogonally dual fibration projection 
$\eta_{V^\perp}:E^n/\Gamma \to(E^n/V^\perp)/(\Gamma/\Kappa)$ is equivalent to the projection 
$\pi: (V/\Nu\times V^\perp/\Kappa)/(\Gamma/\Nu\Kappa) \to V/(\Gamma/\Kappa)$ 
induced by the projection on the first factor of $V/\Nu\times V^\perp/\Kappa$. 
The generalized Calabi construction reveals that the relationship between 
a flat orbifold fibration and its orthogonally dual flat orbifold fibration 
is similar to that of the two sides of the same coin.

\section{The Original Calabi Construction}  

Let $G$ be a finitely generated group. 
The set 
$$I(G) = \big\{g\in G : \hbox{there exists an integer}\ k > 0\ \hbox{such that} \ g^k\in [G,G]\big\}$$
is a characteristic subgroup of $G$ that contains $[G,G]$ 
and corresponds to the torsion subgroup of $G/[G,G]$, 
and so $G/I(G)$ is a free abelian group of rank $\beta_1(G)$, 
the first Betti number of $G$. 

A flat $m$-{\it torus} is a closed flat $m$-manifold affinely equivalent to $(S^1)^m$. 
If $m = 0$, we define $(S^1)^m$ to be a point. 
A 1-torus is a circle.

Let $M = E^n/\Gamma$ be a space-form.  
Then $I(\Gamma)$ is a complete normal subgroup of $\Gamma$, 
since $\Gamma/I(\Gamma)$ is a $m$-space group where $m = \beta_1(\Gamma)$. 
Let $V = \mathrm{Span}(I(\Gamma))$. 
Define $I(M) = V/I(\Gamma)$. 
Then $I(M)$ is a flat $(n-m)$-manifold. 
The flat $n$-manifold $M$ geometrically fibers, with generic fiber $I(M)$, over the flat $m$-torus $V^\perp/(\Gamma/I(\Gamma))$ by Theorems 2 and 4. 
We call $I(M)$ the {\it characteristic fiber} of $M$.

Let $Z(\Gamma)$ be the {\it center} of $\Gamma$. 
Then every element of $Z(\Gamma)$ is a translation and $Z(\Gamma)$ 
is the orthogonal dual of $I(\Gamma)$ in $\Gamma$ by Theorem 15 of \cite{R-T}. 
Let $Z(M) = V^\perp/Z(\Gamma)$.  
Then $Z(M)$ is a $m$-torus where $m = \beta_1(\Gamma)$. 
The original Calabi construction, Theorem 3.6.3 \cite{Wolf}, is the decomposition
$$M = (I(M) \times Z(M))/(\Gamma/I(\Gamma)Z(\Gamma)).$$
The structure group $\Gamma/I(\Gamma)Z(\Gamma)$ is a finite abelian group, since $\Gamma/I(\Gamma)$ is abelian. 

Let $J(\Gamma) = \Gamma/Z(\Gamma)$ and $J(M) = V/J(\Gamma)$.  Then $J(M)$ is a flat $(n-m)$-orbifold  
and $M$ geometrically fibers, with generic fiber $Z(M)$, over $J(M)$. 
We call $J(M)$ the {\it characteristic base} of $M$.  
Usually $J(M)$ is not a flat manifold.  By Theorem 3 of Farkas \cite{Farkas} and Theorem 6 of \cite{R-T}, 
we have that $\beta_1(J(\Gamma)) = 0$. 

The original Calabi construction is canonical.  In particular $I(\Gamma), J(\Gamma)$ and 
the structure group $\Gamma/I(\Gamma)Z(\Gamma)$ are isomorphism invariants of $\Gamma$. 
Moreover the affine equivalence classes of $I(M)$ and $J(M)$ are invariants of the affine equivalence class of $M$, 
in other words, $I(M)$ and $J(M)$ are topological invariants of $M$. 

\vspace{.15in}
\noindent{\bf Remark 3.} It is important to note that a generalized Calabi construction is a geometric construction, and so in general is not canonical, 
that is, it may not be preserved by an affine deformation of the manifold.  
The orthogonally dual geometric fibrations do persist under an affine deformation, but they may no longer be orthogonal to each other. 
The orthogonally dual geometric fibrations in the original Calabi construction remain orthogonal to each other under an affine deformation. 

\vspace{.15in}
Through out this paper, $e_1, e_2, \ldots, e_n$ are the standard basis vectors of $E^n$,
and $t_i = e_i +I$ is the standard translation corresponding to $e_i$ 
for each $i = 1, \ldots, n$.

\vspace{.15in}
\noindent{\bf Example 1.}
The flat Klein bottle $K^2$ is the space-form $E^2/\Gamma$ 
where $\Gamma = \langle t_1, t_2,  \alpha\rangle$, 
and $\alpha  = \frac{1}{2}e_1+ {\rm diag}(1,-1)$. 
The group $\Gamma$ has the presentation $\langle t_2,\alpha\ |\ \alpha t_2\alpha^{-1} = t_2^{-1}\rangle$.  
Hence $I(\Gamma) = \langle t_2\rangle$, and so $I(K^2) = S^1$. 
Therefore $K^2$ geometrically fibers over the circle with circle fibers by Theorem 1. 
Now $Z(\Gamma) = \langle t_1\rangle$.  
As $\alpha^2 = t_1$, we have that $J(\Gamma)$ is an infinite dihedral group, 
with order 2 generators $\langle t_1\rangle \alpha$ and $\langle t_1\rangle t_2\alpha$,  
and so $J(K^2)$ is a closed interval.  
Therefore $K^2$ geometrically fibers over a closed interval with generic fiber 
the circle $V/Z(\Gamma)$ where $V = \mathrm{Span}(Z(\Gamma)) = \mathrm{Span}\{e_1\}$. 

The structure group $\Gamma/I(\Gamma)Z(\Gamma) = \Gamma/\langle t_1,t_2\rangle$ is generated by $\langle t_1,t_2\rangle \alpha$, 
and so has order 2. 
Now $\langle t_1,t_2\rangle \alpha$ acts on the circle $V/Z(\Gamma)$ by a half-turn,  
and $\langle t_1,t_2\rangle \alpha$  acts on the circle $V^\perp/I(\Gamma)$ by a reflection.
Consequently $K^2$ is the union of two isometric M\"obius bands 
joined along their common boundary circle. 
This decomposition of $K^2$ corresponds to the decomposition of $\Gamma$ as 
the free product of the infinite cyclic groups $\langle \alpha\rangle$ and $\langle t_2\alpha\rangle$ 
amalgamated over their index two subgroup $\langle t_1\rangle$. 

\section{The Closed Flat 3-Manifolds}  

We denote the six orientable, closed, flat, 3-manifold homeomorphism classes 
by $O^3_1, \ldots, O^3_6$ and the four nonorientable, closed, flat, 
3-manifold homeomorphism classes by $N^3_1,\ldots, N^3_4$ in the order given by Hantzsche and Wendt \cite{H-W} and Wolf \cite{Wolf}.  
The manifolds are ordered inversely with respect to their first Betti number. 
Thus $O^3_1$ is the 3-torus and $O^3_6$ is the Hantzsche-Wendt 3-manifold. 
The characteristic fibers and bases, structure group, first homology groups,
and holonomy groups of the  closed, flat 3-manifolds are listed in Tables 1 and 2. 
We denote a point by $E^0$, the circle by $S^1$, the closed interval by $\Iota$, the 2-torus by $T^2$, and the Klein bottle by $K^2$. 
We use Conway's notation \cite{Conway} for the compact, connected, flat 2-orbifolds 
in the $J(M)$ column. 
The IT order for these 17 orbifolds is $\circ$, $2222$, $\ast\ast$, $\times\times$, $\ast\times$, $\ast 2222$, $22\ast$, $22\times$, 
$2{\ast}22$, $442$, $\ast 442$, $4{\ast}2$, $333$, $\ast 333$, $3{\ast} 3$, $632$, $\ast 632$. 

In the structure group and holonomy group columns, $C_n$ is the multiplicative cyclic group of order $n$. 
We use exponential notation so that $(C_n)^2 = C_n\times C_n$. 
In the first homology column, $\mathbb Z_n$ is the additive cyclic group of order $n$. 
We use exponential notation so that $(\mathbb Z_n)^2 = \mathbb Z_n\oplus\mathbb Z_n$. 

All orientable closed 3-manifolds are spin manifolds and ``Y" in the spin column 
of Table 1 stands for ``yes".   
The heading ODC in Table 2 stands for ``orientable double cover". 
The BBNWZ column gives the Brown et al.\ designation in \cite{B-Z} for the corresponding 
torsion-free 3-space group.  The IT column gives the IT number in \cite{B-Z} for the corresponding 
torsion-free 3-space group. 

All the closed flat 3-manifolds are fiber bundles over a circle except for the Hantzsche-Wendt 3-manifold $O^3_6$.  
The flat 3-manifold $O^3_6$ does geometrically fiber over a closed interval in a way that is unique up to isometry. 
This fibering is worth examining, since it best reveals the geometry and topology of $O^3_6$. 

\vspace{.15in}
\noindent{\bf Example 2.}
The Hantzsche-Wendt 3-manifold $O^3_6$ is the space-form $E^3/\Gamma$ where $\Gamma$ is the torsion-free 3-space group 3/1/1/4 in \cite{B-Z}.  
After transposing the 1st and 3rd coordinates $\Gamma = \langle t_1, t_2, t_3, \alpha, \beta\rangle$ 
where $\alpha  = \frac{1}{2}e_1+\frac{1}{2}e_3 + \mathrm{diag}(1,-1,-1)$ and $\beta = \frac{1}{2}e_2 + \mathrm{diag}(-1,1,-1)$. 
Let $\Nu = \langle t_1, t_2\rangle$ and $\Kappa = \langle t_3\rangle$, and $V = \mathrm{Span}(\Nu)$.   
Then $\Nu$ and $\Kappa$ are complete normal subgroups of $\Gamma$ with $\Kappa = \Nu^\perp$. 
The structure group $\Gamma/\Nu\Kappa$ is a dihedral group of order 4 generated by $\Nu\Kappa\alpha$ and $\Nu\Kappa\beta$. 
Now $\Nu\Kappa\alpha$ acts on the torus $V/\Nu$ via $\frac{1}{2}e_1 +\mathrm{diag}(1,-1)$  in the first two coordinates,  
and $\Nu\Kappa\beta$ acts on $V/\Nu$ via $\frac{1}{2}e_2 +\mathrm{diag}(-1,1)$. 
Moreover  $\Nu\Kappa\alpha$ acts on the circle $V^\perp/\Kappa$ via the reflection $\frac{1}{2}e_3 - I$ and 
$\Nu\Kappa\beta$ acts on $V^\perp/\Kappa$ via the reflection $-I$. 

The quotient group $\Gamma/\Nu$ is an infinite dihedral group with order 2 generators $\Nu\alpha$ and $\Nu \beta$. 
Therefore $O^3_6$ geometrically fibers over a closed interval with generic fiber the torus $V/\Nu$. 
Hence $O^3_6$ is the union of two twisted $I$-bundles joined over their common torus boundary.  
More precisely, from the action of $\Nu\Kappa\alpha$ and $\Nu\Kappa\beta$ on $V/\Nu$, we see that $O^3_6$ is formed by identifying the boundaries 
of two isometric copies of a twisted $I$-bundle over the Klein bottle, with a square torus boundary, by an isometry induced by interchanging coordinates. 
This decomposition of $O^3_6$ corresponds to the decomposition of $\Gamma$ 
as the free product with amalgamation of the Klein bottle groups $\langle t_2, \alpha\rangle$ and $\langle t_1, \beta\rangle$ 
over their index 2 torus subgroup $\langle t_1, t_2\rangle$. 

The quotient group $\Gamma/\Kappa$ is a 2-space group of type $22\times$. 
Hence the geometric fibration of $O^3_6$ determined by $\Kappa$ is a Seifert fibration over a flat 2-orbifold of type $22\times$, that is, a flat projective pillow. 
This Seifert fibration is orthogonally dual to the geometric fibration of $O^3_6$ determined by $\Nu$.

\begin{table} 

\begin{tabular}{c|c|c|c|c|c|c|c|r}	
	$M$ & $I(M)$ & $J(M)$ & str. grp. & $H_1(M)$  & hol. grp. & Spin & BBNWZ & IT \\ \hline	
	$O^3_1$ & $E^0$ & $E^0$ & $C_1$ & 
	  	$\mathbb Z^3$ & $C_1$ & \phantom{\Big(}Y\phantom{\Big)} & 1/1/1/1 & 1 \\ \hline
	$O^3_2$ & $T^2$ & $2222$ & $C_2$ &
		$\mathbb Z \oplus (\mathbb Z_2)^2$  & $C_2$ & \phantom{\Big(}Y\phantom{\Big)} & 2/1/1/2 & 4 \\ \hline
	$O^3_3$ & $T^2$ & $333$ & $C_3$ &
 		$\mathbb Z \oplus \mathbb Z_3$  & $C_3$ & \phantom{\Big(}Y\phantom{\Big)} & 5/1/2/2 & 144 \\ \hline
	$O^3_4$ & $T^2$ & $442$ & $C_4$ &
 		$\mathbb Z \oplus \mathbb Z_2$  & $C_4$ & \phantom{\Big(}Y\phantom{\Big)} & 4/1/1/2 &  76 \\ \hline
	$O^3_5$ & $T^2$ & $632$ & $C_6$ &
 		$\mathbb Z$  & $C_6$ & \phantom{\Big(}Y\phantom{\Big)} & 6/1/1/4 & 169 \\ \hline
 	$O^3_6$ & $O^3_6$ & $O^3_6$ & $C_1$ &
 	$(\mathbb Z_4)^2$  & $(C_2)^2$ &  \phantom{\Big(}Y\phantom{\Big)}  & 3/1/1/4 & 19 \\ \hline
\end{tabular}

\vspace{.2in}
\caption{Invariants of the orientable, closed, flat 3-manifolds}
\end{table}

\begin{table} 

\begin{tabular}{c|c|c|c|c|c|c|l|r}
	$M$ & $I(M)$ & $J(M)$ & str. grp. & $H_1(M)$  & hol. grp.  & ODC & BBNWZ & IT \\ \hline
	$N^3_1$ & $S^1$ & $\Iota$ & $C_2$ & 
 		$\mathbb Z^2 \oplus \mathbb Z_2$  & $C_2$ & 
 		\phantom{\Big(}$O^3_1$\phantom{\Big)} & 2/2/1/2 & 7 \\ \hline
	$N^3_2$ & $S^1$ &  $\Iota$ & $(C_2)^2$ & 
	    $\mathbb Z^2$  & $C_2$ & \phantom{\Big(}$O^3_1$\phantom{\Big)} & 2/2/2/2 & 9 \\ \hline
	$N^3_3$ & $K^2$ & $22\ast$ & $C_2$ & 
 			$\mathbb Z \oplus (\mathbb Z_2)^2$  & $(C_2)^2$ & 
 			\phantom{\Big(}$O^3_2$\phantom{\Big)} & 3/2/1/9 & 29 \\ \hline
 	$N^3_4$ & $K^2$ &  $22\times$ & $C_2$ &
 			$\mathbb Z \oplus \mathbb Z_4$  & $(C_2)^2$ & 
 			\phantom{\Big(}$O^3_2$\phantom{\Big)} & 3/2/1/10 & 33 \\ \hline
\end{tabular}

\vspace{.2in}
\caption{Invariants of the nonorientable, closed, flat 3-manifolds}
\end{table}

\section{The Closed Flat 4-Manifolds} 

We denote the 27 orientable, closed, flat, 4-manifold homeomorphism classes 
by $O^4_1,\ldots, O^4_{27}$ and the 47 nonorientable, closed, flat, 4-manifold homeomorphism classes 
by $N^4_1,\ldots, N^4_{47}$ in the order given in Lambert's PhD thesis \cite{Lambert}. 
This ordering follows the order in Hillman's classification \cite{H}.  
In particular, the manifolds are ordered inversely with respect to their first Betti number. 
Thus $O^4_1$ is the 4-torus. 
The characteristic fibers and bases, structure groups, first homology groups,
and holonomy groups of the closed flat 4-manifolds 
are listed in Tables 3, 4, and 5. 
The homology groups were first computed by Levine \cite{Levine} 
and were recomputed by us as a check. 
In the holonomy column, $C_n$ is the multiplicative cyclic group of order $n$, 
$D_n$ is the dihedral group of order $2n$, and $A_4$ is the tetrahedral group of order 12. 

We use Conway's notation \cite{Conway} for the compact, connected, flat 2-orbifolds 
in the $J(M)$ column, and we use the IT number \cite{IT} to designate the compact, connected, flat 3-orbifolds. 

In the Spin column of Table 3, ``Y" stands for ``yes" and ``N" stands for ``no". 
The orientable closed flat 4-manifolds that do not have a spin structure were first determined by Ratcliffe and Tschantz in 2000. 
For the existence of the spin structures in Table 3,  
see Ratcliffe and Tschantz \cite{R-Ts} and Putrycz and Szczepa\'nski \cite{P-S}. 
The heading ODC in Tables 4 and 5 stand for ``orientable double cover". 
The orientable double covers of the nonorientable closed flat 4-manifolds 
were determined by Lambert in his PhD thesis \cite{Lambert}.

\begin{table} 

\begin{tabular}{c|c|c|c|c|c|c|l}
	$M$ & $I(M)$ & $J(M)$ & str. grp. & $H_1(M)$ & hol. grp. & Spin & BBNWZ \\ \hline
	$O^4_1$ & $E^0$ & $E^0$ &  $C_1$ &
 		$\mathbb Z^4$ & $C_1$ & \phantom{\Big(}Y\phantom{\Big)} & 1/1/1/1\\ \hline
	$O^4_2$ & $T^2$ &  2222 & $C_2$ & 
		$\mathbb Z^2 \oplus (\mathbb Z_2)^2$  & $C_2$ & \phantom{\Big(}Y\phantom{\Big)} &  3/1/1/2 \\ \hline
	$O^4_3$ & $T^2$ & 2222 & $(C_2)^2$ &
		$\mathbb Z^2 \oplus \mathbb Z_2$  & $C_2$ & \phantom{\Big(}Y\phantom{\Big)} & 3/1/2/2  \\ \hline
	$O^4_4$ & $T^2$ & 333 & $C_3$ & 
		$\mathbb Z^2 \oplus \mathbb Z_3$  & $C_3$ & \phantom{\Big(}Y\phantom{\Big)} & 8/1/2/2\\ \hline
	$O^4_5$ & $T^2$ & 333 & $(C_3)^2$ & 
		$\mathbb Z^2$ & $C_3$ & \phantom{\Big(}Y\phantom{\Big)} & 8/1/1/2 \\ \hline
	$O^4_6$ & $T^2$ & 442 & $C_4$ & 
		$\mathbb Z^2 \oplus \mathbb Z_2$  & $C_4$ & \phantom{\Big(}Y\phantom{\Big)} & 7/2/1/2 \\ \hline
	$O^4_7$ & $T^2$ & 442 & $C_4\times C_2$ & 
			$\mathbb Z^2$ & $C_4$ & \phantom{\Big(}Y\phantom{\Big)} & 7/2/2/2 \\ \hline
	$O^4_8$ & $T^2$ & 632 & $C_6$ & 
			$\mathbb Z^2$  & $C_6$ & \phantom{\Big(}Y\phantom{\Big)} & 9/1/1/2 \\ \hline
	$O^4_9$ & $O^3_2$ & 17 & $C_2$ & 
 		$\mathbb Z \oplus (\mathbb Z_2)^3$  & $(C_2)^2$ & 
 		\phantom{\Big(}Y\phantom{\Big)} & 5/1/2/7\\ \hline
	$O^4_{10}$ & $O^3_2$ & 18 & $C_2$ &
 		$\mathbb Z \oplus \mathbb Z_2 \oplus \mathbb Z_4$ & $(C_2)^2$ & 
 		\phantom{\Big(}Y\phantom{\Big)} & 5/1/2/8 \\ \hline
 	$O^4_{11}$ & $O^3_2$ & $O^3_6$ & $C_2$ &
 		$\mathbb Z \oplus \mathbb Z_2 \oplus \mathbb Z_4$ &  
		$(C_2)^2$ & \phantom{\Big(}Y\phantom{\Big)} & 5/1/2/10 \\ \hline
	$O^4_{12}$ & $O^3_2$ & 20 & $C_2$ & 
 		$\mathbb Z \oplus (\mathbb Z_2)^2$ & 
			$(C_2)^2$ & \phantom{\Big(}Y\phantom{\Big)} & 5/1/3/6 \\ \hline
	$O^4_{13}$ & $O^3_2$ & $O^3_6$ & $C_4$ &
 		$\mathbb Z \oplus \mathbb Z_4$ & 
			$(C_2)^2$ & \phantom{\Big(}Y\phantom{\Big)} & 5/1/10/4\\ \hline	
	$O^4_{14}$ & $O^3_6$ & $O^3_6$ & $C_1$ &
 		$\mathbb Z \oplus (\mathbb Z_4)^2$ & 
			$(C_2)^2$ & \phantom{\Big(}Y\phantom{\Big)} & 5/1/2/9 \\ \hline
	$O^4_{15}$ & $O^3_6$ & 24 & $C_2$ &
 		$\mathbb Z \oplus (\mathbb Z_2)^2$ &  
			$(C_2)^2$ & \phantom{\Big(}Y\phantom{\Big)} & 5/1/7/4 \\ \hline
	$O^4_{16}$ & $O^3_6$ & 20 & $C_2$ &
 		$\mathbb Z \oplus (\mathbb Z_2)^2$ & 
			$(C_2)^2$ & \phantom{\Big(}N\phantom{\Big)} & 5/1/6/6 \\ \hline
	$O^4_{17}$ & $O^3_6$ & 18 & $C_2$ &	
 		$\mathbb Z \oplus \mathbb Z_2 \oplus \mathbb Z_4$ &  
			$(C_2)^2$ & \phantom{\Big(}N\phantom{\Big)} & 5/1/4/6 \\ \hline	
	$O^4_{18}$ & $O^3_3$ & 152 &$C_2$ &
 		$\mathbb Z \oplus \mathbb Z_6$ & 
			$D_3$ & \phantom{\Big(}Y\phantom{\Big)} & 14/3/5/4 \\ \hline
	$O^4_{19}$ & $O^3_3$ & 151 & $C_2$ & 
 		$\mathbb Z \oplus \mathbb Z_2$ &  
 			 $D_3$ & \phantom{\Big(}Y\phantom{\Big)} & 14/3/6/4 \\ \hline
	$O^4_{20}$ & $O^3_3$ & 152 & $C_6$ &
 		$\mathbb Z \oplus \mathbb Z_2$ & 
 			$D_3$ & \phantom{\Big(}Y\phantom{\Big)} & 14/3/1/4 \\ \hline
	$O^4_{21}$ & $O^3_4$ & 91 & $C_2$ &
 		$\mathbb Z \oplus (\mathbb Z_2)^2$ &  
 			$D_4$ & \phantom{\Big(}Y\phantom{\Big)} & 13/4/1/14 \\ \hline
	$O^4_{22}$ & $O^3_4$ &  92 & $C_2$ & 
 		$\mathbb Z \oplus \mathbb Z_4$ &  
 			$D_4$ & \phantom{\Big(}Y\phantom{\Big)} & 13/4/1/20 \\ \hline
	$O^4_{23}$ & $O^3_6$ &  92 & $C_2$ & 
		$\mathbb Z \oplus \mathbb Z_4$ &  
			$D_4$ & \phantom{\Big(}Y\phantom{\Big)} & 13/4/1/23 \\ \hline
	$O^4_{24}$ & $O^3_6$ &  92 & $C_4$ &
		$\mathbb Z \oplus \mathbb Z_2$ &  
		$D_4$ & \phantom{\Big(}N\phantom{\Big)} & 13/4/4/11 \\ \hline
	$O^4_{25}$ & $O^3_5$ & 178 & $C_2$ &
 		$\mathbb Z \oplus \mathbb Z_2$ &  
 		$D_6$ & \phantom{\Big(}Y\phantom{\Big)} & 15/4/1/10 \\ \hline
	$O^4_{26}$ & $O^3_6$ &  198 & $C_3$ &
 		$\mathbb Z$ &  
 		$A_4$ & \phantom{\Big(}Y\phantom{\Big)} & 24/1/2/4 \\ \hline
	$O^4_{27}$ & $O^3_6$ &  199 & $C_6$ &
 		$\mathbb Z$ &  
 		$A_4$ & \phantom{\Big(}Y\phantom{\Big)} & 24/1/4/4\\ \hline
\end{tabular} 
\vspace{.15in}
\caption{Invariants of the orientable, closed, flat 4-manifolds}
\end{table}

\begin{table}  

\begin{tabular}{c|c|c|c|c|c|c|l}
	$M$ & $I(M)$ & $J(M)$ & str. grp. & $H_1(M)$ & hol. grp. & ODC & BBNWZ \\ \hline
	$N^4_1$ & $S^1$ & $I$  & $C_2$ & 
 		$\mathbb Z^3\oplus\mathbb Z_2$ & $C_2$ & 
 		\phantom{\Big(}$O^4_1$\phantom{\Big)} & 2/1/1/2 \\ \hline
	$N^4_2$ & $S^1$ & $I$ & $(C_2)^2$ &
		$\mathbb Z^3$ & $C_2$ & 
		\phantom{\Big(}$O^4_1$\phantom{\Big)} & 2/1/2/2 \\ \hline
	$N^4_3$ & $T^2$ & $\ast 2222$ & $(C_2)^2$ &
		$\mathbb Z^2 \oplus (\mathbb Z_2)^2$ & $(C_2)^2$ & 
		\phantom{\Big(}$O^4_2$\phantom{\Big)} & 4/1/1/10\\ \hline
	$N^4_4$ & $T^2$ & $22\ast$ & $(C_2)^2$ &
		$\mathbb Z^2 \oplus \mathbb Z_2$  & $(C_2)^2$ & 
		\phantom{\Big(}$O^4_2$\phantom{\Big)} & 4/1/1/11 \\ \hline
	$N^4_5$ & $T^2$ & $22\times$ & $(C_2)^2$ &
		$\mathbb Z^2\oplus \mathbb Z_2$ & $(C_2)^2$ & 
		\phantom{\Big(}$O^4_2$\phantom{\Big)} & 4/1/1/13 \\ \hline
	$N^4_6$ & $T^2$ & $2\!\ast\! 22$ & $(C_2)^2$ &
		$\mathbb Z^2 \oplus \mathbb Z_2$ & $(C_2)^2$ & 
		\phantom{\Big(}$O^4_2$\phantom{\Big)} & 4/1/2/4 \\ \hline
	$N^4_7$ & $T^2$ & $22\times$ & $C_4\times C_2$ & 
		$\mathbb Z^2$ & $(C_2)^2$ & 
		\phantom{\Big(}$O^4_3$\phantom{\Big)} & 4/1/6/4 \\ \hline
	$N^4_8$ & $K^2$ & $22\ast$ & $C_2$ &
		$\mathbb Z^2\oplus (\mathbb Z_2)^2$ & $(C_2)^2$ & 
		\phantom{\Big(}$O^4_2$\phantom{\Big)} &  4/1/1/6 \\ \hline
	$N^4_9$ & $K^2$ & $22\times$ & $C_2$ &
 		$\mathbb Z^2 \oplus \mathbb Z_4$ & $(C_2)^2$ & 
 		\phantom{\Big(}$O^4_2$\phantom{\Big)} & 4/1/1/7 \\ \hline
	$N^4_{10}$ & $K^2$ & $\ast 2222$ & $(C_2)^2$ &
 		$\mathbb Z^2 \oplus (\mathbb Z_2)^2$ & $(C_2)^2$ & 
 		\phantom{\Big(}$O^4_3$\phantom{\Big)}  & 4/1/3/11 \\ \hline
 	$N^4_{11}$ & $K^2$ &  $22\ast$ & $(C_2)^2$ &
 		$\mathbb Z^2 \oplus \mathbb Z_2$ & $(C_2)^2$ & 
 		\phantom{\Big(}$O^4_3$\phantom{\Big)} & 4/1/3/4 \\ \hline
	$N^4_{12}$ & $K^2$ & $22\ast$ & $(C_2)^2$ &
 		$\mathbb Z^2 \oplus \mathbb Z_2$ &  
			$(C_2)^2$ & \phantom{\Big(}$O^4_3$\phantom{\Big)} & 4/1/3/12 \\ \hline	
	$N^4_{13}$ & $K^2$ &  $2\!\ast\! 22$ & $(C_2)^2$ &
 		$\mathbb Z^2\oplus\mathbb Z_2$ &  
			$(C_2)^2$ & \phantom{\Big(}$O^4_3$\phantom{\Big)} & 4/1/4/5 \\ \hline
	$N^4_{14}$ & $O^3_1$ & 2 & $C_2$ & 
 		$\mathbb Z \oplus (\mathbb Z_2)^3$ & 
			$C_2$ & \phantom{\Big(}$O^4_1$\phantom{\Big)} & 2/2/1/2 \\ \hline
	$N^4_{15}$ & $O^3_1$ &  81 & $C_4$ &
 		$\mathbb Z \oplus (\mathbb Z_2)^2$ & 
			$C_4$ & \phantom{\Big(}$O^4_2$\phantom{\Big)} & 12/1/2/2  \\ \hline
	$N^4_{16}$ & $O^3_1$ & 82 & $C_4$ &
 		$\mathbb Z \oplus \mathbb Z_4$ &  
			$C_4$ & \phantom{\Big(}$O^4_3$\phantom{\Big)} & 12/1/4/2 \\ \hline
	$N^4_{17}$ & $O^3_2$ & 81 & $C_4$ &
 		$\mathbb Z \oplus (\mathbb Z_2)^2$ &  
			$C_4$ & \phantom{\Big(}$O^4_2$\phantom{\Big)} & 12/1/3/2 \\ \hline	
	$N^4_{18}$ & $O^3_2$ & 82 & $C_4$ &
 		$\mathbb Z \oplus \mathbb Z_4$ &  
			$C_4$ & \phantom{\Big(}$O^4_3$\phantom{\Big)} & 12/1/6/2 \\ \hline
	$N^4_{19}$ & $O^3_1$ & 174 & $C_6$ &
 		$\mathbb Z \oplus \mathbb Z_6$ &  
 			 $C_6$ & \phantom{\Big(}$O^4_4$\phantom{\Big)} & 14/2/3/2 \\ \hline
	$N^4_{20}$ & $O^3_1$ & 147 & $C_6$ &
 		$\mathbb Z \oplus \mathbb Z_2$ & 
 			$C_6$ & \phantom{\Big(}$O^4_4$\phantom{\Big)} & 14/1/3/2 \\ \hline
	$N^4_{21}$ & $O^3_1$ & 148 & $C_6$ &
 		$\mathbb Z \oplus \mathbb Z_2$ &  
 			$C_6$ & \phantom{\Big(}$O^4_5$\phantom{\Big)} & 14/1/1/2 \\ \hline
	$N^4_{22}$ & $O^3_2$ & 11 &  $C_2$ &
 		$\mathbb Z \oplus (\mathbb Z_2)^3$ &  
 			$(C_2)^2$ & \phantom{\Big(}$O^4_2$\phantom{\Big)} & 4/2/1/8 \\ \hline
	$N^4_{23}$ & $O^3_2$ & 14 & $C_2$ &
		$\mathbb Z \oplus\mathbb Z_2\oplus \mathbb Z_4$ &  
			$(C_2)^2$ & \phantom{\Big(}$O^4_2$\phantom{\Big)} & 4/2/1/16 \\ \hline
\end{tabular} 

\vspace{.15in}
\caption{Invariants of the nonorientable, closed, flat 4-manifolds, I}
\end{table}

\begin{table}  

\begin{tabular}{c|c|c|c|c|c|c|l}
	$M$ & $I(M)$ & $J(M)$ &  str. grp. & $H_1(M)$  & hol. grp. & ODC & BBNWZ \\ \hline
	$N^4_{24}$ & $N^3_1$ &  13 & $C_2$ &
 		$\mathbb Z\oplus(\mathbb Z_2)^3$  & $(C_2)^2$ & 
 		\phantom{\Big(}$O^4_2$\phantom{\Big)} & 4/2/1/11 \\ \hline
	$N^4_{25}$ & $N^3_1$ & 14 & $C_2$ &
		$\mathbb Z\oplus\mathbb Z_2\oplus\mathbb Z_4$ & $(C_2)^2$ & 
		\phantom{\Big(}$O^4_2$\phantom{\Big)} & 4/2/1/12 \\ \hline
	$N^4_{26}$ & $N^3_2$ & 15 & $C_2$ & 
		$\mathbb Z \oplus (\mathbb Z_2)^2$  & $(C_2)^2$ & 
		\phantom{\Big(}$O^4_3$\phantom{\Big)} & 4/2/3/4 \\ \hline
	$N^4_{27}$ & $N^3_1$ &  85 & $C_4$ &
		$\mathbb Z \oplus (\mathbb Z_2)^2$   & $C_2 \times C_4$ & 
		\phantom{\Big(}$O^4_6$\phantom{\Big)} & 13/1/1/8 \\ \hline
	$N^4_{28}$ & $N^3_1$ &  86 & $C_4$ &
		$\mathbb Z \oplus (\mathbb Z_2)^2$ & $C_2 \times C_4$ & 
		\phantom{\Big(}$O^4_6$\phantom{\Big)} & 13/1/1/11 \\ \hline
	$N^4_{29}$ & $N^3_2$ &  88 & $C_4$ &
		$\mathbb Z \oplus \mathbb Z_2$  & $C_2 \times C_4 $ & 
		\phantom{\Big(}$O^4_7$\phantom{\Big)} & 13/1/3/8 \\ \hline
	$N^4_{30}$ & $N^3_3$ &  54 & $C_2$ & 
		$\mathbb Z \oplus (\mathbb Z_2)^3$  & $(C_2)^3$ & 
		\phantom{\Big(}$O^4_9$\phantom{\Big)} & 6/1/1/63 \\ \hline
	$N^4_{31}$ & $N^3_3$ &  57 & $C_2$ &
		$\mathbb Z\oplus (\mathbb Z_2)^3$ & $(C_2)^3$ & 
		\phantom{\Big(}$O^4_{10}$\phantom{\Big)} & 6/1/1/41 \\ \hline
	$N^4_{32}$ & $N^3_3$ &  60 & $C_2$ &
 		$\mathbb Z \oplus (\mathbb Z_2)^2$ &  $(C_2)^3$ & 
 		\phantom{\Big(}$O^4_{10}$\phantom{\Big)} & 6/1/1/64 \\ \hline
	$N^4_{33}$ & $N^3_3$ &  61 & $C_2$ &
 		$\mathbb Z \oplus (\mathbb Z_2)^2$ & $(C_2)^3$ & 
 		\phantom{\Big(}$O^4_{11}$\phantom{\Big)} & 6/1/1/66 \\ \hline
 	$N^4_{34}$ & $N^3_4$ &  52 & $C_2$ &
 		$\mathbb Z \oplus (\mathbb Z_2)^2$ & $(C_2)^3$ & 
 		\phantom{\Big(}$O^4_9$\phantom{\Big)} & 6/1/1/82 \\ \hline
	$N^4_{35}$ & $N^3_4$ &  60 & $C_2$ &
 		$\mathbb Z\oplus(\mathbb Z_2)^2$ & 
			$(C_2)^3$ & \phantom{\Big(}$O^4_{10}$\phantom{\Big)} & 6/1/1/83 \\ \hline
	$N^4_{36}$ & $N^3_4$ &  56 & $C_2$ &
 		$\mathbb Z \oplus \mathbb Z_2\oplus \mathbb Z_4$ & 
			$(C_2)^3$ & \phantom{\Big(}$O^4_{10}$\phantom{\Big)} & 6/1/1/81 \\ \hline	
	$N^4_{37}$ & $N^3_4$ &  62 & $C_2$ &
 		$\mathbb Z \oplus \mathbb Z_2 \oplus \mathbb Z_4$ & 
			$(C_2)^3$ & \phantom{\Big(}$O^4_{11}$\phantom{\Big)} & 6/1/1/45 \\ \hline
	$N^4_{38}$ & $O^3_6$ &  62 & $C_2$ &
 		$\mathbb Z \oplus \mathbb Z_2 \oplus \mathbb Z_4$ &  
			$(C_2)^3$ & \phantom{\Big(}$O^4_{14}$\phantom{\Big)} & 6/1/1/49 \\ \hline
	$N^4_{39}$ & $O^3_6$ &  61 & $C_2$ & 
 		$\mathbb Z \oplus (\mathbb Z_2)^2$ &  
			$(C_2)^3$ & \phantom{\Big(}$O^4_{14}$\phantom{\Big)} & 6/1/1/92 \\ \hline
	$N^4_{40}$ & $O^3_6$ & 122 & $C_4$ &
 		$\mathbb Z \oplus \mathbb Z_2$ & 
			$D_4$ & \phantom{\Big(}$O^4_{15}$\phantom{\Big)} & 12/3/10/5 \\ \hline	
	$N^4_{41}$ & $O^3_6$ &  114 & $C_4$ &
 		$\mathbb Z \oplus \mathbb Z_4$ &  
			$D_4$ & \phantom{\Big(}$O^4_{17}$\phantom{\Big)} & 12/3/4/6 \\ \hline
	$N^4_{42}$ & $O^3_2$ &  176 & $C_6$ &
 		$\mathbb Z \oplus \mathbb Z_2$ &  
 			 $C_2 \times C_6$ & \phantom{\Big(}$O^4_8$\phantom{\Big)} & 15/1/1/10 \\ \hline
	$N^4_{43}$ & $O^3_6$ &  205 & $C_6$ &
 		$\mathbb Z$ & $C_2 \times A_4$ & \phantom{\Big(}$O^4_{26}$\phantom{\Big)} & 25/1/1/10 \\ \hline
	$N^4_{44}$ & $N^4_{44}$ &  $N^4_{44}$ & $C_1$ &
 		$\mathbb Z_2 \oplus (\mathbb Z_4)^2$ &  
 			$(C_2)^2$ & \phantom{\Big(}$O^4_2$\phantom{\Big)} & 4/3/1/6 \\ \hline
	$N^4_{45}$ & $N^4_{45}$ &  $N^4_{45}$ & $C_1$ &
 		$(\mathbb Z_2)^2 \oplus \mathbb Z_4$ &  
 			$(C_2)^3$ & \phantom{\Big(}$O^4_9$\phantom{\Big)} & 6/2/1/50 \\ \hline
	$N^4_{46}$ & $N^4_{46}$ &  $N^4_{46}$ & $C_1$ &
		$(\mathbb Z_2)^2 \oplus \mathbb Z_4$ & 
			$(C_2)^3$ & \phantom{\Big(}$O^4_{10}$\phantom{\Big)} & 6/2/1/27 \\ \hline
	$N^4_{47}$ & $N^4_{47}$ &  $N^4_{47}$ & $C_1$ &
		$(\mathbb Z_4)^2$ & $D_4$ & \phantom{\Big(}$O^4_{12}$\phantom{\Big)} & 12/4/3/11 \\ \hline
\end{tabular} 

\vspace{.15in}
\caption{Invariants of the nonorientable, closed, flat 4-manifolds, II}
\end{table}

\section{The Classification Theory for co-Seifert Fibrations} 

A geometric fibration of a closed flat $n$-manifold over a 1-orbifold is called a {\it geometric co-Seifert fibration}. 
There are two cases to consider:  the base 1-orbifold is either a circle or a closed interval. 
To classify geometric co-Seifert fibrations of closed flat $n$-manifolds up to affine equivalence, 
it suffices by Theorem 4 to classify up to isomorphism pairs $(\Gamma,\Nu)$ consisting of a torsion-free $n$-space group and  
a normal subgroup $\Nu$ such that $\Gamma/\Nu$ is either infinite cyclic or infinite dihedral. 

Let $\Mu$ be a torsion-free $(n-1)$-space group, and let $\Delta$ be a 1-space group. 
Define $\mathrm{Iso}(\Delta,\Mu)$ to be the set of isomorphism classes of pairs $(\Gamma, \Nu)$ 
where $\Nu$ is a complete normal subgroup of an $n$-space group $\Gamma$ 
such that $\Nu$ is isomorphic to $\Mu$ and $\Gamma/\Nu$ is isomorphic to $\Delta$. 
We denote the isomorphism class of a pair $(\Gamma,\Nu)$ by $[\Gamma,\Nu]$. 
Define $\mathrm{Iso}_{f}(\Delta,\Mu)$ to be the subset of $\mathrm{Iso}(\Delta,\Mu)$ consisting 
of all the classes $[\Gamma,\Nu]$ such that $\Gamma$ is torsion-free. 
If $\Delta$ is infinite cyclic, then $\mathrm{Iso}_{f}(\Delta,\Mu) = \mathrm{Iso}(\Delta,\Mu)$, since $\Delta$ is torsion-free. 
If $\gamma\in\Gamma$, let $\gamma_\ast$ be the automorphism of $\Nu$ defined by $\gamma_\ast(\nu) = \gamma\nu\gamma^{-1}$. 
The next theorem is well known; for a proof, see Theorem 23 in \cite{R-TII}. 

\begin{theorem} 
Let $\Mu$ be a torsion-free $(n-1)$-space group, and let $\Delta$ be an infinite cyclic 1-space group.  
The set $\mathrm{Iso}_f(\Delta,\Mu)$ is in one-to-one correspondence with the 
set of conjugacy classes of pairs of inverse elements of $\mathrm{Out}(\Mu)$ of finite order.  
If $[\Gamma,\Nu] \in \mathrm{Iso}_f(\Delta,\Mu)$ and 
$\alpha: \Nu \to \Mu$ is an isomorphism and $\gamma$ is an element of $\Gamma$ such that $\Nu\gamma$ generates 
$\Gamma/\Nu$, then $[\Gamma, \Nu]$ corresponds to the conjugacy class 
of the pair of inverse elements $\{\alpha\gamma_\ast^{\pm 1}\alpha^{-1}\mathrm{Inn}(\Mu)\}$ of $\mathrm{Out}(\Mu)$.   
\end{theorem}

In order to treat the case that $\Gamma/\Nu$ is infinite dihedral, we need some more theory. 
Let $\Nu$ be a complete normal subgroup of an $n$-space group $\Gamma$, 
let $V = \mathrm{Span}(\Nu)$, and let $V^\perp$ be the orthogonal complement of $V$ in $E^n$.  
Euclidean $n$-space $E^n$ decomposes as the Cartesian product 
$E^n = V \times V^\perp$. 
The action of $\Nu$ on $V^\perp$ is trivial. 
Hence $E^n/\Nu$ decomposes as the Cartesian product 
$E^n/\Nu = V/\Nu \times V^\perp.$

The action of $\Gamma/\Nu$ on $E^n/\Nu$ corresponds 
to the diagonal action of $\Gamma/\Nu$ on $V/\Nu \times V^\perp$ via isometries 
defined by the formula
$$(\Nu(b+B))(\Nu v,w) = (\Nu(c+Bv),d+Bw)$$
where $b = c + d$ with $c\in V$ and $d\in V^\perp$. 

\begin{theorem} {\rm (Theorem 16 \cite{R-TII})} 
Let $\Nu$ be a complete normal subgroup of an $n$-space group $\Gamma$, 
and let $V= \mathrm{Span}(\Nu)$. 
Then the map
$$\chi: E^n/\Gamma \to (V/\Nu\times V^\perp)/(\Gamma/\Nu)$$
defined by $\chi(\Gamma x) = (\Gamma/\Nu)(\Nu v, w)$, 
with  $x = v + w$ and $v\in V$ and $w\in V^\perp$, is an isometry. 
\end{theorem}

\begin{lemma} 
Let ${\Nu}$ be a complete normal subgroup of an $n$-space group $\Gamma$,  
and let $V= \mathrm{Span}(\Nu)$. 
Then $\Gamma$ is torsion-free if and only if $\Nu$ is torsion-free 
and $\Gamma/\Nu$ acts freely on $V/\Nu\times V^\perp$. 
\end{lemma}
\begin{proof}  If $\Gamma$ is torsion-free, then $\Nu$ is torsion-free, and so $E^n/\Gamma$ and $V/\Nu\times V^\perp$ are flat $n$-manifolds,   
Hence $\Gamma/\Nu$ acts freely on $V/\Nu\times V^\perp$ by Theorem 8.  
Conversely, if $\Nu$ is torsion-free and $\Gamma/\Nu$ acts freely on $V/\Nu\times V^\perp$, 
Then $V/\Nu\times V^\perp$ is a flat $n$-manifold, and so $E^n/\Gamma$ is flat $n$-manifold by Theorem 8. 
Hence $\Gamma$ is torsion-free. 
\end{proof}

\begin{theorem} 
Let $\Nu$ be a torsion-free normal subgroup of an $n$-space group $\Gamma$ such that $\Gamma/\Nu$ is infinite dihedral, 
let $V= \mathrm{Span}(\Nu)$, and let  $\gamma_1, \gamma_2$ be elements of $\Gamma$ such that $\Nu\gamma_1$ and $\Nu\gamma_2$ 
are order $2$ generators of $\Gamma/\Nu$. 
Then $\Gamma$ is torsion-free if and only if $\Nu\gamma_1$ and $\Nu\gamma_2$ do not fix a point of $V/\Nu$.  
\end{theorem}
\begin{proof}
The group $\Nu$ is the kernel of the action of $\Gamma$ on $V^\perp$, and so we have 
an effective action of $\Gamma/\Nu$ on $V^\perp$. 
Therefore $\Nu\gamma_1$ and $\Nu\gamma_2$ act as reflections of the line $V^\perp$. 
If $\Gamma$ is torsion-free, then $\Gamma/\Nu$ acts freely on $V/\Nu\times V^\perp$ by Lemma 2, 
and so $\Nu\gamma_1$ and $\Nu\gamma_2$ do not fix a point of $V/\Nu$. 

Conversely, suppose $\Nu\gamma_1$ and $\Nu\gamma_2$ do not fix a point of $V/\Nu$. 
Let $\gamma$ be an element of $\Gamma/\Nu$ such that $\Nu\gamma$ is not the identity.  
If $\Nu\gamma$ has infinite order, then $\Nu\gamma$ acts as a nontrivial translation of the line $V^\perp$, 
and so $\Nu\gamma$ does not fix a point of $V/\Nu\times V^\perp$. 
If $\Nu\gamma$ has finite order, then $\Nu\gamma$ is conjugate to either $\Nu\gamma_1$ or $\Nu\gamma_2$, 
and so $\Nu\gamma$ does not fix a point of $V/\Nu$ or $V/\Nu\times V^\perp$. 
Thus $\Gamma/\Nu$ acts freely on $V/\Nu\times V^\perp$. 
Hence $\Gamma$ is torsion-free by Lemma 2. 
\end{proof}

Let $\Mu$ be a torsion-free $(n-1)$-space group, and  
let $\alpha = a + A$, with $a \in E^{n-1}$ and $A \in \mathrm{GL}(n-1, \realnos)$, be an affinity of $E^{n-1}$ such that $\alpha\Mu\alpha^{-1} = \Mu$. 
Then $\alpha$ induces an affinity $\alpha_\star$ of the flat $(n-1)$-manifold $E^{n-1}/\Mu$ defined by $\alpha_\star(\Mu x) = \Mu \alpha x$. 
Let $\mathrm{Aff}(\Mu)$ be the group of affinities of $E^{n-1}/\Mu$, and 
let $N_A(\Mu)$ be the normalizer of $\Mu$ in the group of affinities of $E^{n-1}$. 
Then the map $\Phi: N_A(\Mu) \to \mathrm{Aff}(\Mu)$ defined by $\Phi(\alpha) = \alpha_\star$ is an epimorphism with kernel $\Mu$ 
by Lemma 1 of \cite{C-V}. 

Define $\Omega: \mathrm{Aff}(\Mu) \to \mathrm{Out}(\Mu)$ by $\Omega(\alpha_\star) = \alpha_\ast\mathrm{Inn}(\Mu)$ 
where $\alpha_\ast$ is the automorphism of $\Mu$ defined by $\alpha_\ast(\mu) = \alpha\mu\alpha^{-1}$. 
Let $C_A(\Mu)$ be the centralizer of $\Mu$ in the group of affinities of $E^{n-1}$. 
Then $C_A(\Mu) = \{c+I: c \in \mathrm{Span}(Z(\Mu))\}$ and 
$\Omega: \mathrm{Aff}(\Mu) \to \mathrm{Out}(\Mu)$ is an epimorphism with kernel $\Phi(C_A(\Mu))$ by Lemmas 1 and 2 of \cite{C-V}. 

Let $\Nu$ be a complete normal subgroup of an $n$-space group $\Gamma$, and let $V = \mathrm{Span}(\Nu)$. 
Let $\gamma \in \Gamma$.   Then $\gamma = b+B$ with $b\in E^n$ and $B\in \mathrm{O}(n)$. 
Write $b = \overline b + b'$ with $\overline b \in V$ and $b' \in V^\perp$. 
Let $\overline B$ and $B'$ be the orthogonal transformations of $V$ and $V^\perp$, respectively, 
obtained by restricting $B$. 
Let $\overline \gamma = \overline b + \overline B$ and $\gamma' = b' + B'$. 
Then $\overline \gamma$ and $\gamma'$ are isometries of $V$ and $V^\perp$, respectively. 
The diagonal action of $\Gamma$ on $V\times V^\perp$ is given by $\gamma(v,w) = (\overline \gamma v, \gamma' w)$. 
The diagonal action of $\Gamma/\Nu$ on $V/\Nu \times V^\perp$ is given by 
$\Nu\gamma(\Nu v, w) = (\Nu \overline \gamma v, \gamma' w)$. 

Let $\ov \Gamma = \{\ov \gamma: \gamma \in \Gamma\}$. 
Then $\ov \Gamma$ is a subgroup of $\mathrm{Isom}(V)$. 
The map $\Beta: \Gamma \to \ov \Gamma$ defined by $\Beta(\gamma)= \ov\gamma$ is an epimorphism 
with kernel equal to the kernel  $\Kappa$ of the action of $\Gamma$ on $V$. 

Let $\ov \Nu = \{\ov \nu: \nu \in \Nu\}$.  
Then $\ov \Nu$ is a subgroup of $\mathrm{Isom}(V)$. 
The map $\Beta: \Nu \to \ov \Nu$ defined by $\Beta(\nu) = \ov\nu$ is an isomorphism. 
The group $\ov\Nu$ is a space group with $V/\ov{\Nu} = V/\Nu$ by Lemma 1.

\begin{theorem} 
Let $\Mu$ be a torsion-free $(n-1)$-space group, and let $\Delta$ be an infinite dihedral group. 
The set $\mathrm{Iso}_f(\Delta,\Mu)$ is in one-to-one correspondence with the set of equivalence classes 
of pairs of order $2$ elements $\alpha_\star$ and $\beta_\star$ of the group $\mathrm{Aff}(\Mu)$, 
where $\alpha = a + A$ and $\beta = b+B$,  
such that $\alpha_\star$ and $\beta_\star$ do not fix a point of $E^{n-1}/\Mu$ and $\Omega(\alpha_\star\beta_\star)$ has finite order in $\mathrm{Out}(\Mu)$. 
The equivalence class of the pair $\{\alpha_\star, \beta_\star\}$ is the union of the conjugacy classes of the pairs $\{\alpha_\star, (u+I)_\star\beta_\star\}$ 
for each $u$ in $E^{n-1}$ such that $u\in \mathrm{Span}(Z(\Mu))$ and $Au = Bu = -u$. 

Let $[\Gamma, \Nu] \in \mathrm{Iso}_f(\Delta,\Mu)$, let $V = \mathrm{Span}(\Nu)$, 
and let $\phi:V\to E^{n-1}$ be an affinity such that $\phi\overline{\Nu}\phi^{-1} = \Mu$. 
Let $\gamma_1$ and $\gamma_2$ be elements of $\Gamma$ such that $\Nu\gamma_1$ and $\Nu\gamma_2$ are order two generators of $\Gamma/\Nu$. 
Then $[\Gamma, \Nu]$ corresponds to the equivalence class of the pair $\{(\phi\overline{\gamma}_1\phi^{-1})_\star, (\phi\overline{\gamma}_2\phi^{-1})_\star\}$. 
\end{theorem}
\begin{proof} Note that $(u+I)\beta = (u/2+I)\beta(u/2+I)^{-1}$, since $Bu = -u$. 
Hence  $(u+I)_\star\beta_\star$ has order $2$ and does not fix a point.  
Moreover $\Omega(\alpha_\star (u+I)_\star\beta_\star) = \Omega(\alpha_\star\beta_\star)$, 
since $(u+I)_\star$ is in the kernel of $\Omega$. 
Therefore the statement of the theorem is consistent. 
The proof now follows from Theorem 9 and Theorem 26 of \cite{R-TII}. 
\end{proof}

\begin{lemma}  
Let $\Mu$ be a torsion-free $n$-space group whose group of translations is $\Tau$. 
Let $\alpha_0 = a_0+A_0, \ldots, \alpha_k = a_k+A_k$ be elements of $\Mu$ 
such that $\alpha_0 = I$, and $\{A_0, \ldots, A_k\}$ is the holonomy group of $\Mu$. 
Let $\beta$ be an affinity of $E^n$ that normalizes $\Mu$. 
Then $\beta$ normalizes $\Tau$, 
and so $\beta$ induces an affinity $\tilde\beta_\star$ of the $n$-torus $E^n/\Tau$. 
Moreover $\beta_\star$ does not fix a point of $E^n/\Mu$ if and only if $(\widetilde{\alpha_i\beta})_\star$ does not 
fix a point of $E^n/\Tau$ for each $i = 0, 1, \ldots, k$. 
\end{lemma}
\begin{proof}
The affinity $\beta$ normalizes $\Tau$, since $\Tau$ is a characteristic subgroup of $\Mu$. 
Suppose $\beta_\star$ fixes a point $\Mu x$ of $E^n/\Mu$.  
Then $\Mu\beta x = \Mu x$.  Hence, there is a $\mu \in \Mu$ such that $\mu\beta x = x$. 
Now $\{\alpha_0,\ldots, \alpha_k\}$ is a set of coset representatives for $\Tau$ in $\Mu$. 
Hence there exists $\tau \in \Tau$ and an index $i$ such that $\mu = \tau\alpha_i$. 
Then we have $\tau\alpha_i\beta x = x$,  and so $\Tau\alpha_i\beta x = \Tau x$. 
Thus $(\widetilde{\alpha_i\beta})_\star$ fixes the point $\Tau x$ of $E^n/\Tau$. 

Conversely, if $(\widetilde{\alpha_i\beta})_\star$ fixes the point $\Tau x$ of $E^n/\Tau$, 
then $\Tau\alpha_i\beta x = \Tau x$, and so $\Mu\beta x = \Mu x$. 
Therefore $\beta_\star$ fixes the point $\Mu x$ of $E^n/\Mu$. 
Hence $\beta_\star$ fixes a point of $E^n/\Mu$ if and only if there is an index $i$ 
such that $(\widetilde{\alpha_i\beta})_\star$ fixes a point of $E^n/\Tau$. 
\end{proof}

\section{Table Formatting}  

In the rest of the paper, we will apply the classification theory of \S 7  
to classify all the geometric co-Seifert fibrations of closed flat 3- or 4-manifolds up to affine equivalence. 
We will describe our classification by a series of Tables.  
In this section, we describe the formatting of these tables and the information that can be derived from these tables.   

We denote a geometric fibration, with generic fiber $F$, over a base $B$ by $F\, |\, B$. 
The are two types for $B$.  Either $B$ is a circle, denoted by $S^1$, or a closed interval, denoted by $\Iota$. 
There are two types of tables corresponding to the two types of $B$. 

Consider the table for the case of flat $F\, |\, S^1$ fiberings. See Table 6 for example. 
The headings of the table are number (no.), manifold (mfd.), co-Seifert fibration (cS-fbr.), structure group (grp.), and classifying pair representative (representative). 
The rows of the table correspond to the affine equivalence classes of flat $F\, |\, S^1$ fiberings. 
Each row of the table describes a flat $F\, |\, S^1$ fibering that represents the corresponding equivalence class. 
The flat $F\, |\, S^1$ fibering is described by a generalized Calabi construction. 

We represent the flat $(n-1)$-manifold $F$ as a space-form $E^{n-1}/\Mu$. 
The torsion-free $(n-1)$-space group $\Mu$ is specified by a standard set of generators $t_1,\ldots, t_{n-1}$, $\alpha_1, \ldots, \alpha_k$ 
where $\alpha_i = a_i + A_i$ with $a_i \in E^{n-1}$ and $A_i\in \mathrm{GL}(n-1,\integers)$ for each $i$.  Here $k$ may be $0$. 
We represent $S^1$ as the space-form $E^1/\langle t_1\rangle$. 

The flat $n$-manifold in Column 2 is obtained as the quotient space of $F \times S^1$ 
by the diagonal action of the representative $\beta = b+B$ in Column 5. 
The representative $\beta$ is an affinity of $E^{n-1}$ that normalizes $\Mu$. 
If $b \neq 0$, we denote $b+B$ by an augmented matrix with the vector $b$ as the first column. 
The representative $\beta$ acts on $F$ by the affinity $\beta_\star$ of order $m$, and acts on $S^1$ by the rotation $(e_1/m+I)_\star$ 
of $2\pi/m$ where $m$ is the order of the cyclic structure group $C_m$ in Column 4. 
If $B$ is not orthogonal, the quotient manifold is an affine $n$-manifold that is affinely equivalent to the corresponding flat $n$-manifold in Column 2. 
The Calabi construction is equivalent to a mapping torus construction with monodromy map $\beta_\star$. 

Note that if $m = 1$, then the flat $n$-manifold in Column 2 is affinely equivalent to the Cartesian product $F \times S^1$. 

The Calabi construction determines a pair $(\Gamma, \Nu)$ consisting of a torsion-free $n$-space group $\Gamma$ 
and a complete normal subgroup $\Nu$ such that $\Gamma/\Nu$ is infinite cyclic, $\ov \Nu = \Mu$, 
and $[\Gamma, \Nu]$ corresponds to the conjugacy class of the pair 
of inverse elements $\{\Omega(\beta_\star),\Omega(\beta_\star)^{-1}\}$ of $\mathrm{Out}(\Mu)$ by Theorem 18 in \cite{R-TII} and Theorem 7. 
The geometric fibration of $E^n/\Gamma$ determined by $\Nu$ is equivalent to the $F\,|\,S^1$ fibering determined by the generalized Calabi construction 
by Theorem 6. 

The torsion-free, affine, $n$-space group $\Gamma$ is easy to describe. 
Extend the affinity $\alpha_i = a_i + A_i$ of $E^{n-1}$ to an affinity $\hat\alpha_i = a_i + \hat A_i$ of $E^n$ 
by considering $a_i$ as a vector in $E^n$ whose $n$-coordinate is $0$, 
and by extending $A_i$ to $\hat A_i \in \mathrm{GL}(n,\integers)$ so that $\hat A_i(e_n) = e_n$ for each $i$. . 
We have 
$$\Gamma = \langle t_1,\ldots, t_n,\hat\alpha_1,\ldots, \hat\alpha_k, \hat\beta\rangle$$
where $\hat\beta = b + e_n/m + \hat B$ and $\hat B$ extends $B$ so that $\hat Be_n = e_n$. 
Moreover 
$$\Nu = \langle t_1, \ldots, t_{n-1}, \hat\alpha_1, \ldots, \hat\alpha_k\rangle$$
and $\Nu\hat\beta$ generates the infinite cyclic group $\Gamma/\Nu$. 

The flat $n$-manifold in Column 2 was identified in the case $n = 3$ 
by comparing $\Gamma$ with the corresponding torsion-free 3-space groups in Table 1B of \cite{B-Z},   
and in the case $n = 4$ by applying CARAT's recognition routine to $\Gamma$. 
After this identification has been made, the flat $n$-manifold in Column 2 can be identified by the co-Seifert fibration data in Columns 3 and 5 via Theorem 7. 

Next consider the table for the case of flat $F\, |\, \Iota$ fiberings.  See Table 7 for example. 
The headings of the table are number (no.), manifold (mfd.), co-Seifert fibration (cS-fbr.), structure group (grp.), classifying pair representatives (pair representatives), 
and singular fibers (s-fbrs.). 
The rows of the table correspond to the affine equivalence classes of flat $F\, |\, \Iota$ fiberings. 
Each row of the table describes a flat $F\, |\, \Iota$ fibering that represents the corresponding equivalence class. 
The flat $F\, |\, \Iota$ fibering is described by a generalized Calabi construction. 

We represent the flat $(n-1)$-manifold $F$ as a space-form $E^{n-1}/\Mu$. 
The torsion-free $(n-1)$-space group $\Mu$ is specified by a standard set of generators $t_1,\ldots, t_{n-1}$, $\alpha_1, \ldots, \alpha_k$ 
where $\alpha_i = a_i + A_i$ with $a_i \in E^{n-1}$ and $A_i\in \mathrm{GL}(n-1,\integers)$ for each $i$.  Here $k$ may be $0$. 
We represent $S^1$ as the space-form $E^1/\langle t_1\rangle$. 

The flat $n$-manifold in Column 2 is obtained as the quotient space of $F \times S^1$ 
by the diagonal action of the representatives $\beta = b+B$ and $\gamma = c +C$ in Column 5. 
The representatives $\beta$ and $\gamma$ are affinities of $E^{n-1}$ that normalize $\Mu$. 
If $b \neq 0$, we denote $b+B$ by an augmented matrix with the vector $b$ as the first column and likewise for $c+C$. 
The representative $\beta$ acts on $F$ by the order 2 affinity $\beta_\star$ without fixing a point of $F$, and acts on $S^1$ by the reflection $(-I)_\star$. 
The representative $\gamma$ acts on $F$ by the order 2 affinity $\gamma_\star$ without fixing a point of $F$, and acts on $S^1$ by the reflection $(e_1/m-I)_\star$ 
where $2m$ is the order of the dihedral structure group $D_m$ in Column 4. 
Moreover $\gamma_\star\beta_\star$ has order $m$ 
and $\gamma\beta$ acts on $S^1$ by the rotation $(e_1/m+I)_\star$ of $2\pi/m$. 
If $B$ or $C$ is not orthogonal, the quotient manifold is an affine $n$-manifold that is affinely equivalent to the corresponding flat $n$-manifold. 

The Calabi construction determines a pair $(\Gamma, \Nu)$ consisting of a torsion-free $n$-space group $\Gamma$ 
and a complete normal subgroup $\Nu$ such that $\Gamma/\Nu$ is infinite dihedral, $\ov \Nu = \Mu$, 
and $[\Gamma, \Nu]$ corresponds to the equivalence class of the pair 
of order 2 elements $\{\beta_\star, \gamma_\star\}$ of $\mathrm{Aff}(\Mu)$ by Theorem 18 in \cite{R-TII} and Theorem 10. 
The geometric fibration of $E^n/\Gamma$ determined by $\Nu$ is equivalent to the $F\,|\,\Iota$ fibering determined by the generalized Calabi construction 
by Theorem 6. 

The torsion-free, affine, $n$-space group $\Gamma$ is easy to describe. 
Extend the affinity $\alpha_i = a_i + A_i$ of $E^{n-1}$ to an affinity $\hat\alpha_i = a_i + \hat A_i$ of $E^n$ by considering $a_i$ as a vector in $E^n$ whose $n$-coordinate is $0$, 
and by extending $A_i$ to $\hat A_i \in \mathrm{GL}(n,\integers)$ so that $\hat A_i(e_n) = e_n$ for each $i$. . 
We have 
$$\Gamma = \langle t_1,\ldots, t_n,\hat\alpha_1,\ldots, \hat\alpha_k, \hat\beta, \hat\gamma \rangle$$
where $\hat\beta = b+ \hat B$ and $\hat B$ extends $B$ so that $\hat Be_n = -e_n$ 
and $\hat\gamma = c + e_n/m+ \hat C$ and $\hat C$ extends $C$ so that $\hat Ce_n = -e_n$. 
Moreover 
$$\Nu = \langle t_1, \ldots, t_{n-1}, \hat\alpha_1, \ldots, \hat\alpha_k\rangle$$
and $\Nu\hat\beta$ and $\Nu\hat\gamma$ are order 2 generators of the infinite dihedral group $\Gamma/\Nu$.  

Reversing the order of $\beta$ and $\gamma$ does not change the isomorphism class $[\Gamma,\Nu]$, 
since conjugating $\Gamma$ by the reflection $e_n/2m+\mathrm{diag}(1,\ldots, 1, -1)$ conjugates $\Gamma$, while keeping $\Nu$ fixed, 
to the group defined by reversing the order of $\beta$ and $\gamma$.  

The group $\Gamma$ is the free product of the subgroups $\langle \Nu, \hat\beta\rangle$ and $\langle \Nu, \hat\gamma\rangle$ 
amalgamated over their index 2 subgroup $\Nu$. 
The space group types of the factors $\langle \Nu,\hat\beta\rangle$ and $\langle \Nu, \hat\gamma\rangle$ 
are the types of the singular fibers in Column 6. 
Geometrically this corresponds to the decomposition of $E^n/\Gamma$ as the union of two twisted $\Iota$-bundles, with cores the singular fibers, 
over their common boundary, which is a generic fiber.  

The flat $n$-manifold in Column 2 was identified in the case $n = 3$ 
by comparing $\Gamma$ with the corresponding torsion-free 3-space groups in Table 1B of \cite{B-Z},   
and in the case $n = 4$ by applying CARAT's recognition routine to $\Gamma$. 
After this identification has been made, the flat $n$-manifold in Column 2 can be identified by the co-Seifert fibration data in Columns 3 and 5 via Theorem 10.

\section{The Closed Flat 3-Manifold co-Seifert Fibrations} 

In this section, we describe the affine classification of the geometric co-Seifert fibrations of closed flat 3-manifolds. 
There are two possible generic fibers, the 2-torus $T^2$ or the Klein bottle $K^2$.  
We begin with the $T^2$ case.  Here $T^2 = E^2/\Mu$ with $\Mu = \langle t_1,t_2\rangle$. 
The following lemma is well known. 

\begin{lemma} 
Let $\Mu = \langle t_1, t_2\rangle$. We have that $\mathrm{Out}(\Mu) = \mathrm{Aut}(\Mu)= \mathrm{GL}(2,\integers)$ and 
$$N_A(\Mu) = \{a+A: a\in E^2 \ and\ A \in \mathrm{GL}(2,\integers)\}.$$
The map $\eta: \mathrm{Aff}(\Mu) \to \mathrm{GL}(2,\integers)$, defined by 
$\eta((a+A)_\star) = A$, is an epimorphism with kernel $\Kappa = \{(a+I)_\star: a \in E^2\}$. 
The map $\sigma:  \mathrm{GL}(2,\integers) \to \mathrm{Aff}(\Mu)$, defined by $\sigma(A) = A_\star$, 
is a monomorphism,  and $\sigma$ is a right inverse of $\eta$. 
The group $\Kappa$ is the maximal torus of the Lie group $\mathrm{Aff}(\Mu)$ of rank $2$. 
\end{lemma}

The conjugacy classes of finite subgroups of $\mathrm{GL}(2,\integers)$ are called the $\integers$-classes of finite subgroups of $\mathrm{GL}(2,\integers)$ in \cite{B-Z}. 
From Tables 1A and 6A in \cite{B-Z}, we see that there are 7 conjugacy classes of finite cyclic subgroups of $\mathrm{GL}(2,\integers)$. 
An element of $\mathrm{GL}(2,\integers)$ of finite order has order at most 6. 
Hence there are 7 conjugacy classes of inverse pairs of elements of $\mathrm{GL}(2,\integers)$ of finite order.  
Representatives of these classes, taken from Table 1A in \cite{B-Z}, are listed in Column 5 of Table 6. 
By Theorem 7, there are 7 affine equivalence classes of flat $T^2\,|\, S^1$ fiberings. 
These fiberings are represented by the generalized Calabi constructions described in Table 6 
as explained in \S 8. 

\begin{table} 

\begin{tabular}{c|c|c|c|c}	
no. & mfd. & cS-fbr.  &  grp. & representative      \\ \hline	
 1	&$O^3_1$ & $T^2\,|\,S^1$ &  $C_1$ & 
	  	$ \left(\begin{array}{rr} 
		1 &  0 \\ 0 & 1 \end{array}\right)$   \\ \hline
 2	&$O^3_2$ & $T^2\,|\,S^1$ &  $C_2$ &
		$ \left(\begin{array}{rr} 
		-1 &  0 \\ 0 & -1 \end{array}\right)$    \\ \hline
 3	&$O^3_3$ & $T^2\,|\,S^1$ &  $C_3$ &
 		$ \left(\begin{array}{rr} 
		0 &  -1 \\ 1 & -1 \end{array}\right)$   \\ \hline
 4	&$O^3_4$ & $T^2\,|\,S^1$ &  $C_4$ &
 		$ \left(\begin{array}{rr} 
		0 &  -1\\ 1 & 0 \end{array}\right)$    \\ \hline
 5	&$O^3_5$ & $T^2\,|\,S^1$ &  $C_6$ &
 		$ \left(\begin{array}{rr} 
		1 &  -1 \\ 1 & 0 \end{array}\right)$     \\ \hline
 6 	&$N^3_1$ & $T^2\,|\,S^1$ &  $C_2$ &
 	$ \left(\begin{array}{rr} 
		1 &  0 \\ 0 & -1 \end{array}\right)$   \\ \hline
 7 	&$N^3_2$ & $T^2\,|\,S^1$ &  $C_2$ &
 	$ \left(\begin{array}{rr} 
		0 &  1 \\ 1 & 0 \end{array}\right)$    \\ \hline
\end{tabular}

\vspace{.2in}
\caption{The flat $T^2$ fiberings over $S^1$}
\end{table}

The representative, for cases 3, 4, 5, was chosen so that the resulting flat 3-manifold has the corresponding IT number in Table 1. 
Replacing a representative by its inverse changes the IT number in these 3 cases. 
In other words, the space group types for cases 3, 4, 5 subdivide into two enantiomorphic proper space group types which have their own IT numbers.

Next we consider $T^2\,|\, \Iota$ fiberings. 

\begin{lemma} 
Let $\Mu = \langle t_1, t_2\rangle$, and let $\alpha = a + A$ be an affinity of $E^2$ that normalizes $\Mu$ 
such that $\alpha_\star$ has order $2$ and $\alpha_\star$ does not fix a point of the $2$-torus $E^2/\Mu$. 
Then $A$ has order $1$ or $2$.  
If $A$ has order $1$, then $\alpha_\star$ is conjugate in $\mathrm{Aff}(\Mu)$ to $(e_1/2+I)_\star$. 
The quotient of $E^2/\Mu$ by the action of $(e_1/2+I)_\star$ is a $2$-torus. 
If $A$ has order $2$, then $\alpha_\star$ is conjugate in $\mathrm{Aff}(\Mu)$ to $(e_1/2 + \mathrm{diag}(1,-1))_\star$. 
The quotient of $E^2/\Mu$ by the action of $(e_1/2 + \mathrm{diag}(1,-1))_\star$ is a Klein bottle. 
\end{lemma}
\begin{proof}
The matrix $A$ has order $1$ or $2$ by Lemma 4. 
If $A$ has order $1$, then $2a \in\integers^2$, 
and so $\alpha_\star$ is conjugate in $\mathrm{Aff}(\Mu)$ to $(e_1/2+I)_\star$.  
The quotient of $E^2/\Mu$ by the action of $(e_1/2+I)_\star$ is a $2$-torus, 
since $\langle e_1/2+I, t_2\rangle$ is a $2$-torus group. 

Now suppose $A$ has order 2.  
According to Table 1A in \cite{B-Z}, there are 3 possible $\integers$-classes for $A$. 
The group $\langle t_1, t_2, \alpha\rangle$ is a torsion-free 2-space group, 
since $\alpha_\star$ fixes no point of $E^2/\Mu$. 
According to Table 1A in \cite{B-Z}, only the $\integers$-class of 2-space groups corresponding to the $\integers$-class of $\mathrm{diag}(1,-1)$ contains 
a torsion-free group, namely $\langle t_1, t_2, e_1/2 + \mathrm{diag}(1,-1)\rangle$. 
Therefore $\alpha_\star$ is conjugate in $\mathrm{Aff}(\Mu)$ to $(e_1/2 + \mathrm{diag}(1,-1))_\star$, 
since $\mathrm{Aff}(\Mu) = N_A(\Mu)/\Mu$. 
The quotient of $E^2/\Mu$ by the action of $(e_1/2 + \mathrm{diag}(1,-1))_\star$ is a Klein bottle, 
since $\langle t_2, e_1/2 + \mathrm{diag}(1,-1)\rangle$ is a Klein bottle group. 
\end{proof}

Let $\alpha = a + A$ and $\beta = b+B$ be an affinities of $E^2$ such that $\alpha$ and $\beta$ normalize $\Mu$, and 
$\alpha_\star$ and $\beta_\star$ are order 2 affinities of $T^2$ which do not fix a point of $T^2$, and $\Omega(\alpha_\star\beta_\star)$ has finite order. 
Then $A$ and $B$ have order 1 or 2.  If $A = I$, then $\alpha_\star$ is conjugate to $(e_1/2+I)_\star$, and  
if $A$ has order 2, then $\alpha_\star$ is conjugate to $(e_1/2+\mathrm{diag}(1,-1))_\star$ by Lemma 5. 
The same is true for $B$. 
For each positive integer $m$, let $D_m$ be the dihedral group of order $2m$. 
Now $\langle A, B\rangle$ is a finite subgroup of $\mathrm{GL}(2,\integers)$ 
which either has order 1 or is isomorphic to $D_m$ for some $m$. 
By considering all possible $\integers$-classes of groups isomorphic to $D_m$ for some $m$ in Table 1A in \cite{B-Z}, 
we see that $\langle A, B\rangle$ has order 1 or 2 or $\langle A, B\rangle$ is conjugate 
to $\langle \mathrm{diag}(1,-1), \mathrm{diag}(-1,1)\rangle$. 
We find that there are 7 equivalence classes of pairs $\{\alpha_\star,\beta_\star\}$, and so 
there are 7 affine equivalence classes of flat $T^2\,|\, \Iota$ fiberings by Theorem 10. 
These fiberings are represented by the generalized Calabi constructions described in Table 7 
as explained in \S 8. 
Note that the $O^3_6\,|\,\Iota$ fibering described in Example 2 corresponds to the fibering 
described in Row 2 of Table 7. 

\begin{table} 

\begin{tabular}{c|c|c|c|c|c}	
no. & mfd. & cS-fbr.  & grp. & pair representatives  & s-fbrs. \\ \hline	
 1	&$O^3_2$ & $T^2\,|\,\Iota$ & $D_1$ & 
	  	$ \left(\begin{array}{rrr} 
		\frac{1}{2}& 1 &  0 \\ 0 & 0 & -1 \end{array}\right)$, 
		$ \left(\begin{array}{rrr} 
		\frac{1}{2}& 1 &  0 \\ 0 & 0 & -1 \end{array}\right)$ & $K^2, K^2$ \\ \hline
 2	&$O^3_6$ & $T^2\,|\,\Iota$ &  $D_2$ &
		$ \left(\begin{array}{rrr} 
		\frac{1}{2}& 1 &  0 \\ 0 & 0 & -1 \end{array}\right)$, 
		$ \left(\begin{array}{rrr} 
		0 & -1 &  0 \\ \frac{1}{2} & 0 & 1 \end{array}\right)$ & $K^2, K^2$ \\ \hline 
3	&$N^3_1$ & $T^2\,|\,\Iota$ & $D_1$ &
 		$ \left(\begin{array}{rrr} 
		\frac{1}{2}& 1 &  0 \\ 0 & 0 & 1 \end{array}\right)$, 
		$ \left(\begin{array}{rrr} 
		\frac{1}{2}& 1 &  0 \\ 0 & 0 & 1 \end{array}\right)$ & $T^2, T^2$ \\ \hline
 4	&$N^3_2$ & $T^2\,|\,\Iota$ &  $D_2$ &
 		$ \left(\begin{array}{rrr} 
		\frac{1}{2}& 1 &  0 \\ 0 & 0 & 1 \end{array}\right)$, 
		$ \left(\begin{array}{rrr} 
		0 & 1 & 0 \\  \frac{1}{2} & 0 & 1 \end{array}\right)$ & $T^2, T^2$ \\ \hline
 5	&$N^3_3$ & $T^2\,|\,\Iota$ &  $D_2$ &
 		$ \left(\begin{array}{rrr} 
		\frac{1}{2}& 1 &  0 \\ 0 & 0 & 1 \end{array}\right)$, 
		$ \left(\begin{array}{rrr} 
		\frac{1}{2}& 1 &  0 \\ 0 & 0 & -1 \end{array}\right)$ & $T^2, K^2$ \\ \hline
 6 	&$N^3_4$ & $T^2\,|\,\Iota$ &  $D_2$ &
 	$ \left(\begin{array}{rrr} 
		0& 1 &  0 \\ \frac{1}{2} & 0 & 1 \end{array}\right)$, 
		$ \left(\begin{array}{rrr} 
		\frac{1}{2}& 1 &  0 \\ 0 & 0 & -1 \end{array}\right)$  & $T^2, K^2$ \\ \hline
 7 	&$N^3_4$ & $T^2\,|\,\Iota$ &  $D_2$ &
 	$ \left(\begin{array}{rrr} 
		\frac{1}{2} & 1 &  0 \\ \frac{1}{2} & 0 & 1 \end{array}\right)$, 
		$ \left(\begin{array}{rrr} 
		\frac{1}{2}& 1 &  0 \\ 0 & 0 & -1 \end{array}\right)$ & $T^2, K^2$ \\ \hline
\end{tabular}

\vspace{.2in}
\caption{The flat $T^2$ fiberings over $\Iota$}
\end{table}

Next we consider $K^2\,|\, S^1$ fiberings. Here $K^2=E^2/\Mu$ 
with $\Mu= \langle t_1, t_2, \alpha \rangle$ and $\alpha =  e_1/2+\mathrm{diag}(1,-1)$.  

\begin{lemma} {\rm (Lemmas 38, 39 \cite{R-TII})}  
Let  $\Mu = \langle t_1, t_2, \alpha \rangle$, with $\alpha = e_1/2+ A$ and $A = \mathrm{diag}(1,-1)$. 
Then 
$$N_A(\Mu) = \{b+B: 2b_2 \in \integers \ \, \hbox{and}\ \, B \in \langle -I, A\rangle\}. $$
The Lie group $\mathrm{Aff}(\Mu)$ is the Cartesian product of the subgroup $\langle (e_2/2+I)_\star\rangle$ of order $2$ 
and the orthogonal group $\langle (te_1\pm I)_\star: t\in \realnos\rangle$ of rank $1$. 
Moreover, the epimorphism $\Omega:  \mathrm{Aff}(\Mu) \to \mathrm{Out}(\Mu)$ maps the subgroup 
$ \{I_\star , (e_2/2+I)_\star, (-I)_\star, (e_2/2-I)_\star\}$  
isomorphically onto $\mathrm{Out}(\Mu)$.  

\end{lemma}
By Lemma 6 and Theorem 7, there are 4 affine equivalence classes of flat $K^2\,|\, S^1$ fiberings represented 
by the generalized Calabi constructions described in Table 8 as explained in \S 8. 

\begin{table} 

\begin{tabular}{c|c|c|c|c}	
no. & mfd. & cS-fbr.  & grp. & representative      \\ \hline	
 1	&$N^3_1$ & $K^2\,|\,S^1$ &  $C_1$ & 
	  	$ \left(\begin{array}{rr} 
		1 &  0 \\ 0 & 1 \end{array}\right)$   \\ \hline
 2	&$N^3_2$ & $K^2\,|\,S^1$ &  $C_2$ &
		$ \left(\begin{array}{rrr} 
		0& 1 &  0 \\ \frac{1}{2} & 0 & 1 \end{array}\right)$    \\ \hline
 3	&$N^3_3$ & $K^2\,|\,S^1$ &  $C_2$ &
 		$ \left(\begin{array}{rr} 
		-1 &  0 \\ 0 & -1 \end{array}\right)$   \\ \hline
 4	&$N^3_4$ & $K^2\,|\,S^1$ &  $C_2$ &
 		$ \left(\begin{array}{rrr} 
		0& -1 &  0\\ \frac{1}{2}& 0 & -1 \end{array}\right)$    \\ \hline
\end{tabular}

\vspace{.2in}
\caption{The flat $K^2$ fiberings over $S^1$}
\end{table}

Next we consider $K^2\,|\, \Iota$ fiberings. 
By Lemmas 3 and 6, the isometry $(e_2/2 + I)_\star$ is the unique order 2 affinity of $K^2$ that does not fix a point of $K^2$. 
By Theorem 10, there is exactly 1 affine equivalence class of flat $K^2\,|\, \Iota$ fiberings represented 
by the generalized Calabi construction described in Table 9 as explained in \S 8. 

\begin{table} 

\begin{tabular}{c|c|c|c|c|c}	
no. & mfd. & cS-fbr.  &  grp. & pair representatives  & s-fbrs.  \\ \hline	
 1	&$N^3_3$ & $K^2\,|\,\Iota$ &  $D_1$ &
 		$ \left(\begin{array}{rrr} 
		0 & 1 &  0 \\ \frac{1}{2} & 0 & 1 \end{array}\right)$, 
		$ \left(\begin{array}{rrr} 
		0 & 1 &  0 \\ \frac{1}{2} & 0 & 1 \end{array}\right)$  & $K^2, K^2$ \\ \hline
 \end{tabular}

\vspace{.2in}
\caption{The flat $K^2$ fiberings over $\Iota$}
\end{table}

\section{Fibrations of Closed Flat 4-Manifolds with Generic Fiber $O^3_1$} 

In this section, we describe the affine classification of the geometric co-Seifert fibrations of closed flat 4-manifolds with generic fiber 
the 3-torus $O^3_1$.  Here $O^3_1 = E^3/\Mu$ with $\Mu = \langle t_1,t_2,t_3\rangle$. 
The following lemma is well known. 

\begin{lemma} 
Let $\Mu = \langle t_1, t_2, t_3\rangle$. We have that $\mathrm{Out}(\Mu) = \mathrm{Aut}(\Mu)= \mathrm{GL}(3,\integers)$ and 
$$N_A(\Mu) = \{a+A: a\in E^3 \ and\ A \in \mathrm{GL}(3,\integers)\}.$$
The map $\eta: \mathrm{Aff}(\Mu) \to \mathrm{GL}(3,\integers)$, defined by 
$\eta((a+A)_\star) = A$, is an epimorphism with kernel $\Kappa = \{(a+I)_\star: a \in E^3\}$. 
The map $\sigma:  \mathrm{GL}(3,\integers) \to \mathrm{Aff}(\Mu)$, defined by $\sigma(A) = A_\star$, 
is a monomorphism,  and $\sigma$ is a right inverse of $\eta$. 
The group $\Kappa$ is the maximal torus of the Lie group $\mathrm{Aff}(\Mu)$ of rank 3. 
\end{lemma}

First we consider $O^3_1\,|\, S^1$ fiberings. 
The conjugacy classes of finite subgroups of $\mathrm{GL}(3,\integers)$ are called the $\integers$-classes of finite subgroups of $\mathrm{GL}(3,\integers)$ in \cite{B-Z}. 
From Tables 1B and 6B in \cite{B-Z}, we see that there are 16 conjugacy classes of finite cyclic subgroups of $\mathrm{GL}(3,\integers)$. 
An element of $\mathrm{GL}(3,\integers)$ of finite order has order at most 6. 
Hence there are 16 conjugacy classes of inverse pairs of elements of $\mathrm{GL}(3,\integers)$ of finite order.  
Representatives of these classes, taken from Table 1B in \cite{B-Z}, are listed in Column 5 of Tables 10 and 11. 
By Theorem 7, there are 16 affine equivalence classes of flat $O^3_1\,|\, S^1$ fiberings. 
These fiberings are represented by the generalized Calabi constructions described in Tables 10 and 11 
as explained in \S 8. 

Replacing a representative in Column 5 by its inverse does not change the proper space group type, 
since no torsion-free 4-space group type splits into two enantiomorphic proper space group types;  see Table 7C of \cite{B-Z} and the correction \cite{NSW}.

\begin{table} 

\begin{tabular}{c|c|c|c|c}	
no. & mfd. & cS-fbr.  &  grp. & representative      \\ \hline	
 1	&$O^4_1$ & $O^3_1\,|\,S^1$  & $C_1$ & 
	  	$ \left(\begin{array}{rrr} 
		1 &  0 & 0 \\ 0 & 1 & 0 \\ 0 & 0 & 1 \end{array}\right)$   \\ \hline
 2	&$O^4_2$ & $O^3_1\,|\,S^1$ &  $C_2$ &
		$ \left(\begin{array}{rrr} 
		-1 &  0 & 0\\ 0 & 1 & 0 \\ 0 & 0 & -1\end{array}\right)$    \\ \hline
 3	&$O^4_3$ & $O^3_1\,|\,S^1$ &  $C_2$ &
 		$ \left(\begin{array}{rrr} 
		0 &  -1 & 0\\ -1 & 0 & 0 \\ 0 & 0 & -1 \end{array}\right)$   \\ \hline
 4	&$O^4_4$ & $O^3_1\,|\,S^1$ &  $C_3$ &
 		$ \left(\begin{array}{rrr} 
		0 &  -1 & 0\\ 1 & -1 & 0 \\ 0 & 0 & 1 \end{array}\right)$    \\ \hline
 5	&$O^4_5$ & $O^3_1\,|\,S^1$ &  $C_3$ &
 		$ \left(\begin{array}{rrr} 
		0 &  0 & 1 \\ 1 & 0 & 0 \\ 0 & 1 & 0 \end{array}\right)$     \\ \hline
 6 	&$O^4_6$ & $O^3_1\,|\,S^1$ &  $C_4$ &
 	$ \left(\begin{array}{rrr} 
		0 & -1 &  0 \\ 1 & 0 & 0 \\ 0 & 0 & 1 \end{array}\right)$   \\ \hline
 7 	&$O^4_7$ & $O^3_1\,|\,S^1$ &  $C_4$ &
 	$ \left(\begin{array}{rrr} 
		0 &  1 & 0\\ 0 & 1 & -1 \\ -1 & 1 & 0 \end{array}\right)$    \\ \hline
8 	&$O^4_8$ & $O^3_1\,|\,S^1$ &  $C_6$ &
 	$ \left(\begin{array}{rrr} 
		0 &  1 & 0\\ -1 & 1 & 0 \\ 0 & 0 & 1 \end{array}\right)$    \\ \hline
\end{tabular}

\vspace{.2in}
\caption{The orientable flat $O^3_1$ fiberings over $S^1$}
\end{table}
\begin{table} 

\begin{tabular}{c|c|c|c|c}	
no. & mfd. & cS-fbr.  &  grp. & representative      \\ \hline
9	&$N^4_1$ & $O^3_1\,|\,S^1$ &  $C_2$ & 
	  	$ \left(\begin{array}{rrr} 
		1 &  0 & 0 \\ 0 & -1 & 0 \\ 0 & 0 & 1 \end{array}\right)$   \\ \hline
10	&$N^4_2$ & $O^3_1\,|\,S^1$ &  $C_2$ & 
	  	$ \left(\begin{array}{rrr} 
		0 &  1 & 0 \\ 1 & 0 & 0 \\ 0 & 0 & 1 \end{array}\right)$   \\ \hline
11	&$N^4_{14}$ & $O^3_1\,|\,S^1$ &  $C_2$ & 
	  	$ \left(\begin{array}{rrr} 
		-1 &  0 & 0 \\ 0 & -1 & 0 \\ 0 & 0 & -1 \end{array}\right)$   \\ \hline
12	&$N^4_{15}$ & $O^3_1\,|\,S^1$ & $C_4$ & 
	  	$ \left(\begin{array}{rrr} 
		0 &  1 & 0 \\ -1 & 0 & 0 \\ 0 & 0 & -1 \end{array}\right)$   \\ \hline
13	&$N^4_{16}$ & $O^3_1\,|\,S^1$ &  $C_4$ & 
	  	$ \left(\begin{array}{rrr} 
		0 &  -1 & 0 \\ 0 & -1 & 1 \\ 1 & -1 & 0 \end{array}\right)$   \\ \hline
14	&$N^4_{19}$ & $O^3_1\,|\,S^1$ &  $C_6$ & 
	  	$ \left(\begin{array}{rrr} 
		0 &  -1 & 0 \\ 1 & -1 & 0 \\ 0 & 0 & -1 \end{array}\right)$   \\ \hline
15	&$N^4_{20}$ & $O^3_1\,|\,S^1$ &  $C_6$ & 
	  	$ \left(\begin{array}{rrr} 
		0 &  1 & 0 \\ -1 & 1 & 0 \\ 0 & 0 & -1 \end{array}\right)$   \\ \hline
16	&$N^4_{21}$ & $O^3_1\,|\,S^1$ &  $C_6$ & 
	  	$ \left(\begin{array}{rrr} 
		0 &  0 & -1 \\ -1 & 0 & 0 \\ 0 & -1 & 0 \end{array}\right)$   \\ \hline
\end{tabular}

\vspace{.2in}
\caption{The nonorientable flat $O^3_1$ fiberings over $S^1$}
\end{table}
\begin{table} 

\begin{tabular}{c|c|c|c|c|c}	
no. & mfd. & cS-fbr.  &  grp. & pair representatives  & s-fbrs.  \\ \hline	
 1	&$O^4_2$ & $O^3_1\,|\,\Iota$ & $D_1$ &
 		$ \left(\begin{array}{rrrr} 
		\frac{1}{2} & 1 & 0 &  0 \\ 0 & 0 & -1 & 0 \\ 0 & 0 & 0 & 1\end{array}\right)$, 
		$ \left(\begin{array}{rrrr} 
		\frac{1}{2} & 1 & 0 &  0 \\ 0 & 0 & -1 & 0 \\ 0 & 0 & 0 & 1\end{array}\right)$ & $N^3_1, N^3_1$ \\ \hline
2	&$O^4_3$ & $O^3_1\,|\,\Iota$ & $D_1$ &
 		$ \left(\begin{array}{rrrr} 
		0 & 0 & 1 &  0 \\ 0 & 1 & 0 & 0 \\ \frac{1}{2} & 0 & 0 & 1\end{array}\right)$, 
		$ \left(\begin{array}{rrrr} 
		0 & 0 & 1 &  0 \\ 0 & 1 & 0 & 0 \\ \frac{1}{2} & 0 & 0 & 1\end{array}\right)$ & $N^3_2, N^3_2$  \\ \hline
3	&$O^4_3$ & $O^3_1\,|\,\Iota$ & $D_2$ &
 		$ \left(\begin{array}{rrrr} 
		\frac{1}{2} & 1 & 0 &  0 \\ 0 & 0 & -1 & 0 \\ 0 & 0 & 0 & 1\end{array}\right)$, 
		$ \left(\begin{array}{rrrr} 
		0 & 1 & 0 &  0 \\ 0 & 0 & -1 & 0 \\ \frac{1}{2} & 0 & 0 & 1\end{array}\right)$ & $N^3_1, N^3_1$ \\ \hline
4	&$O^4_9$ & $O^3_1\,|\,\Iota$ & $D_2$ &
 		$ \left(\begin{array}{rrrr} 
		0 & -1 & 0 &  0 \\ 0 & 0 & 1 & 0 \\ \frac{1}{2} & 0 & 0 & 1\end{array}\right)$, 
		$ \left(\begin{array}{rrrr} 
		0 & 1 & 0 &  0 \\ 0 & 0 & -1 & 0 \\ \frac{1}{2} & 0 & 0 & 1\end{array}\right)$ & $N^3_1, N^3_1$ \\ \hline
5	&$O^4_{10}$ & $O^3_1\,|\,\Iota$ & $D_2$ &
 		$ \left(\begin{array}{rrrr} 
		0 & -1 & 0 &  0 \\ \frac{1}{2} & 0 & 1 & 0 \\ 0 & 0 & 0 & 1\end{array}\right)$, 
		$ \left(\begin{array}{rrrr} 
		0 & 1 & 0 &  0 \\ 0 & 0 & -1 & 0 \\ \frac{1}{2} & 0 & 0 & 1\end{array}\right)$ & $N^3_1, N^3_1$ \\ \hline
6	&$O^4_{10}$ & $O^3_1\,|\,\Iota$ & $D_2$ &
 		$ \left(\begin{array}{rrrr} 
		0 & -1 & 0 &  0 \\ 0 & 0 & 1 & 0 \\ \frac{1}{2} & 0 & 0 & 1\end{array}\right)$, 
		$ \left(\begin{array}{rrrr} 
		\frac{1}{2} & 1 & 0 &  0 \\ 0 & 0 & -1 & 0 \\ \frac{1}{2} & 0 & 0 & 1\end{array}\right)$ & $N^3_1, N^3_1$ \\ \hline
7	&$O^4_{11}$ & $O^3_1\,|\,\Iota$ & $D_2$ &
 		$ \left(\begin{array}{rrrr} 
		0 & -1 & 0 &  0 \\ \frac{1}{2} & 0 & 1 & 0 \\ 0 & 0 & 0 & 1\end{array}\right)$, 
		$ \left(\begin{array}{rrrr} 
		\frac{1}{2} & 1 & 0 &  0 \\ 0 & 0 & -1 & 0 \\ \frac{1}{2} & 0 & 0 & 1\end{array}\right)$ & $N^3_1, N^3_1$ \\ \hline
8	&$O^4_{11}$ & $O^3_1\,|\,\Iota$ & $D_2$ &
 		$ \left(\begin{array}{rrrr} 
		0 & -1 & 0 &  0 \\ \frac{1}{2} & 0 & 1 & 0 \\ \frac{1}{2} & 0 & 0 & 1\end{array}\right)$, 
		$ \left(\begin{array}{rrrr} 
		\frac{1}{2} & 1 & 0 &  0 \\ 0 & 0 & -1 & 0 \\ \frac{1}{2} & 0 & 0 & 1\end{array}\right)$ & $N^3_1, N^3_1$ \\ \hline
9	&$O^4_{12}$ & $O^3_1\,|\,\Iota$ & $D_2$ &
 		$ \left(\begin{array}{rrrr} 
		0 & 0 & -1 &  0 \\ 0 & -1 & 0 & 0 \\ \frac{1}{2} & 0 & 0 & 1\end{array}\right)$, 
		$ \left(\begin{array}{rrrr} 
		0 & 0 & 1 &  0 \\ 0 & 1 & 0 & 0 \\ \frac{1}{2} & 0 & 0 & 1\end{array}\right)$ & $N^3_2, N^3_2$  \\ \hline
10	&$O^4_{13}$ & $O^3_1\,|\,\Iota$ & $D_2$ &
 		$ \left(\begin{array}{rrrr} 
		0 & 0 & 0 & -1 \\ \frac{1}{2} & 1 & 1 & 1 \\ 0 & -1 & 0 & 0\end{array}\right)$, 
		$ \left(\begin{array}{rrrr} 
		\frac{1}{2} & 1 & 1 &  1 \\ 0 & 0 & 0 & -1 \\ 0 & 0 & -1 & 0\end{array}\right)$ & $N^3_2, N^3_2$  \\ \hline
11	&$O^4_{13}$ & $O^3_1\,|\,\Iota$ & $D_4$ &
		$ \left(\begin{array}{rrrr} 
		 \frac{1}{2} & 1 & 0 &  0 \\ 0 & 0 & 1 & 0 \\ 0 & 0 & 0 & -1\end{array}\right)$,
		 $ \left(\begin{array}{rrrr} 
		0 & 0 & 1 &  0 \\ 0 & 1 & 0 & 0 \\ \frac{1}{2} & 0 & 0 & 1\end{array}\right)$ & $N^3_1, N^3_2$ \\  \hline
12	&$O^4_{14}$ & $O^3_1\,|\,\Iota$ & $D_2$ &
 		$ \left(\begin{array}{rrrr} 
		0 & -1 & 0 &  0 \\ \frac{1}{2} & 0 & 1 & 0 \\ 0 & 0 & 0 & 1\end{array}\right)$, 
		$ \left(\begin{array}{rrrr} 
		\frac{1}{2} & 1 & 0 &  0 \\ 0 & 0 & -1 & 0 \\ 0 & 0 & 0 & 1\end{array}\right)$ & $N^3_1, N^3_1$ \\ \hline
13	&$O^4_{16}$ & $O^3_1\,|\,\Iota$ & $D_2$ &
 		$ \left(\begin{array}{rrrr} 
		0 & 0 & 1 & -1 \\ \frac{1}{2} & 0 & 1 & 0 \\ \frac{1}{2} & -1 & 1 & 0\end{array}\right)$, 
		$ \left(\begin{array}{rrrr} 
		\frac{1}{2} & 1 & 0 &  0 \\ 0 & 1 & 0 & -1 \\ \frac{1}{2} & 1 & -1 & 0\end{array}\right)$ & $N^3_2, N^3_2$  \\ \hline
14	&$O^4_{17}$ & $O^3_1\,|\,\Iota$ & $D_2$ &
		$ \left(\begin{array}{rrrr} 
		 \frac{1}{2} & 1 & 0 &  0 \\ \frac{1}{2} & 0 & 1 & 0 \\ 0 & 0 & 0 & -1\end{array}\right)$,
		 $ \left(\begin{array}{rrrr} 
		0 & 0 & 1 &  0 \\ 0 & 1 & 0 & 0 \\ \frac{1}{2} & 0 & 0 & 1\end{array}\right)$ & $N^3_1, N^3_2$  \\ \hline
 \end{tabular}

\vspace{.2in}
\caption{The flat $O^3_1$ fiberings over $\Iota$, Part I}
\end{table}

\begin{table} 

\begin{tabular}{c|c|c|c|c|c}	
no. & mfd. & cS-fbr.  &  grp. & pair representatives  & s-fbrs.  \\ \hline	
15	&$O^4_{18}$ & $O^3_1\,|\,\Iota$ & $D_3$ &
 		$ \left(\begin{array}{rrrr} 
		0 & -1 & 0 &  0 \\ 0 & -1 & 1 & 0 \\ \frac{1}{2} & 0 & 0 & 1\end{array}\right)$, 
		$ \left(\begin{array}{rrrr} 
		0 & 0 & 1 &  0 \\ 0 & 1 & 0 & 0 \\ \frac{1}{2} & 0 & 0 & 1\end{array}\right)$ & $N^3_2, N^3_2$ \\ \hline
16	&$O^4_{19}$ & $O^3_1\,|\,\Iota$ & $D_3$ &
 		$ \left(\begin{array}{rrrr} 
		0 & -1 & -1 &  0 \\ 0 & 0 & 1 & 0 \\ \frac{1}{2} & 0 & 0 & 1\end{array}\right)$, 
		$ \left(\begin{array}{rrrr} 
		0 & 0 & 1 &  0 \\ 0 & 1 & 0 & 0 \\ \frac{1}{2} & 0 & 0 & 1\end{array}\right)$ & $N^3_2, N^3_2$ \\ \hline
17	&$O^4_{20}$ & $O^3_1\,|\,\Iota$ & $D_3$ &
 		$ \left(\begin{array}{rrrr} 
		0 & 0 & 0 &  1 \\ \frac{1}{2} & 0 & 1 & 0 \\ 0 & 1 & 0 & 0\end{array}\right)$, 
		$ \left(\begin{array}{rrrr} 
		0 & 0 & 1 &  0 \\ 0 & 1 & 0 & 0 \\ \frac{1}{2} & 0 & 0 & 1\end{array}\right)$ & $N^3_2, N^3_2$ \\ \hline 
18	&$O^4_{21}$ & $O^3_1\,|\,\Iota$ & $D_4$ &
 		$ \left(\begin{array}{rrrr} 
		0 & -1 & 0 &  0 \\ 0 & 0 & 1 & 0 \\ \frac{1}{2} & 0 & 0 & 1\end{array}\right)$, 
		$ \left(\begin{array}{rrrr} 
		0 & 0 & 1 &  0 \\ 0 & 1 & 0 & 0 \\ \frac{1}{2} & 0 & 0 & 1\end{array}\right)$ & $N^3_1, N^3_2$ \\ \hline  
19	&$O^4_{22}$ & $O^3_1\,|\,\Iota$ & $D_4$ &
 		$ \left(\begin{array}{rrrr} 
		0 & -1 & 0 &  0 \\ \frac{1}{2} & 0 & 1 & 0 \\ \frac{1}{2} & 0 & 0 & 1\end{array}\right)$, 
		$ \left(\begin{array}{rrrr} 
		0 & 0 & 1 &  0 \\ 0 & 1 & 0 & 0 \\ \frac{1}{2} & 0 & 0 & 1\end{array}\right)$ & $N^3_1, N^3_2$ \\ \hline 
20	&$O^4_{23}$ & $O^3_1\,|\,\Iota$ & $D_4$ &
 		$ \left(\begin{array}{rrrr} 
		0 & -1 & 0 &  0 \\ \frac{1}{2} & 0 & 1 & 0 \\ 0 & 0 & 0 & 1\end{array}\right)$, 
		$ \left(\begin{array}{rrrr} 
		0 & 0 & 1 &  0 \\ 0 & 1 & 0 & 0 \\ \frac{1}{2} & 0 & 0 & 1\end{array}\right)$ & $N^3_1, N^3_2$ \\ \hline  
21	&$O^4_{24}$ & $O^3_1\,|\,\Iota$ & $D_4$ &
 		$ \left(\begin{array}{rrrr} 
		0 & 0 & 1 &  -1 \\ \frac{1}{2} & 0 & 1 & 0 \\ \frac{1}{2} & -1 & 1 & 0\end{array}\right)$, 
		$ \left(\begin{array}{rrrr} 
		0 & 0 & 1 &  0 \\ 0 & 1 & 0 & 0 \\ \frac{1}{2} & 0 & 0 & 1\end{array}\right)$ & $N^3_2, N^3_2$ \\ \hline 
22	&$O^4_{25}$ & $O^3_1\,|\,\Iota$ & $D_6$ &
 		$ \left(\begin{array}{rrrr} 
		0 & -1 & 0 &  0 \\ 0 & 1 & 1 & 0 \\ \frac{1}{2} & 0 & 0 & 1\end{array}\right)$, 
		$ \left(\begin{array}{rrrr} 
		0 & 0 & 1 &  0 \\ 0 & 1 & 0 & 0 \\ \frac{1}{2} & 0 & 0 & 1\end{array}\right)$ & $N^3_2, N^3_2$ \\ \hline  
23	&$N^4_1$ & $O^3_1\,|\,\Iota$ & $D_1$ &
 		$ \left(\begin{array}{rrrr} 
		\frac{1}{2} & 1 & 0 &  0 \\ 0 & 0 & 1 & 0 \\ 0 & 0 & 0 & 1\end{array}\right)$, 
		$ \left(\begin{array}{rrrr} 
		\frac{1}{2} & 1 & 0 &  0 \\ 0 & 0 & 1 & 0 \\ 0 & 0 & 0 & 1\end{array}\right)$ & $O^3_1, O^3_1$ \\ \hline
24	&$N^4_2$ & $O^3_1\,|\,\Iota$ & $D_2$ &
 		$ \left(\begin{array}{rrrr} 
		\frac{1}{2} & 1 & 0 &  0 \\ 0 & 0 & 1 & 0 \\ 0 & 0 & 0 & 1\end{array}\right)$, 
		$ \left(\begin{array}{rrrr} 
		0 & 1 & 0 &  0 \\ \frac{1}{2} & 0 & 1 & 0 \\ 0 & 0 & 0 & 1\end{array}\right)$ & $O^3_1, O^3_1$ \\ \hline
25	&$N^4_4$ & $O^3_1\,|\,\Iota$ & $D_2$ &
		$ \left(\begin{array}{rrrr} 
		0 & 1 & 0 &  0 \\ 0 & 0 & 1 & 0 \\ \frac{1}{2} & 0 & 0 & 1\end{array}\right)$, 
		$ \left(\begin{array}{rrrr} 
		\frac{1}{2} & 1 & 0 &  0 \\ 0 & 0 & -1 & 0 \\ 0 & 0 & 0 & 1\end{array}\right)$ & $O^3_1, N^3_1$ \\ \hline
26	&$N^4_5$ & $O^3_1\,|\,\Iota$ & $D_2$ &
		$ \left(\begin{array}{rrrr} 
		0 & 1 & 0 &  0 \\ \frac{1}{2} & 0 & 1 & 0 \\ \frac{1}{2} & 0 & 0 & 1\end{array}\right)$, 
		$ \left(\begin{array}{rrrr} 
		\frac{1}{2} & 1 & 0 &  0 \\ 0 & 0 & -1 & 0 \\ 0 & 0 & 0 & 1\end{array}\right)$ & $O^3_1, N^3_1$ \\ \hline
27	&$N^4_7$ & $O^3_1\,|\,\Iota$ & $D_4$ &
		$ \left(\begin{array}{rrrr} 
		\frac{1}{2} & 1 & 0 &  0 \\ 0 & 0 & 1 & 0 \\ 0 & 0 & 0 & 1\end{array}\right)$, 
		$ \left(\begin{array}{rrrr} 
		0 & 0 & 1 &  0 \\ 0 & 1 & 0 & 0 \\ \frac{1}{2} & 0 & 0 & 1\end{array}\right)$ & $O^3_1, N^3_2$ \\ \hline
28	&$N^4_8$ & $O^3_1\,|\,\Iota$ & $D_2$ &
		$ \left(\begin{array}{rrrr} 
		\frac{1}{2} & 1 & 0 &  0 \\ 0 & 0 & 1 & 0 \\ 0 & 0 & 0 & 1\end{array}\right)$, 
		$ \left(\begin{array}{rrrr} 
		\frac{1}{2} & 1 & 0 &  0 \\ 0 & 0 & -1 & 0 \\ 0 & 0 & 0 & 1\end{array}\right)$ & $O^3_1, N^3_1$ \\ \hline
 \end{tabular}

\vspace{.2in}
\caption{The flat $O^3_1$ fiberings over $\Iota$, Part II}
\end{table}

\begin{table} 

\begin{tabular}{c|c|c|c|c|c}	
no. & mfd. & cS-fbr.  &  grp. & pair representatives  & s-fbrs.  \\ \hline	
29	&$N^4_9$ & $O^3_1\,|\,\Iota$ & $D_2$ &
		$ \left(\begin{array}{rrrr} 
		0 & 1 & 0 &  0 \\ \frac{1}{2} & 0 & 1 & 0 \\ 0 & 0 & 0 & 1\end{array}\right)$,
		$ \left(\begin{array}{rrrr} 
		\frac{1}{2} & 1 & 0 &  0 \\ 0 & 0 & -1 & 0 \\ 0 & 0 & 0 & 1\end{array}\right)$ & $O^3_1, N^3_1$ \\ \hline
30	&$N^4_9$ & $O^3_1\,|\,\Iota$ & $D_2$ &
		$ \left(\begin{array}{rrrr} 
		\frac{1}{2} & 1 & 0 &  0 \\ \frac{1}{2} & 0 & 1 & 0 \\ 0 & 0 & 0 & 1\end{array}\right)$, 
		$ \left(\begin{array}{rrrr} 
		\frac{1}{2} & 1 & 0 &  0 \\ 0 & 0 & -1 & 0 \\ 0 & 0 & 0 & 1\end{array}\right)$ & $O^3_1, N^3_1$ \\ \hline
31	&$N^4_{11}$ & $O^3_1\,|\,\Iota$ & $D_2$ &
		$ \left(\begin{array}{rrrr} 
		0 & 1 & 0 &  0 \\ 0 & 0 & 1 & 0 \\ \frac{1}{2} & 0 & 0 & 1\end{array}\right)$, 
		$ \left(\begin{array}{rrrr} 
		0 & 0 & 1 &  0 \\ 0 & 1 & 0 & 0 \\ \frac{1}{2} & 0 & 0 & 1\end{array}\right)$ & $O^3_1, N^3_2$ \\ \hline
32	&$N^4_{12}$ & $O^3_1\,|\,\Iota$ & $D_2$ &
		$ \left(\begin{array}{rrrr} 
		\frac{1}{2} & 1 & 0 &  0 \\ \frac{1}{2} & 0 & 1 & 0 \\ 0 & 0 & 0 & 1\end{array}\right)$, 
		$ \left(\begin{array}{rrrr} 
		0 & 0 & 1 &  0 \\ 0 & 1 & 0 & 0 \\ \frac{1}{2} & 0 & 0 & 1\end{array}\right)$ & $O^3_1, N^3_2$ \\ \hline
33	&$N^4_{14}$ & $O^3_1\,|\,\Iota$ & $D_1$ &
 		$ \left(\begin{array}{rrrr} 
		0 & -1 & 0 &  0 \\ \frac{1}{2} & 0 & 1 & 0 \\ 0 & 0 & 0 & -1\end{array}\right)$, 
		$ \left(\begin{array}{rrrr} 
		0 & -1 & 0 &  0 \\ \frac{1}{2} & 0 & 1 & 0 \\ 0 & 0 & 0 & -1\end{array}\right)$ & $O^3_2, O^3_2$ \\ \hline
34	&$N^4_{22}$ & $O^3_1\,|\,\Iota$ & $D_2$ &
		$ \left(\begin{array}{rrrr} 
		0 & 1 & 0 &  0 \\ \frac{1}{2} & 0 & 1 & 0 \\ 0 & 0 & 0 & 1\end{array}\right)$, 
		$ \left(\begin{array}{rrrr} 
		0 & -1 & 0 &  0 \\ \frac{1}{2} & 0 & 1 & 0 \\ 0 & 0 & 0 & -1\end{array}\right)$ & $O^3_1, O^3_2$ \\ \hline
35	&$N^4_{23}$ & $O^3_1\,|\,\Iota$ & $D_2$ &
		$ \left(\begin{array}{rrrr} 
		 \frac{1}{2} & 1 & 0 &  0 \\ \frac{1}{2} & 0 & 1 & 0 \\ 0 & 0 & 0 & 1\end{array}\right)$, 
		 $ \left(\begin{array}{rrrr} 
		0 & -1 & 0 &  0 \\ \frac{1}{2} & 0 & 1 & 0 \\ 0 & 0 & 0 & -1\end{array}\right)$ & $O^3_1, O^3_2$ \\ \hline
36	&$N^4_{23}$ & $O^3_1\,|\,\Iota$ & $D_2$ &
 		$ \left(\begin{array}{rrrr} 
		0 & -1 & 0 &  0 \\ 0 & 0 & -1 & 0 \\ \frac{1}{2} & 0 & 0 & 1\end{array}\right)$, 
		$ \left(\begin{array}{rrrr} 
		0 & -1 & 0 &  0 \\ \frac{1}{2} & 0 & 1 & 0 \\ 0 & 0 & 0 & 1\end{array}\right)$ & $O^3_2, N^3_1$ \\ \hline
37	&$N^4_{24}$ & $O^3_1\,|\,\Iota$ & $D_2$ &
 		$ \left(\begin{array}{rrrr} 
		0 & -1 & 0 &  0 \\ 0 & 0 & -1 & 0 \\ \frac{1}{2} & 0 & 0 & 1\end{array}\right)$, 
		$ \left(\begin{array}{rrrr} 
		0 & -1 & 0 &  0 \\ 0 & 0 & 1 & 0 \\ \frac{1}{2} & 0 & 0 & 1\end{array}\right)$ & $O^3_2, N^3_1$ \\ \hline
38	&$N^4_{25}$ & $O^3_1\,|\,\Iota$ & $D_2$ &
		$ \left(\begin{array}{rrrr} 
		\frac{1}{2} & 1 & 0 &  0 \\ 0 & 0 & 1 & 0 \\ 0 & 0 & 0 & 1\end{array}\right)$, 
		$ \left(\begin{array}{rrrr} 
		0 & -1 & 0 &  0 \\ \frac{1}{2} & 0 & 1 & 0 \\ 0 & 0 & 0 & -1\end{array}\right)$ & $O^3_1, O^3_2$ \\ \hline
39	&$N^4_{25}$ & $O^3_1\,|\,\Iota$ & $D_2$ &
 		$ \left(\begin{array}{rrrr} 
		0 & -1 & 0 &  0 \\ 0 & 0 & -1 & 0 \\ \frac{1}{2} & 0 & 0 & 1\end{array}\right)$, 
		$ \left(\begin{array}{rrrr} 
		0 & -1 & 0 &  0 \\ \frac{1}{2} & 0 & 1 & 0 \\ \frac{1}{2} & 0 & 0 & 1\end{array}\right)$ & $O^3_2, N^3_1$ \\ \hline
40	&$N^4_{26}$ & $O^3_1\,|\,\Iota$ & $D_2$ &
 		$ \left(\begin{array}{rrrr} 
		0 & -1 & 0 &  0 \\ 0 & 0 & -1 & 0 \\ \frac{1}{2} & 0 & 0 & 1\end{array}\right)$, 
		$ \left(\begin{array}{rrrr} 
		0 & 0 & 1 &  0 \\ 0 & 1 & 0 & 0 \\ \frac{1}{2} & 0 & 0 & 1\end{array}\right)$ & $O^3_2, N^3_2$ \\ \hline
41	&$N^4_{44}$ & $O^3_1\,|\,\Iota$ & $D_2$ &
 		$ \left(\begin{array}{rrrr} 
		0 & -1 & 0 &  0 \\ \frac{1}{2} & 0 & 1 & 0 \\ 0 & 0 & 0 & -1\end{array}\right)$, 
		$ \left(\begin{array}{rrrr} 
		\frac{1}{2} & 1 & 0 &  0 \\ 0 & 0 & -1 & 0 \\ 0 & 0 & 0 & 1\end{array}\right)$ & $O^3_2, N^3_1$ \\ \hline
42	&$N^4_{44}$ & $O^3_1\,|\,\Iota$ & $D_2$ &
 		$ \left(\begin{array}{rrrr} 
		0 & -1 & 0 &  0 \\ 0 & 0 & -1 & 0 \\ \frac{1}{2} & 0 & 0 & 1\end{array}\right)$, 
		$ \left(\begin{array}{rrrr} 
		0 & -1 & 0 &  0 \\ \frac{1}{2} & 0 & 1 & 0 \\ 0 & 0 & 0 & -1\end{array}\right)$ & $O^3_2, O^3_2$ \\ \hline
43	&$N^4_{47}$ & $O^3_1\,|\,\Iota$ & $D_4$ &
 		$ \left(\begin{array}{rrrr} 
		0 & -1 & 0 &  0 \\ \frac{1}{2} & 0 & 1 & 0 \\ 0 & 0 & 0 & -1\end{array}\right)$, 
		$ \left(\begin{array}{rrrr} 
		0 & 0 & 1 &  0 \\ 0 & 1 & 0 & 0 \\ \frac{1}{2} & 0 & 0 & 1\end{array}\right)$ & $O^3_2, N^3_2$ \\ \hline
\end{tabular}

\vspace{.2in}
\caption{The flat $O^3_1$ fiberings over $\Iota$, Part III}
\end{table}

Next we consider $O^3_1\,|\, \Iota$ fiberings. 

\begin{lemma} 
Let $\Mu = \langle t_1, t_2, t_3\rangle$, and let $\alpha = a + A$ be an affinity of $E^3$ that normalizes $\Mu$ 
such that $\alpha_\star$ has order $2$ and $\alpha_\star$ does not fix a point of the $3$-torus $E^3/\Mu$. 
Then $A$ has order $1$ or $2$.  
If $A$ has order $1$, then $\alpha_\star$ is conjugate in $\mathrm{Aff}(\Mu)$ to $(e_1/2+I)_\star$. 
The quotient of $E^3/\Mu$ by the action of $(e_1/2+I)_\star$ is a $3$-torus. 
If $A$ has order $2$, then $\alpha_\star$ is conjugate in $\mathrm{Aff}(\Mu)$ to one of 
$(e_2/2 + \mathrm{diag}(-1,1,-1))_\star$, $(e_1/2 + \mathrm{diag}(1,-1,1))_\star$ or $(e_3/2 + B)_\star$ 
with $Be_1 = e_2, Be_2 = e_1$ and $Be_3 = e_3$. 
The quotient of $E^3/\Mu$ by the action of $\alpha_\star$ is of type $O^3_2$, $N^3_1$, or $N^3_2$, respectively. 
\end{lemma}
\begin{proof}
The matrix $A$ has order $1$ or $2$ by Lemma 7. 
If $A$ has order $1$, then $2a \in\integers^3$, 
and so $\alpha_\star$ is conjugate in $\mathrm{Aff}(\Mu)$ to $(e_1/2+I)_\star$. 
The quotient of $E^3/\Mu$ by the action of $(e_1/2+I)_\star$ is a $3$-torus, 
since $\langle e_1/2+I, t_2, t_3\rangle$ is a $3$-torus group. 

Now suppose $A$ has order 2.  
According to Table 1B in \cite{B-Z}, there are 5 possible $\integers$-classes for $A$. 
The group $\langle t_1, t_2, t_3, \alpha\rangle$ is a torsion-free 3-space group, 
since $\alpha_\star$ fixes no point of $E^3/\Mu$. 
According to Table 1B in \cite{B-Z}, only the $\integers$-classes of 3-space groups corresponding to the $\integers$-classes of 
$\mathrm{diag}(-1,1,-1)$, $\mathrm{diag}(1,-1,1)$ and $B$, with $Be_1 = e_2, Be_2 = e_1$ and $Be_3 = e_3$, 
contain a torsion-free group, namely $\langle t_1, t_2, t_3,\beta\rangle$ for $\beta = e_2/2 + \mathrm{diag}(-1,1,-1)$, 
$e_1/2 + \mathrm{diag}(1,-1,1)$ and $e_3/2 + B$ with $Be_1 = e_2, Be_2 = e_1$ and $Be_3 = e_3$. 
Therefore $\alpha_\star$ is conjugate in $\mathrm{Aff}(\Mu)$ to one of $(e_2/2 + \mathrm{diag}(-1,1,-1))_\star$, 
$(e_1/2 + \mathrm{diag}(1,-1,1))_\star$ or $(e_3/2 + B)_\star$ 
with $Be_1 = e_2, Be_2 = e_1$ and $Be_3 = e_3$, 
since $\mathrm{Aff}(\Mu) = N_A(\Mu)/\Mu$. 

The quotient of $E^3/\Mu$ by the action of $(e_2/2 + \mathrm{diag}(-1,1,-1))_\star$ is of type $O^3_2$, 
since the 3-space group $\langle t_1, t_3, e_2/2 + \mathrm{diag}(-1,1,-1)\rangle$ has IT number 4 according to Table 1B of \cite{B-Z}. 
The quotient of $E^3/\Mu$ by the action of $(e_1/2 + \mathrm{diag}(1,-1,1))_\star$ is of type $N^3_1$, 
since the 3-space group $\langle t_2, t_3, e_1/2 + \mathrm{diag}(1,-1,1)\rangle$ has IT number 7 according to Table 1B of \cite{B-Z}. 
The quotient of $E^3/\Mu$ by the action of $(e_3/2 + B))_\star$ is of type $N^3_2$, 
since the 3-space group $\langle t_1, t_2, e_3/2 + B\rangle$ has IT number 9 according to Table 1B of \cite{B-Z}. 
\end{proof}

Let $\alpha = a + A$ and $\beta = b+B$ be an affinities of $E^3$ such that $\alpha$ and $\beta$ normalize $\Mu$, and
$\alpha_\star$ and $\beta_\star$ are order 2 affinities of $O^3_1$ which do not fix a point of $O^3_1$, and $\Omega(\alpha_\star\beta_\star)$ has finite order. 
Then $A$ and $B$ have order 1 or 2.  If $A = I$, then $\alpha_\star$ is conjugate to $(e_1/2+I)_\star$, and  
if $A$ has order 2, then $\alpha_\star$ is conjugate to one of $(e_2/2+\mathrm{diag}(-1,1,-1))_\star$, 
$(e_1/2+\mathrm{diag}(1,-1,1))_\star$, or $(e_3/2+C)_\star$ with $Ce_1 = e_2, Ce_2 = e_1, Ce_3 = e_3$ by Lemma 8.
The same is true for $B$. 
Now $\langle A, B\rangle$ is a finite subgroup of $\mathrm{GL}(3,\integers)$ 
which either has order 1 or is isomorphic to $D_m$ for some $m$. 
By considering all $\integers$-classes of groups isomorphic to $D_m$ for some $m$ in Table 1B of \cite{B-Z}, 
we can find all the $\integers$-classes of the possible groups $\langle A, B\rangle$. 
To distinguish the conjugacy classes of pairs $\{\alpha_\star,\beta_\star\}$ with the same group $\langle A, B\rangle$, 
it is necessary to consider the action of the normalizer of $\langle A, B\rangle$ on the set of possible pairs $\{\alpha, \beta\}$. 
The normalizer of $\langle A, B\rangle$ is given in Table 5B of \cite{B-Z}. 

We find that there are 43 equivalence classes of pairs $\{\alpha_\star, \beta_\star\}$, 
and so there are 43 affine equivalence classes of flat $O^3_1\,|\, \Iota$ fiberings by Theorem 10. 
These fiberings are represented by the generalized Calabi constructions described in Tables 12, 13, 14 as explained in \S 8. 

Rows 41, 42, and 43 of Table 14 correspond to the three free products with amalgamation over $\integers^3$ described by Hillman on page 37 of \cite{H}.

\section{Fibrations of Closed Flat 4-Manifolds with Generic Fiber $O^3_2$} 

In this section, we describe the affine classification of the geometric co-Seifert fibrations of closed flat 4-manifolds with generic fiber 
the flat 3-manifold $O^3_2$.  Here $O^3_2 = E^3/\Mu$ with $\Mu = \langle t_1,t_2,t_3,\alpha\rangle$ and $\alpha = e_3/2 + \mathrm{diag}(-1,-1,1)$.

\begin{lemma}  
Let $\Mu = \langle t_1, t_2, t_3, \alpha\rangle$ with $\alpha = e_3/2 +  \mathrm{diag}(-1,-1,1)$.  
Then we have
$$N_A(\Mu) = \{b+B: b\in E^3, 2b_1, 2b_2 \in \integers\ and\ B = \mathrm{diag}(C,\pm 1)\ with\ C\in\mathrm{GL}(2,\integers)\}. $$
Moreover
$$\mathrm{Out}(\Mu) \cong ((\integers/2\integers)^2 \rtimes \mathrm{PGL}(2,\integers)) \times \{\pm 1\},$$
where the action of $\mathrm{PGL}(2,\integers)$ on $ (\integers/2\integers)^2$ is induced by the action of $\mathrm{GL}(2,\integers)$ on $E^2$, 
and if $b+ B \in N_A(\Mu)$ and $B = \mathrm{diag}(C,\delta)$, then $\Omega((b+B)_\star)$ 
corresponds to $(\epsilon(b_1), \epsilon(b_2), \pm C, \delta)$ with $\epsilon(b_i) = 0$ if $b_i \in \integers$ and $\epsilon(b_i) = 1$ if $b_i\not\in\integers$. 
\end{lemma}

The computation of the normalizer $N_A(\Mu)$ is elementary and is left to the reader. 
The outer automorphism group of $\Mu$ was computed by Hillman in \cite{H} for all the closed flat 3-manifold groups $\Mu$ 
by first computing the automorphism group of $\Mu$. 
We checked Hillman's computation of $\mathrm{Out}(\Mu)$ by computing $\mathrm{Out}(\Mu)$ from $N_A(\Mu)$ 
via the epimorphisms $\Phi: N_A(\Mu) \to \mathrm{Aff}(\Mu)$ and $\Omega: \mathrm{Aff}(\Mu) \to \mathrm{Out}(\Mu)$. 

First we consider $O^3_2\,|\, S^1$ fiberings. 

\begin{lemma}  
Every finite subgroup of $\mathrm{PGL}(2,\integers)$ is conjugate to a subgroup of one of the following dihedral groups 
$$\Big\langle \pm\left(\begin{array}{rr} 0 & -1 \\  1 & 0 \end{array}\right),  \pm\left(\begin{array}{rr} 0 & 1 \\  1 & 0 \end{array}\right)\Big\rangle, \ 
\Big\langle \pm\left(\begin{array}{rr} 0 & -1 \\  1 & 1 \end{array}\right),  \pm\left(\begin{array}{rr} 0 & 1 \\  1 & 0 \end{array}\right)\Big\rangle$$
of orders $4$ and $6$, respectively, 
The group $\mathrm{PGL}(2,\integers)$ has $5$ conjugacy classes of elements of finite order represented by 
$$\pm\left(\begin{array}{rr} 1 & 0 \\  0 & 1 \end{array}\right), \pm\left(\begin{array}{rr} 0 & -1 \\  1 & 0 \end{array}\right), 
\pm\left(\begin{array}{rr} 0 & 1 \\  1 & 0 \end{array}\right),\pm\left(\begin{array}{rr} -1 & 0 \\  0 & 1 \end{array}\right),\pm\left(\begin{array}{rr} 0 & -1 \\  1 & 1\end{array}\right)$$
of orders $1, 2, 2, 2, 3$, respectively. 
\end{lemma}
\begin{proof}
By Lemma 45 of \cite{R-TII}, the group $\mathrm{PGL}(2,\integers)$ is the free product with amalgamation 
of the two dihedral groups over their intersection. 
\end{proof}

There are 2 conjugacy classes of elements of $(\integers/2\integers)^2 \rtimes \mathrm{PGL}(2,\integers)$ corresponding to $\pm I$ 
represented by $\pm I$ and $e_2/2\pm I$.  Likewise for the 2nd and 3rd cases in Lemma 10. 
There are 3 conjugacy classes of elements of $(\integers/2\integers)^2 \rtimes \mathrm{PGL}(2,\integers)$ corresponding to $\pm\mathrm{diag}(-1,1)$ 
represented by $\pm\mathrm{diag}(-1,1)$, $e_2/2 \pm\mathrm{diag}(-1,1)$,  and $e_1/2+e_2/2 \pm\mathrm{diag}(-1,1)$. 
There is 1 conjugacy class of elements of $(\integers/2\integers)^2 \rtimes \mathrm{PGL}(2,\integers)$ corresponding to the last case in Lemma 10. 
Thus $(\integers/2\integers)^2 \rtimes \mathrm{PGL}(2,\integers)$ has 10 conjugacy classes of elements of finite order. 
Hence $\mathrm{Out}(\Mu)$ has 20 conjugacy classes of inverse pairs of elements of finite order. 
By Theorem 7, there are 20 affine equivalence classes of flat $O^3_2\,|\, S^1$ fiberings. 
These fiberings are represented by the generalized Calabi constructions described in Tables 15 and 16 as explained in \S 8. 

We next consider an example to illustrate how to use our tables to identify a closed flat 4-manifold. 

\vspace{.15in}
\noindent{\bf Example 3.}
Consider the torsion-free 4-space group 4/2/1/16 in \cite{B-Z}. 
The group is defined by 
$$\Gamma = \langle t_1, t_2, t_3, t_4, \alpha, \beta\rangle$$
with $\alpha = e_2/2 +e_4/2+ \mathrm{diag}(1,1,-1,1)$, and $\beta = e_3/2+e_4/2+\mathrm{diag}(-1,-1,-1,1)$. 

Let $\Nu = \langle t_1, t_2, t_3, \beta\alpha^{-1}\rangle$ with 
$$\beta\alpha^{-1} = e_2/2+e_3/2 + \mathrm{diag}(-1,-1,1,1).$$
Then $\Nu$ is a complete normal subgroup of $\Gamma$.  
By conjugating $\Nu$ by $-e_2/4+I$, we see that $\Nu$ is conjugate to 
$$\langle t_1,t_2,t_3, e_3/2+\mathrm{diag}(-1,-1,1,1)\rangle,$$
and so $\Nu$ is of type $O^3_2$. 
Let $V = \mathrm{Span}\{e_1, e_2, e_3\}$.  Then $V/\Nu$ is a closed flat 3-manifold of type $O^3_2$. 
The group $\Gamma/\Nu$ is infinite cyclic,  
since $\Nu\alpha$ generates $\Gamma/\Nu$ and $\Nu\alpha$ acts on the line $V^\perp$ by the translation $e_4/2+I$. 
Moreover $\Nu\alpha$ acts on $V/\Nu$ by $(e_2/2+\mathrm{diag}(1,1,-1))_\star$ in the first three coordinates. 
Therefore the geometric fibration of $E^4/\Gamma$ determined by $\Nu$ is affinely equivalent to the $O^3_2\,|\,S^1$ fibering described in Row 19 of Table 16 
by Theorem 7.  Hence $E^4/\Gamma$ is of type $N^4_{23}$.

\begin{table} 

\begin{tabular}{c|c|c|c|c}	
no. & mfd. & cS-fbr.  &  grp. & representative      \\ \hline	
 1	&$O^4_2$ & $O^3_2\,|\,S^1$  & $C_1$ & 
	  	$ \left(\begin{array}{rrr} 
		1 &  0 & 0 \\ 0 & 1 & 0 \\ 0 & 0 & 1 \end{array}\right)$   \\ \hline
 2	&$O^4_3$ & $O^3_2\,|\,S^1$ &  $C_2$ &
		$ \left(\begin{array}{rrrr} 
		0 & 1 &  0 & 0 \\ \frac{1}{2} & 0 & 1 & 0 \\ 0 & 0 & 0 & 1\end{array}\right)$    \\ \hline
 3	&$O^4_6$ & $O^3_2\,|\,S^1$ &  $C_2$ &
 		$ \left(\begin{array}{rrrr} 
		0& 0 &  -1 & 0\\ 0& 1 & 0 & 0 \\ \frac{1}{4} & 0 & 0 & 1 \end{array}\right)$   \\ \hline
 4	&$O^4_7$ & $O^3_2\,|\,S^1$ &  $C_4$ &
 		$ \left(\begin{array}{rrrr} 
		0 & 0 &  -1 & 0 \\ \frac{1}{2} & 1 & 0 & 0 \\ 0 & 0 & 0 & 1 \end{array}\right)$    \\ \hline
 5	&$O^4_8$ & $O^3_2\,|\,S^1$ &  $C_3$ &
 		$ \left(\begin{array}{rrrr} 
		0 & 0 &  -1 & 0 \\ 0 & 1 & 1 & 0 \\ \frac{1}{6} & 0 & 0 & 1 \end{array}\right)$     \\ \hline
 6 	&$O^4_9$ & $O^3_2\,|\,S^1$ &  $C_2$ &
 	$ \left(\begin{array}{rrr} 
		-1 & 0 &  0 \\ 0 & 1 & 0 \\ 0 & 0 & -1 \end{array}\right)$   \\ \hline
 7 	&$O^4_{10}$ & $O^3_2\,|\,S^1$ &  $C_2$ &
 	$ \left(\begin{array}{rrrr} 
		0 & -1 &  0 & 0\\ \frac{1}{2} & 0 & 1& 0 \\ 0 & 0 & 0 & -1 \end{array}\right)$    \\ \hline
8 	&$O^4_{11}$ & $O^3_2\,|\,S^1$ &  $C_2$ &
 	$ \left(\begin{array}{rrrr} 
		 \frac{1}{2} & -1 &  0 & 0\\ \frac{1}{2} & 0 & 1& 0 \\ 0 & 0 & 0 & -1 \end{array}\right)$    \\ \hline
9 	&$O^4_{12}$ & $O^3_2\,|\,S^1$ &  $C_2$ &
 	$ \left(\begin{array}{rrr} 
		0 & 1 &  0 \\ 1 & 0 & 0 \\ 0 & 0 & -1 \end{array}\right)$   \\ \hline
10 	&$O^4_{13}$ & $O^3_2\,|\,S^1$ &  $C_4$ &
 	$ \left(\begin{array}{rrrr} 
		0 & 0 &  1 & 0\\ \frac{1}{2} & 1 & 0 & 0 \\ 0 & 0 & 0 & -1 \end{array}\right)$    \\ \hline
\end{tabular}

\vspace{.2in}
\caption{The orientable flat $O^3_2$ fiberings over $S^1$}
\end{table}

\begin{table} 

\begin{tabular}{c|c|c|c|c}	
no. & mfd. & cS-fbr.  &  grp. & representative      \\ \hline	
 11	&$N^4_3$ & $O^3_2\,|\,S^1$  & $C_2$ & 
	  	$ \left(\begin{array}{rrr} 
		-1 &  0 & 0 \\ 0 & 1 & 0 \\ 0 & 0 & 1 \end{array}\right)$   \\ \hline
 12	&$N^4_4$ & $O^3_2\,|\,S^1$  & $C_2$ & 
	  	$ \left(\begin{array}{rrrr} 
		0 & -1 &  0 & 0  \\ \frac{1}{2} & 0 & 1 & 0 \\ 0 & 0 & 0 & 1 \end{array}\right)$   \\ \hline
 13	&$N^4_5$ & $O^3_2\,|\,S^1$  & $C_2$ & 
	  	$ \left(\begin{array}{rrrr} 
		\frac{1}{2} & -1 &  0 & 0  \\ \frac{1}{2} & 0 & 1 & 0 \\ 0 & 0 & 0 & 1 \end{array}\right)$   \\ \hline
 14 	&$N^4_6$ & $O^3_2\,|\,S^1$ &  $C_2$ &
 	$ \left(\begin{array}{rrr} 
		0 & 1 &  0 \\ 1 & 0 & 0 \\ 0 & 0 & 1 \end{array}\right)$   \\ \hline
 15 	&$N^4_7$ & $O^3_2\,|\,S^1$ &  $C_4$ &
 	$ \left(\begin{array}{rrrr} 
		0 & 0 &  1 & 0\\ \frac{1}{2} & 1 & 0 & 0 \\ 0 & 0 & 0 & 1 \end{array}\right)$    \\ \hline
16	&$N^4_{17}$ & $O^3_2\,|\,S^1$ &  $C_4$ &
 		$ \left(\begin{array}{rrr} 
		0 &  -1 & 0\\ 1 & 0 & 0 \\ 0 & 0 & -1 \end{array}\right)$   \\ \hline
 17	&$N^4_{18}$ & $O^3_2\,|\,S^1$ &  $C_4$ &
 		$ \left(\begin{array}{rrrr} 
		0 & 0 &  -1 & 0 \\ \frac{1}{2} & 1 & 0 & 0 \\ 0 & 0 & 0 & -1 \end{array}\right)$    \\ \hline
18	&$N^4_{22}$ & $O^3_2\,|\,S^1$  & $C_2$ & 
	  	$ \left(\begin{array}{rrr} 
		1 &  0 & 0 \\ 0 & 1 & 0 \\ 0 & 0 & -1 \end{array}\right)$   \\ \hline
19	&$N^4_{23}$ & $O^3_2\,|\,S^1$  & $C_2$ & 
	  	$ \left(\begin{array}{rrrr} 
		0 & 1 &  0 & 0  \\ \frac{1}{2} & 0 & 1 & 0 \\ 0 & 0 & 0 & -1 \end{array}\right)$   \\ \hline
20	&$N^4_{42}$ & $O^3_2\,|\,S^1$ &  $C_6$ &
 		$ \left(\begin{array}{rrr} 
		0 &  -1 & 0 \\ 1 & 1 & 0 \\ 0 & 0 & -1 \end{array}\right)$     \\ \hline
\end{tabular}

\vspace{.2in}
\caption{The nonorientable flat $O^3_2$ fiberings over $S^1$}
\end{table}

\vspace{.15in}
Next we consider $O^3_2\,|\, \Iota$ fiberings. 

\begin{lemma}  
Let $\Mu = \langle t_1, t_2, t_3, \alpha\rangle$ with $\alpha = e_3/2 + \mathrm{diag}(-1,-1,1)$,  
and let $\beta = b + \mathrm{diag}(C,\pm 1)$ be an affinity of $E^3$ that normalizes $\Mu$ 
such that $\beta_\star$ has order $2$ and $\beta_\star$ does not fix a point of $E^3/\Mu$. 
Then $\pm C$ has order $1, 2$ in $\mathrm{PGL}(2,\integers)$. 
If $\pm C$ has order $1$, then $\beta_\star$ is conjugate to $(e_2/2+I)_\star$. 
The quotient of $E^3/\Mu$ by the action of  $(e_2/2+I)_\star$ is of type $O^3_2$. 
If $\pm C$ has order $2$, then $\beta_\star$ is conjugate to either $(e_3/4 +B)_\star$, with $Be_1 = e_2, Be_2 = -e_1$ and $Be_3 = e_3$, 
or $(e_1/2+e_2/2+\mathrm{diag}(-1,1,-1))_\star$, or $(e_2/2+\mathrm{diag}(-1,1,1))_\star$, or $(e_1/2+e_2/2+\mathrm{diag}(-1,1,1))_\star$. 
The quotient of $E^3/\Mu$ by the action of $\beta_\star$ is of type $O^3_4$, $O^3_6$, $N^3_3$, or $N^3_4$, respectively. 
\end{lemma} 
\begin{proof}
If $\pm C$ has order $1$, then $\beta_\star$ is conjugate to $(e_2/2+I)_\star$ by Lemmas 3 and 8. 
The quotient of $E^3/\Mu$ by the action of  $(e_2/2+I)_\star$ is clearly of type $O^3_2$. 

If $\pm C$ has order $2$, then $\beta_\star$ is conjugate to either $(e_3/4 +B)_\star$, with $Be_1 = e_2, Be_2 = -e_1$ and $Be_3 = e_3$, 
or $(e_1/2+e_2/2+\mathrm{diag}(-1,1,-1))_\star$, or $(e_2/2+\mathrm{diag}(-1,1,1))_\star$, or $(e_1/2+e_2/2+\mathrm{diag}(-1,1,1))_\star$ 
by Lemmas 3, 8, and 10. 

The quotient of $E^3/\Mu$ by the action of $(e_3/4 +B)_\star$ is of type $O^3_4$, 
since the 3-space group $\langle t_1, t_2, \alpha, e_3/2 +B\rangle$ has IT number 76 according to Table 1B of \cite{B-Z}. 
The quotient of $E^3/\Mu$ by the action of $(e_1/2+e_2/2+\mathrm{diag}(-1,1,-1))_\star$ is of type $O^3_6$, 
since the 3-space group $\langle t_1, \alpha, e_1/2+e_2/2+\mathrm{diag}(-1,1,-1)\rangle$ has IT number 19 according to Table 1B of \cite{B-Z}. 
The quotient of $E^3/\Mu$ by the action of $(e_2/2+\mathrm{diag}(-1,1,1))_\star$ is of type $N^3_3$, 
since the 3-space group $\langle t_1, \alpha, e_2/2+\mathrm{diag}(-1,1,1)\rangle$ has IT number 29. 
The quotient of $E^3/\Mu$ by the action of $(e_1/2+e_2/2+\mathrm{diag}(-1,1,1))_\star$ is of type $N^3_4$, 
since the 3-space group $\langle t_1, \alpha, e_1/2+e_2/2+\mathrm{diag}(-1,1,1)\rangle$ has IT number 33. 
\end{proof}

Let $\beta = b+\mathrm{diag}(B, \pm 1)$ and $\gamma = c +\mathrm{diag}(C,\pm 1)$ be affinities of $E^3$ that normalize $\Mu$ and such that 
$\beta_\star$ and $\gamma_\star$ are order 2 affinities of $O^3_2$ which do not fix a point of $O^3_2$, and $\Omega(\beta_\star\gamma_\star)$ has finite order. 
By Lemmas 9 and 11, the elements $\Omega(\beta_\star)$ and $\Omega(\gamma_\star)$ have order 2. 
Therefore $\Omega(\langle \beta_\star,\gamma_\star\rangle)$ is a dihedral group of order $2m$ for some positive integer $m$. 
If $\langle \pm B, \pm C\rangle$ has order greater than 2, then $\langle \pm B, \pm C\rangle$ has order 4 
and the pair $\{\pm B, \pm C\}$ is conjugate to the pair $\{\pm D, \pm \mathrm{diag}(-1,1)\}$, 
where $De_1 = e_2$ and $De_2 = -e_1$, by Lemmas 9, 10, 11.  
Hence we have that  $m = 1, 2$ or $4$. 
Therefore by conjugating the pair $\{\beta_\star, \gamma_\star\}$, we may assume $B = I, D$ or $\mathrm{diag}(-1,1)$ 
and $C = \pm I, \pm D$ or $\pm \mathrm{diag}(-1,1)$. 
By considering all the possibilities for the conjugacy classes of $\beta_\star$ and $\gamma_\star$, 
we can find that there are 28 equivalence classes of pairs $\{\beta_\star,\gamma_\star\}$, and so there are 28 
affine equivalence classes of flat $O^3_2\,|\, \Iota$ fiberings by Theorem 10.  
These fiberings are represented by the generalized Calabi constructions described in Tables 17 and 18 as explained in \S 8.  

Rows 23, 25, 26, 27, and 28 of Table 18 correspond to the five amalgamated products described by Hillman on page 38 of \cite{H}. 
The amalgamated product corresponding to Row 24 of Table 18 was missed by Hillman in \cite{H}; see Example 4. 

\begin{table} 

\begin{tabular}{c|c|c|c|c|c}	
no. & mfd. & cS-fbr.  &  grp. & pair representatives  & s-fbrs.  \\ \hline	
 1	&$O^4_9$ & $O^3_2\,|\,\Iota$ & $D_1$ &
 		$ \left(\begin{array}{rrrr} 
		0 & -1 & 0 &  0 \\ \frac{1}{2} & 0 & 1 & 0 \\ 0 & 0 & 0 & 1\end{array}\right)$, 
		$ \left(\begin{array}{rrrr} 
		0 & -1 & 0 &  0 \\ \frac{1}{2} & 0 & 1 & 0 \\ 0 & 0 & 0 & 1\end{array}\right)$ & $N^3_3, N^3_3$ \\ \hline
2	&$O^4_{10}$ & $O^3_2\,|\,\Iota$ & $D_1$ &
 		$ \left(\begin{array}{rrrr} 
		\frac{1}{2}  & -1 & 0 &  0 \\ \frac{1}{2} & 0 & 1 & 0 \\ 0 & 0 & 0 & 1\end{array}\right)$, 
		$ \left(\begin{array}{rrrr} 
		\frac{1}{2}  & -1 & 0 &  0 \\ \frac{1}{2} & 0 & 1 & 0 \\ 0 & 0 & 0 & 1\end{array}\right)$ & $N^3_4, N^3_4$ \\ \hline
3	&$O^4_{12}$ & $O^3_2\,|\,\Iota$ & $D_2$ &
 		$ \left(\begin{array}{rrrr} 
		0  & -1 & 0 &  0 \\ \frac{1}{2} & 0 & 1 & 0 \\ 0 & 0 & 0 & 1\end{array}\right)$, 
		$ \left(\begin{array}{rrrr} 
		\frac{1}{2}  & -1 & 0 &  0 \\ \frac{1}{2} & 0 & 1 & 0 \\ 0 & 0 & 0 & 1\end{array}\right)$ & $N^3_3, N^3_4$ \\ \hline
4	&$O^4_{15}$ & $O^3_2\,|\,\Iota$ & $D_2$ &
 		$ \left(\begin{array}{rrrr} 
		0  & -1 & 0 &  0 \\ \frac{1}{2} & 0 & 1 & 0 \\ 0 & 0 & 0 & 1\end{array}\right)$, 
		$ \left(\begin{array}{rrrr} 
		\frac{1}{2}  & 1 & 0 &  0 \\ 0 & 0 & -1 & 0 \\ 0 & 0 & 0 & 1\end{array}\right)$ & $N^3_3, N^3_3$ \\ \hline
5	&$O^4_{16}$ & $O^3_2\,|\,\Iota$ & $D_2$ &
 		$ \left(\begin{array}{rrrr} 
		0  & -1 & 0 &  0 \\ \frac{1}{2} & 0 & 1 & 0 \\ 0 & 0 & 0 & 1\end{array}\right)$, 
		$ \left(\begin{array}{rrrr} 
		\frac{1}{2}  & 1 & 0 &  0 \\ \frac{1}{2} & 0 & -1 & 0 \\ 0 & 0 & 0 & 1\end{array}\right)$ & $N^3_3, N^3_4$ \\ \hline
6	&$O^4_{17}$ & $O^3_2\,|\,\Iota$ & $D_2$ &
 		$ \left(\begin{array}{rrrr} 
		\frac{1}{2}  & -1 & 0 &  0 \\ \frac{1}{2} & 0 & 1 & 0 \\ 0 & 0 & 0 & 1\end{array}\right)$, 
		$ \left(\begin{array}{rrrr} 
		\frac{1}{2}  & 1 & 0 &  0 \\ \frac{1}{2} & 0 & -1 & 0 \\ 0 & 0 & 0 & 1\end{array}\right)$ & $N^3_4, N^3_4$ \\ \hline
7	&$N^4_{15}$ & $O^3_2\,|\,\Iota$ & $D_1$ &
 		$ \left(\begin{array}{rrrr} 
		0 & 0 & -1 &  0 \\ 0 & 1 & 0 & 0 \\ \frac{1}{4} & 0 & 0 & 1\end{array}\right)$, 
		$ \left(\begin{array}{rrrr} 
		0 & 0 & -1 &  0 \\ 0 & 1 & 0 & 0 \\ \frac{1}{4} & 0 & 0 & 1\end{array}\right)$ & $O^3_4, O^3_4$ \\ \hline
8	&$N^4_{16}$ & $O^3_2\,|\,\Iota$ & $D_2$ &
 		$ \left(\begin{array}{rrrr} 
		0 & 0 & -1 &  0 \\ 0 & 1 & 0 & 0 \\ \frac{1}{4} & 0 & 0 & 1\end{array}\right)$, 
		$ \left(\begin{array}{rrrr} 
		\frac{1}{2} & 0 & -1 &  0 \\ \frac{1}{2} & 1 & 0 & 0 \\ \frac{1}{4} & 0 & 0 & 1\end{array}\right)$ & $O^3_4, O^3_4$ \\ \hline
9	&$N^4_{17}$ & $O^3_2\,|\,\Iota$ & $D_2$ &
 		$ \left(\begin{array}{rrrr} 
		0 & 0 & -1 &  0 \\ 0 & 1 & 0 & 0 \\ \frac{1}{4} & 0 & 0 & 1\end{array}\right)$, 
		$ \left(\begin{array}{rrrr} 
		0 & 0 & 1 &  0 \\ 0 & -1 & 0 & 0 \\ \frac{1}{4} & 0 & 0 & 1\end{array}\right)$ & $O^3_4, O^3_4$ \\ \hline
10	&$N^4_{18}$ & $O^3_2\,|\,\Iota$ & $D_2$ &
 		$ \left(\begin{array}{rrrr} 
		0 & 0 & -1 &  0 \\ 0 & 1 & 0 & 0 \\ \frac{1}{4} & 0 & 0 & 1\end{array}\right)$, 
		$ \left(\begin{array}{rrrr} 
		\frac{1}{2}  & 0 & 1 &  0 \\ \frac{1}{2} & -1 & 0 & 0 \\ \frac{1}{4} & 0 & 0 & 1\end{array}\right)$ & $O^3_4, O^3_4$ \\ \hline
11	&$N^4_{24}$ & $O^3_2\,|\,\Iota$ & $D_1$ &
 		$ \left(\begin{array}{rrrr} 
		0 & 1 & 0 &  0 \\ \frac{1}{2} & 0 & 1 & 0 \\ 0 & 0 & 0 & 1\end{array}\right)$, 
		$ \left(\begin{array}{rrrr} 
		0 & 1 & 0 &  0 \\ \frac{1}{2} & 0 & 1 & 0 \\ 0 & 0 & 0 & 1\end{array}\right)$ & $O^3_2, O^3_2$ \\ \hline
12	&$N^4_{26}$ & $O^3_2\,|\,\Iota$ & $D_2$ &
 		$ \left(\begin{array}{rrrr} 
		\frac{1}{2} & 1 & 0 &  0 \\ 0 & 0 & 1 & 0 \\ 0 & 0 & 0 & 1\end{array}\right)$, 
		$ \left(\begin{array}{rrrr} 
		0 & 1 & 0 &  0 \\ \frac{1}{2} & 0 & 1 & 0 \\ 0 & 0 & 0 & 1\end{array}\right)$ & $O^3_2, O^3_2$ \\ \hline
13	&$N^4_{28}$ & $O^3_2\,|\,\Iota$ & $D_2$ &
 		$ \left(\begin{array}{rrrr} 
		\frac{1}{2} & 1 & 0 &  0 \\ \frac{1}{2} & 0 & 1 & 0 \\ 0 & 0 & 0 & 1\end{array}\right)$, 
		$ \left(\begin{array}{rrrr} 
		0 & 0 & -1 &  0 \\ 0 & 1 & 0 & 0 \\ \frac{1}{4} & 0 & 0 & 1\end{array}\right)$ & $O^3_2, O^3_4$ \\ \hline
14	&$N^4_{29}$ & $O^3_2\,|\,\Iota$ & $D_4$ &
 		$ \left(\begin{array}{rrrr} 
		0 & 1 & 0 &  0 \\ \frac{1}{2} & 0 & 1 & 0 \\ 0 & 0 & 0 & 1\end{array}\right)$, 
		$ \left(\begin{array}{rrrr} 
		0 & 0 & -1 &  0 \\ 0 & 1 & 0 & 0 \\ \frac{1}{4} & 0 & 0 & 1\end{array}\right)$ & $O^3_2, O^3_4$ \\ \hline \end{tabular}

\vspace{.2in}
\caption{The flat $O^3_2$ fiberings over $\Iota$, Part I}
\end{table}

\begin{table} 

\begin{tabular}{c|c|c|c|c|c}	
no. & mfd. & cS-fbr.  &  grp. & pair representatives  & s-fbrs.  \\ \hline	
15	&$N^4_{30}$ & $O^3_2\,|\,\Iota$ & $D_2$ &
 		$ \left(\begin{array}{rrrr} 
		0 & 1 & 0 &  0 \\ \frac{1}{2} & 0 & 1 & 0 \\ 0 & 0 & 0 & 1\end{array}\right)$, 
		$ \left(\begin{array}{rrrr} 
		0 & -1& 0 &  0 \\ \frac{1}{2} & 0 & 1 & 0 \\ 0 & 0 & 0 & 1\end{array}\right)$ & $O^3_2, N^3_3$ \\ \hline
16	&$N^4_{32}$ & $O^3_2\,|\,\Iota$ & $D_2$ &
 		$ \left(\begin{array}{rrrr} 
		0 & 1 & 0 &  0 \\ \frac{1}{2} & 0 & 1 & 0 \\ 0 & 0 & 0 & 1\end{array}\right)$, 
		$ \left(\begin{array}{rrrr} 
		\frac{1}{2} & -1& 0 &  0 \\ \frac{1}{2} & 0 & 1 & 0 \\ 0 & 0 & 0 & 1\end{array}\right)$ & $O^3_2, N^3_4$ \\ \hline
17	&$N^4_{34}$ & $O^3_2\,|\,\Iota$ & $D_2$ &
 		$ \left(\begin{array}{rrrr} 
		 \frac{1}{2} & 1 & 0 &  0 \\ 0 & 0 & 1 & 0 \\ 0 & 0 & 0 & 1\end{array}\right)$, 
		$ \left(\begin{array}{rrrr} 
		0 & -1& 0 &  0 \\ \frac{1}{2} & 0 & 1 & 0 \\ 0 & 0 & 0 & 1\end{array}\right)$ & $O^3_2, N^3_3$ \\ \hline
18	&$N^4_{34}$ & $O^3_2\,|\,\Iota$ & $D_2$ &
 		$ \left(\begin{array}{rrrr} 
		\frac{1}{2} & 1 & 0 &  0 \\ \frac{1}{2} & 0 & 1 & 0 \\ 0 & 0 & 0 & 1\end{array}\right)$, 
		$ \left(\begin{array}{rrrr} 
		0 & -1& 0 &  0 \\ \frac{1}{2} & 0 & 1 & 0 \\ 0 & 0 & 0 & 1\end{array}\right)$ & $O^3_2, N^3_3$ \\ \hline
19	&$N^4_{35}$ & $O^3_2\,|\,\Iota$ & $D_2$ &
 		$ \left(\begin{array}{rrrr} 
		 \frac{1}{2} & 1 & 0 &  0 \\ 0 & 0 & 1 & 0 \\ 0 & 0 & 0 & 1\end{array}\right)$, 
		$ \left(\begin{array}{rrrr} 
		\frac{1}{2} & -1& 0 &  0 \\ \frac{1}{2} & 0 & 1 & 0 \\ 0 & 0 & 0 & 1\end{array}\right)$ & $O^3_2, N^3_4$ \\ \hline
20	&$N^4_{36}$ & $O^3_2\,|\,\Iota$ & $D_2$ &
 		$ \left(\begin{array}{rrrr} 
		\frac{1}{2} & 1 & 0 &  0 \\ \frac{1}{2} & 0 & 1 & 0 \\ 0 & 0 & 0 & 1\end{array}\right)$, 
		$ \left(\begin{array}{rrrr} 
		\frac{1}{2} & -1& 0 &  0 \\ \frac{1}{2} & 0 & 1 & 0 \\ 0 & 0 & 0 & 1\end{array}\right)$ & $O^3_2, N^3_4$ \\ \hline
21	&$N^4_{40}$ & $O^3_2\,|\,\Iota$ & $D_4$ &
 		$ \left(\begin{array}{rrrr} 
		0 & 0 & -1 &  0 \\ 0 & 1 & 0 & 0 \\ \frac{1}{4} & 0 & 0 & 1\end{array}\right)$, 
		$ \left(\begin{array}{rrrr} 
		0 & -1 & 0 &  0 \\ \frac{1}{2} & 0 & 1 & 0 \\ 0 & 0 & 0 & 1\end{array}\right)$ & $O^3_4, N^3_3$ \\ \hline
22	&$N^4_{41}$ & $O^3_2\,|\,\Iota$ & $D_4$ &
 		$ \left(\begin{array}{rrrr} 
		0 & 0 & -1 &  0 \\ 0 & 1 & 0 & 0 \\ \frac{1}{4} & 0 & 0 & 1\end{array}\right)$, 
		$ \left(\begin{array}{rrrr} 
		\frac{1}{2} & -1 & 0 &  0 \\ \frac{1}{2} & 0 & 1 & 0 \\ 0 & 0 & 0 & 1\end{array}\right)$ & $O^3_4, N^3_4$ \\ \hline
23	&$N^4_{44}$ & $O^3_2\,|\,\Iota$ & $D_1$ &
 		$ \left(\begin{array}{rrrr} 
		\frac{1}{2} & -1 & 0 &  0 \\ \frac{1}{2} & 0 & 1 & 0 \\ 0 & 0 & 0 & -1\end{array}\right)$, 
		$ \left(\begin{array}{rrrr} 
		\frac{1}{2} & -1& 0 &  0 \\ \frac{1}{2} & 0 & 1 & 0 \\ 0 & 0 & 0 & -1\end{array}\right)$ & $O^3_6, O^3_6$ \\ \hline
24	&$N^4_{45}$ & $O^3_2\,|\,\Iota$ & $D_2$ &
 		$ \left(\begin{array}{rrrr} 
		\frac{1}{2} & 1 & 0 &  0 \\ \frac{1}{2} & 0 & 1 & 0 \\ 0 & 0 & 0 & 1\end{array}\right)$, 
		$ \left(\begin{array}{rrrr} 
		\frac{1}{2} & -1& 0 &  0 \\ \frac{1}{2} & 0 & 1 & 0 \\ 0 & 0 & 0 & -1\end{array}\right)$ & $O^3_2, O^3_6$ \\ \hline
25	&$N^4_{45}$ & $O^3_2\,|\,\Iota$ & $D_2$ &
 		$ \left(\begin{array}{rrrr} 
		\frac{1}{2} & -1 & 0 &  0 \\ \frac{1}{2} & 0 & 1 & 0 \\ 0 & 0 & 0 & -1\end{array}\right)$, 
		$ \left(\begin{array}{rrrr} 
		0 & -1& 0 &  0 \\ \frac{1}{2} & 0 & 1 & 0 \\ 0 & 0 & 0 & 1\end{array}\right)$ & $O^3_6, N^3_3$ \\ \hline
26	&$N^4_{46}$ & $O^3_2\,|\,\Iota$ & $D_2$ &
 		$ \left(\begin{array}{rrrr} 
		0 & 1 & 0 &  0 \\ \frac{1}{2} & 0 & 1 & 0 \\ 0 & 0 & 0 & 1\end{array}\right)$, 
		$ \left(\begin{array}{rrrr} 
		\frac{1}{2} & -1& 0 &  0 \\ \frac{1}{2} & 0 & 1 & 0 \\ 0 & 0 & 0 & -1\end{array}\right)$ & $O^3_2, O^3_6$ \\ \hline
27	&$N^4_{46}$ & $O^3_2\,|\,\Iota$ & $D_2$ &
 		$ \left(\begin{array}{rrrr} 
		\frac{1}{2} & -1 & 0 &  0 \\ \frac{1}{2} & 0 & 1 & 0 \\ 0 & 0 & 0 & -1\end{array}\right)$, 
		$ \left(\begin{array}{rrrr} 
		\frac{1}{2} & 1& 0 &  0 \\ \frac{1}{2} & 0 & -1 & 0 \\ 0 & 0 & 0 & 1\end{array}\right)$ & $O^3_6, N^3_4$ \\ \hline
28	&$N^4_{47}$ & $O^3_2\,|\,\Iota$ & $D_2$ &
 		$ \left(\begin{array}{rrrr} 
		0 & 0 & -1 &  0 \\ 0 & 1 & 0 & 0 \\ \frac{1}{4} & 0 & 0 & 1\end{array}\right)$, 
		$ \left(\begin{array}{rrrr} 
		\frac{1}{2} & -1 & 0 &  0 \\ \frac{1}{2} & 0 & 1 & 0 \\ 0 & 0 & 0 & -1\end{array}\right)$ & $O^3_4, O^3_6$ \\ \hline \end{tabular}

\vspace{.2in}
\caption{The flat $O^3_2$ fiberings over $\Iota$, Part II}
\end{table}

\vspace{.15in}
\noindent{\bf Example 4.}
Consider the torsion-free 4-space group 6/2/1/50 in \cite{B-Z}. 
After transposing the 1st and 3rd coordinates, the group is defined by 
$$\Gamma = \langle t_1, t_2, t_3, t_4, \alpha, \beta, \gamma\rangle$$
with $\alpha = e_1/2 +e_2/2+ \mathrm{diag}(1,1,1,-1)$, and $\beta = e_2/2+e_4/2+\mathrm{diag}(-1,1,-1,-1)$, 
and $\gamma = e_3/2+e_4/2+\mathrm{diag}(1,-1,-1,1)$.  
We have that 
$$\alpha\beta^{-1}\gamma = e_1/2 -e_3/2+ \mathrm{diag}(-1,-1,1,1).$$ 

Let $\Nu = \langle t_1, t_2, t_3,\alpha\beta^{-1}\gamma\rangle$. Then $\Nu$ is a complete normal subgroup of $\Gamma$.  
By conjugating $\Nu$ by $-e_1/4+I$, we see that $\Nu$ is conjugate to 
$$\langle t_1,t_2,t_3, e_3/2+\mathrm{diag}(-1,-1,1,1)\rangle,$$
and so $\Nu$ is of type $O^3_2$. 
Let $V = \mathrm{Span}\{e_1, e_2, e_3\}$.  Then $V/\Nu$ is a closed flat 3-manifold of type $O^3_2$. 
The group $\Gamma/\Nu$ is infinite dihedral,  
since $\Nu\alpha$ and $\Nu\beta$ are order 2 generators of $\Gamma/\Nu$, and $\Nu\alpha$ acts on the line $V^\perp$ by the reflection $-I$ 
and $\Nu\beta$ acts on the line $V^\perp$ by the reflection $e_4/2-I$.

Now $\Nu\alpha$ acts on $V/\Nu$ by $(e_1/2+e_2/2 + \mathrm{diag}(1,1,1))_\star$ in the first three coordinates. 
Conjugating by $-e_1/4+I$ does not change the action of $\Nu\alpha$ on $V/\Nu$. 
Now $\Nu\beta$ acts on $V/\Nu$ by $(e_2/2+\mathrm{diag}(-1,1,-1))_\star$.   
Conjugating by $-e_1/4+I$ changes the action of $\Nu\beta$ on $V/\Nu$ to $(e_1/2+e_2/2+\mathrm{diag}(-1,1,-1))_\star$. 
Therefore the geometric fibration of $E^4/\Gamma$ determined by $\Nu$ is affinely equivalent to the $O^3_2\,|\,\Iota$ fibering described in Row 24 of Table 18 
by Theorem 10.  Hence $E^4/\Gamma$ is of type $N^4_{45}$.

\section{Fibrations of Closed Flat 4-Manifolds with Generic Fiber $O^3_3$} 

In this section, we describe the affine classification of the geometric co-Seifert fibrations of closed flat 4-manifolds with generic fiber 
the flat 3-manifold $O^3_3$.  Here $O^3_3 = E^3/\Mu$ with $\Mu = \langle t_1,t_2,t_3,\alpha\rangle$ with $\alpha = e_3/3 + A$ and 
$$A = \left(\begin{array}{rrr}
0 & -1 & 0 \\ 1 & -1 & 0 \\ 0 & 0 & 1 
\end{array}\right).$$  
 
\begin{lemma} 
Let $\Mu = \langle t_1, t_2, t_3, \alpha\rangle$ with $\alpha = e_3/3 + A$ and $A$ defined as above.  
Then we have
$$N_A(\Mu) = \{b+B: b\in E^3, 3b_1, 3b_2, b_1+b_2 \in \integers\ and\ B \in G\},$$
where $G$ is the dihedral group of order $12$ given by 
$$G = \Bigg\langle \left(\begin{array}{rrr}
0 & 1 & 0 \\ -1 & 1 & 0 \\ 0 & 0 & 1 
\end{array}\right), \ \left(\begin{array}{rrr}
0 & 1 & 0 \\ 1 & 0 & 0 \\ 0 & 0 & -1 
\end{array}\right)\Bigg\rangle.$$
Moreover, $\mathrm{Out}(\Mu) \cong (\integers/3\integers) \rtimes H$, with $H$ a dihedral group of order $4$,   
and $\Omega((e_1/3 + 2e_2/3+I)_\star)$ corresponding to a generator of $\integers/3\integers$, 
and $\Omega((\mathrm{diag}(-1,-1,1))_\star)$ and $\Omega(D_\star)$, with $D$ the second generator of $G$,  corresponding to generators of $H$. 
\end{lemma}
It follows from Lemma 12 that $\mathrm{Out}(\Mu)$ has order 12 and $\mathrm{Out}(\Mu)$ has 
6 conjugacy classes of inverse pairs of elements of finite order. 
By Theorem 7, there are 6 affine equivalence classes of flat $O^3_3\,|\, S^1$ fiberings. 
These fiberings are described in Table 19 as explained in \S 8. 

\begin{table} 

\begin{tabular}{c|c|c|c|c}	
no. & mfd. & cS-fbr.  &  grp. & representative      \\ \hline	
 1	&$O^4_4$ & $O^3_3\,|\,S^1$  & $C_1$ & 
	  	$ \left(\begin{array}{rrr} 
		1 &  0 & 0 \\ 0 & 1 & 0 \\ 0 & 0 & 1 \end{array}\right)$   \\ \hline
 2	&$O^4_5$ & $O^3_3\,|\,S^1$ &  $C_3$ &
		$ \left(\begin{array}{rrrr} 
		\frac{1}{3} & 1 &  0 & 0 \\ \frac{2}{3} & 0 & 1 & 0 \\ 0 & 0 & 0 & 1\end{array}\right)$    \\ \hline
 3	&$O^4_8$ & $O^3_3\,|\,S^1$ &  $C_2$ &
 		$ \left(\begin{array}{rrr} 
		-1 &  0 & 0\\ 0 & -1 & 0 \\ 0 & 0 & 1 \end{array}\right)$   \\ \hline
 4	&$O^4_{18}$ & $O^3_3\,|\,S^1$ &  $C_2$ &
 		$ \left(\begin{array}{rrr} 
		 0 &  1 & 0 \\ 1 & 0 & 0 \\  0 & 0 & -1 \end{array}\right)$    \\ \hline
 5	&$O^4_{19}$ & $O^3_3\,|\,S^1$ &  $C_2$ &
 		$ \left(\begin{array}{rrr} 
		0 &  -1 & 0 \\ -1 & 0 & 0 \\ 0 & 0 & -1 \end{array}\right)$     \\ \hline
 6 	&$O^4_{20}$ & $O^3_3\,|\,S^1$ &  $C_6$ &
 	$ \left(\begin{array}{rrrr} 
		\frac{1}{3} & 0 & -1 &  0 \\ \frac{2}{3} & -1 & 0 & 0 \\ 0 & 0 & 0 & -1 \end{array}\right)$   \\ \hline
 \end{tabular}

\vspace{.2in}
\caption{The flat $O^3_3$ fiberings over $S^1$}
\end{table}

\begin{table} 

\begin{tabular}{c|c|c|c|c|c}	
no. & mfd. & cS-fbr.  &  grp. & pair representatives  & s-fbrs.  \\ \hline	
 1	&$N^4_{19}$ & $O^3_3\,|\,\Iota$ & $D_1$ &
 		$ \left(\begin{array}{rrrr} 
		0 & 1 & 0 &  0 \\ 0 & 0 & 1 & 0 \\ \frac{1}{2} & 0 & 0 & 1\end{array}\right)$, 
		$ \left(\begin{array}{rrrr} 
		0 & 1 & 0 &  0 \\ 0 & 0 & 1 & 0 \\ \frac{1}{2} & 0 & 0 & 1\end{array}\right)$ & $O^3_3, O^3_3$ \\ \hline
 2	&$N^4_{20}$ & $O^3_3\,|\,\Iota$ & $D_1$ &
 		$ \left(\begin{array}{rrrr} 
		0 & -1 & 0 &  0 \\ 0 & 0 & -1 & 0 \\ \frac{1}{2} & 0 & 0 & 1\end{array}\right)$, 
		$ \left(\begin{array}{rrrr} 
		0 & -1 & 0 &  0 \\ 0 & 0 & -1 & 0 \\ \frac{1}{2} & 0 & 0 & 1\end{array}\right)$ & $O^3_5, O^3_5$ \\ \hline
3	&$N^4_{21}$ & $O^3_3\,|\,\Iota$ & $D_3$ &
 		$ \left(\begin{array}{rrrr} 
		0 & -1 & 0 &  0 \\ 0 & 0 & -1 & 0 \\ \frac{1}{2} & 0 & 0 & 1\end{array}\right)$, 
		$ \left(\begin{array}{rrrr} 
		\frac{1}{3} & -1 & 0 &  0 \\ \frac{2}{3} & 0 & -1 & 0 \\ \frac{1}{2} & 0 & 0 & 1\end{array}\right)$ & $O^3_5, O^3_5$ \\ \hline
4	&$N^4_{42}$ & $O^3_3\,|\,\Iota$ & $D_2$ &
 		$ \left(\begin{array}{rrrr} 
		0 & 1 & 0 &  0 \\ 0 & 0 & 1 & 0 \\ \frac{1}{2} & 0 & 0 & 1\end{array}\right)$, 
		$ \left(\begin{array}{rrrr} 
		0 & -1 & 0 &  0 \\ 0 & 0 & -1 & 0 \\ \frac{1}{2} & 0 & 0 & 1\end{array}\right)$ & $O^3_3, O^3_5$ \\ \hline
 \end{tabular}

\vspace{.2in}
\caption{The flat $O^3_3$ fiberings over $\Iota$}
\end{table}

Next we consider $O^3_3\,|\, \Iota$ fiberings. 

\begin{lemma}  
Let $\Mu = \langle t_1, t_2, t_3, \alpha\rangle$ with $\alpha = e_3/3 + A$ and $A$ defined as above,   
and let $\beta$ be an affinity of $E^3$ that normalizes $\Mu$ 
such that $\beta_\star$ has order $2$ and $\beta_\star$ does not fix a point of $E^3/\Mu$. 
Then $\Omega(\beta_\star)$ has order $1$ or $2$ in $\mathrm{Out}(\Mu)$. 
If $\Omega(\beta_\star)$ has order $1$, then $\beta_\star = (e_3/2+I)_\star$. 
The quotient of $E^3/\Mu$ by the action of  $(e_3/2+I)_\star$ is of type $O^3_3$. 
If $\Omega(\beta_\star)$ has order $2$, then $\beta_\star$ is conjugate to $(e_3/2+\mathrm{diag}(-1,-1,1))_\star$. 
The quotient of $E^3/\Mu$ by the action of $\beta_\star$ is of type $O^3_5$. 
\end{lemma} 
\begin{proof}
If $\Omega(\beta_\star)$ has order $1$, then $\beta_\star= (e_3/2+I)_\star$ by Lemmas 3, 8, and 12. 
The quotient of $E^3/\Mu$ by the action of  $(e_3/2+I)_\star$ is clearly of type $O^3_3$. 

If $\Omega(\beta_\star)$ has order $2$, then $\beta_\star$ is conjugate to $(e_3/2+\mathrm{diag}(-1,-1,1))_\star$ by Lemmas 3, 8, and 12. 
The quotient of $E^3/\Mu$ by the action of $(e_3/2+\mathrm{diag}(-1,-1,1))_\star$ is of type $O^3_5$, 
since the 3-space group $\langle t_1, t_2, \alpha, e_3/2+\mathrm{diag}(-1,-1,1)\rangle$ has IT number 169 according to Table 1B of \cite{B-Z}. 
\end{proof}

Let $\beta$ and $\gamma$ be affinities of $E^3$ that normalize $\Mu$ and such that 
$\beta_\star$ and $\gamma_\star$ are order 2 affinities of $O^3_3$ which do not fix a point of $O^3_3$, and $\Omega(\beta_\star\gamma_\star)$ has finite order. 
By Lemmas 12 and 13, there are only 4 conjugacy classes of pairs $\{\beta_\star, \gamma_\star\}$, 
and so there are only 4 affine equivalence classes of flat $O^3_3\,|\, \Iota$ fiberings by Theorem 10.  
These fiberings are described in Table 20 as explained in \S 8.

\section{Fibrations of Closed Flat 4-Manifolds with Generic Fiber $O^3_4$} 

In this section, we describe the affine classification of the geometric co-Seifert fibrations of closed flat 4-manifolds with generic fiber 
the flat 3-manifold $O^3_4$.  Here $O^3_4 = E^3/\Mu$ with $\Mu = \langle t_1,t_2,t_3,\alpha\rangle$ with $\alpha = e_3/4 + A$ and 
$$A = \left(\begin{array}{rrr}
0 & -1 & 0 \\ 1 & 0 & 0 \\ 0 & 0 & 1 
\end{array}\right).$$  
 
\begin{lemma} 
Let $\Mu = \langle t_1, t_2, t_3, \alpha\rangle$ with $\alpha = e_3/4 + A$ and $A$ defined as above.  
Then we have
$$N_A(\Mu) = \{b+B: b\in E^3, 2b_1, 2b_2, b_1+b_2 \in \integers\ and\ B \in G\},$$
where $G$ is the dihedral group of order $8$ generated by $A$ and $\mathrm{diag}(-1,1,-1)$.  
Moreover $\mathrm{Out}(\Mu)$ is a dihedral group of order $4$ 
with generators represented by $e_1/2+e_2/2+I$ and $\mathrm{diag}(-1,1,-1)$. 
\end{lemma}
It follows from Lemma 14 that $\mathrm{Out}(\Mu)$ has 
4 conjugacy classes of elements of finite order represented by the matrices in Table 21. 
Hence $\mathrm{Out}(\Mu)$ has 4 conjugacy classes of inverse pairs of elements of finite order. 
By Theorem 7, there are 4 affine equivalence classes of flat $O^3_4\,|\, S^1$ fiberings. 
These fiberings are described in Table 21 as explained in \S 8. 

Next we consider $O^3_4\,|\, \Iota$ fiberings. 

\begin{lemma}  
Let $\Mu = \langle t_1, t_2, t_3, \alpha\rangle$ with $\alpha = e_3/4 + A$ and $A$ defined as above,   
and let $\beta$ be an affinity of $E^3$ that normalizes $\Mu$ 
such that $\beta_\star$ has order $2$ and $\beta_\star$ does not fix a point of the flat 3-manifold $E^3/\Mu$. 
Then $\beta_\star = (e_1/2+e_2/2+I)_\star$. 
The quotient of $E^3/\Mu$ by the action of  $(e_1/2+e_2/2+I)_\star$ is of type $O^3_4$. 
\end{lemma} 
\begin{proof}
We have that $\beta_\star = (e_1/2+e_2/2+I)_\star$ by Lemmas 3, 8, and 14. 
The quotient of $E^3/\Mu$ by the action of  $(e_1/2+e_2/2+I)_\star$ is clearly of type $O^3_4$. 
\end{proof}

By Lemma 15, the isometry $(e_1/2+e_2/2+I)_\star$ is the unique order 2 affinity of $O^3_4$ that 
does not fix a point of $O^3_4$. 
By Theorem 10, there is exactly 1 affine equivalence class of  flat $O^3_4\,|\, \Iota$ fiberings, 
which is described in Table 22 as explained in \S 8. 

\begin{table} 

\begin{tabular}{c|c|c|c|c}	
no. & mfd. & cS-fbr.  &  grp. & representative      \\ \hline	
 1	&$O^4_6$ & $O^3_4\,|\,S^1$  & $C_1$ & 
	  	$ \left(\begin{array}{rrr} 
		1 &  0 & 0 \\ 0 & 1 & 0 \\ 0 & 0 & 1 \end{array}\right)$   \\ \hline
 2	&$O^4_7$ & $O^3_4\,|\,S^1$ &  $C_2$ &
		$ \left(\begin{array}{rrrr} 
		\frac{1}{2} & 1 &  0 & 0 \\ \frac{1}{2} & 0 & 1 & 0 \\ 0 & 0 & 0 & 1\end{array}\right)$    \\ \hline
 3	&$O^4_{21}$ & $O^3_4\,|\,S^1$ &  $C_2$ &
 		$ \left(\begin{array}{rrr} 
		-1 &  0 & 0\\ 0 & 1 & 0 \\ 0 & 0 & -1 \end{array}\right)$   \\ \hline
 4	&$O^4_{22}$ & $O^3_4\,|\,S^1$ &  $C_2$ &
 		$ \left(\begin{array}{rrrr} 
		\frac{1}{2} & -1 &  0 & 0 \\ \frac{1}{2} & 0 & 1 & 0 \\ 0 & 0 & 0 & -1 \end{array}\right)$    \\ \hline
 \end{tabular}

\vspace{.2in}
\caption{The flat $O^3_4$ fiberings over $S^1$}
\end{table}

\begin{table} 

\begin{tabular}{c|c|c|c|c|c}	
no. & mfd. & cS-fbr.  &  grp. & pair representatives  & s-fbrs.  \\ \hline	
 1	&$N^4_{27}$ & $O^3_4\,|\,\Iota$ & $D_1$ &
 		$ \left(\begin{array}{rrrr} 
		\frac{1}{2} & 1 & 0 &  0 \\ \frac{1}{2} & 0 & 1 & 0 \\ 0 & 0 & 0 & 1\end{array}\right)$, 
		$ \left(\begin{array}{rrrr} 
		\frac{1}{2} & 1 & 0 &  0 \\ \frac{1}{2} & 0 & 1 & 0 \\ 0 & 0 & 0 & 1\end{array}\right)$ & $O^3_4, O^3_4$ \\ \hline
 \end{tabular}

\vspace{.2in}
\caption{The flat $O^3_4$ fiberings over $\Iota$}
\end{table}

\section{Fibrations of Closed Flat 4-Manifolds with Generic Fiber $O^3_5$} 

In this section, we describe the affine classification of the geometric co-Seifert fibrations of closed flat 4-manifolds with generic fiber 
the flat 3-manifold $O^3_5$.  Here $O^3_5 = E^3/\Mu$ with $\Mu = \langle t_1,t_2,t_3,\alpha\rangle$ with $\alpha = e_3/6 + A$ and 
$$A = \left(\begin{array}{rrr}
1 & -1 & 0 \\ 1 & 0 & 0 \\ 0 & 0 & 1 
\end{array}\right).$$

\begin{lemma} 
Let $\Mu = \langle t_1, t_2, t_3, \alpha\rangle$ with $\alpha = e_3/6 + A$ and $A$ defined as above.  
Then we have
$$N_A(\Mu) = \{b+B: b\in E^3, b_1, b_2 \in \integers\ and\ B \in G\},$$
where $G$ is the dihedral group of order $12$ given by 
$$G = \Bigg\langle \left(\begin{array}{rrr}
1 & -1 & 0 \\ 1 & 0 & 0 \\ 0 & 0 & 1 
\end{array}\right), \ \left(\begin{array}{rrr}
0 & 1 & 0 \\ 1 & 0 & 0 \\ 0 & 0 & -1 
\end{array}\right)\Bigg\rangle.$$
Moreover $\mathrm{Out}(\Mu)$ has order $2$ with generator represented by the second generator of $G$. 
\end{lemma}
It follows from Lemma 12 that $\mathrm{Out}(\Mu)$ has 
2 conjugacy classes of elements of finite order represented by the matrices in Table 23. 
Hence $\mathrm{Out}(\Mu)$ has 2 conjugacy classes of inverse pairs of elements of finite order. 
By Theorem 7, there are 2 affine equivalence classes of flat $O^3_5\,|\, S^1$ fiberings. 
These fiberings are described in Table 23 as explained in \S 8. 

\begin{table} 

\begin{tabular}{c|c|c|c|c}	
no. & mfd. & cS-fbr.  &  grp. & representative      \\ \hline	
 1	&$O^4_8$ & $O^3_5\,|\,S^1$  & $C_1$ & 
	  	$ \left(\begin{array}{rrr} 
		1 &  0 & 0 \\ 0 & 1 & 0 \\ 0 & 0 & 1 \end{array}\right)$   \\ \hline
 2	&$O^4_{25}$ & $O^3_5\,|\,S^1$ &  $C_2$ &
		$ \left(\begin{array}{rrr} 
		 0 &  1 & 0 \\ 1 & 0 & 0 \\  0 & 0 & -1\end{array}\right)$    \\ \hline
 \end{tabular}

\vspace{.2in}
\caption{The flat $O^3_5$ fiberings over $S^1$}
\end{table}

There are no $O^3_5\,|\, \Iota$ fiberings, since by Lemmas 3, 8, and 16, there are no fix point free affinities of $E^3/\Mu$ of order 2.

\section{Fibrations of Closed Flat 4-Manifolds with Generic Fiber $O^3_6$} 

In this section, we describe the affine classification of the geometric co-Seifert fibrations of closed flat 4-manifolds with generic fiber 
the flat 3-manifold $O^3_6$.  Here $O^3_6 = E^3/\Mu$ with $\Mu = \langle t_1,t_2,t_3,\alpha, \beta\rangle$ with $\alpha = e_1/2 + e_3/2 + 
\mathrm{diag}(-1,-1,1)$, and $\beta = e_2/2 + \mathrm{diag}(-1,1,-1)$. 

\begin{lemma}  
Let $\Mu = \langle t_1, t_2, t_3, \alpha, \beta\rangle$ with $\alpha$ and $\beta$ defined as above.  
Then we have a short exact sequence of groups and homomorphisms
$$1 \to (\textstyle{\frac{1}{2}}\integers)^3 \to N_A(\Mu) \to \mathrm{O}(3)\cap \mathrm{GL}(3, \integers) \to 1$$
with  $c \in (\frac{1}{2}\integers)^3$ mapping to $c+I$, and  $c + C \in N_A(\Mu)$ mapping to $C$. 
The group $\mathrm{O}(3)\cap \mathrm{GL}(3, \integers)$ 
consists of the $48$ $(3\times 3)$-matrices in which each row and each column 
has exactly one nonzero entry and each nonzero entry is $\pm 1$.  

Furthermore, we have a short exact sequence of groups and homomorphisms
$$1 \to (\integers/2\integers)^3 \to \mathrm{Out}(\Mu) \to S_3 \times \{\pm I\} \to 1.$$
such that if $c+ C\in N_A(\Mu)$, then $\Omega((c+C)_\star)$ has the following properties. 
If $C= I$, then $\Omega((c+C)_\star)$ corresponds to $(\epsilon(c_1), \epsilon(c_2), \epsilon(c_3))$ in $(\integers/2\integers)^3$, where $\epsilon(c_i) = 0$ 
if $c_i \in \integers$ and $\epsilon(c_i) = 1$ if $c_i \not\in \integers$.
If $C$ is a permutation matrix, then $\Omega((c+C)_\star)$ projects to the corresponding element 
of the symmetric group $S_3$ factor of $S_3\times \{\pm I \}$. 
If $C = - I$, then $\Omega((c+C)_\star)$ projects to $-I$ in the $\{\pm I\}$ factor of $S_3\times \{\pm I\}$. 
\end{lemma}

It follows from Lemma 17 that $\mathrm{Out}(\Mu)$ has order 96.  
This fact was first determined by Charlap and Vasquez \cite{C-V}.  
Zimmermann \cite{Z} showed that the second exact sequence in Lemma 17 does not split. 

Hillman \cite{H} determined that $\mathrm{Out}(\Mu)$ has 13 conjugacy classes of inverse pairs. 
This was checked by us by a computer calculation. 
By Theorem 7, there are 13 affine equivalence classes of flat $O^3_6\,|\, S^1$ fiberings. 
These fiberings are described in Table 24 as explained in \S 8. 

There are no $O^3_6\,|\, \Iota$ fiberings, since by Lemmas 3, 8, and 17, there are no fix point free affinities of $E^3/\Mu$ of order 2.  

\begin{table} 

\begin{tabular}{c|c|c|c|c}	
no. & mfd. & cS-fbr.  &  grp. & representative      \\ \hline	
 1	&$O^4_{14}$ & $O^3_6\,|\,S^1$  & $C_1$ & 
	  	$ \left(\begin{array}{rrr} 
		1 &  0 & 0 \\ 0 & 1 & 0 \\ 0 & 0 & 1 \end{array}\right)$   \\ \hline
 2	&$O^4_{15}$ & $O^3_6\,|\,S^1$ &  $C_2$ &
		$ \left(\begin{array}{rrrr} 
		 \frac{1}{2} & 1 &  0 &  0 \\ \frac{1}{2} & 0 & 1 & 0 \\  \frac{1}{2} & 0 & 0 & 1\end{array}\right)$    \\ \hline
3	&$O^4_{16}$ & $O^3_6\,|\,S^1$ &  $C_2$ &
		$ \left(\begin{array}{rrrr} 
		 0 & 1 &  0 &  0 \\ \frac{1}{2} & 0 & 1 & 0 \\  \frac{1}{2} & 0 & 0 & 1\end{array}\right)$    \\ \hline
4	&$O^4_{17}$ & $O^3_6\,|\,S^1$ &  $C_2$ &
		$ \left(\begin{array}{rrrr} 
		 0 & 1 &  0 &  0 \\ 0 & 0 & 1 & 0 \\  \frac{1}{2} & 0 & 0 & 1\end{array}\right)$    \\ \hline
5	&$O^4_{23}$ & $O^3_6\,|\,S^1$ &  $C_2$ &
		$ \left(\begin{array}{rrrr} 
		 \frac{1}{4}  & -1 &  0 &  0 \\ 0 & 0 & 0 & -1 \\  0 & 0 & -1 & 0\end{array}\right)$    \\ \hline
6	&$O^4_{24}$ & $O^3_6\,|\,S^1$ &  $C_4$ &
		$ \left(\begin{array}{rrrr} 
		 \frac{1}{4}  & -1 &  0 &  0 \\ 0 & 0 & 0 & -1 \\  \frac{1}{2} & 0 & -1 & 0\end{array}\right)$    \\ \hline
7	&$O^4_{26}$ & $O^3_6\,|\,S^1$ &  $C_3$ &
		$ \left(\begin{array}{rrrr} 
		 \frac{1}{4}  & 0 &  0 &  1 \\ \frac{1}{2} & 1 & 0 & 0 \\  \frac{1}{4} & 0 & 1 & 0\end{array}\right)$    \\ \hline
8	&$O^4_{27}$ & $O^3_6\,|\,S^1$ &  $C_6$ &
		$ \left(\begin{array}{rrrr} 
		 \frac{1}{4}  & 0 &  0 &  1 \\ 0 & 1 & 0 & 0 \\  \frac{1}{4} & 0 & 1 & 0\end{array}\right)$    \\ \hline 
9	&$N^4_{38}$ & $O^3_6\,|\,S^1$  & $C_2$ & 
	  	$ \left(\begin{array}{rrr} 
		-1 &  0 & 0 \\ 0 & -1 & 0 \\ 0 & 0 & -1 \end{array}\right)$   \\ \hline
10	&$N^4_{39}$ & $O^3_6\,|\,S^1$  & $C_2$ & 
	  	$ \left(\begin{array}{rrrr} 
		0 & -1 &  0 & 0 \\ 0 & 0 & -1 & 0 \\ \frac{1}{2} & 0 & 0 & -1 \end{array}\right)$   \\ \hline
11	&$N^4_{40}$ & $O^3_6\,|\,S^1$  & $C_4$ & 
	  	$ \left(\begin{array}{rrrr} 
		\frac{1}{4} & 1 &  0 & 0 \\ 0 & 0 & 0 & 1 \\ \frac{1}{2} & 0 & 1 & 0 \end{array}\right)$   \\ \hline
12	&$N^4_{41}$ & $O^3_6\,|\,S^1$  & $C_4$ & 
	  	$ \left(\begin{array}{rrrr} 
		\frac{1}{4} & 1 &  0 & 0 \\ 0 & 0 & 0 & 1 \\ 0 & 0 & 1 & 0 \end{array}\right)$   \\ \hline
13	&$N^4_{43}$ & $O^3_6\,|\,S^1$ &  $C_6$ &
		$ \left(\begin{array}{rrrr} 
		 \frac{1}{4}  & 0 &  0 &  -1 \\ 0 & -1 & 0 & 0 \\  \frac{1}{4} & 0 & -1 & 0\end{array}\right)$    \\ \hline		
 \end{tabular}

\vspace{.2in}
\caption{The flat $O^3_6$ fiberings over $S^1$}
\end{table}

\section{Fibrations of Closed Flat 4-Manifolds with Generic Fiber $N^3_1$} 

In this section, we describe the affine classification of the geometric co-Seifert fibrations of closed flat 4-manifolds with generic fiber 
the flat 3-manifold $N^3_1$.  Here $N^3_1 = E^3/\Mu$ with $\Mu = \langle t_1,t_2,t_3,\alpha\rangle$ with $\alpha = e_1/2  + \mathrm{diag}(1,1,-1)$. 

\begin{lemma} 
Let $\Mu = \langle t_1, t_2, t_3, \alpha\rangle$ with $\alpha = e_1/2 +  \mathrm{diag}(1,1,-1)$.  
Then we have
$$N_A(\Mu) = \{b+B: 2b_3 \in \integers\ and\ B = \mathrm{diag}(C,\pm 1)\ with\ C\in\mathrm{GL}(2,\integers)\ and\ c_{21}\in 2\integers\}.$$
Moreover 
$\mathrm{Out}(\Mu) \cong (\integers/2\integers) \times H$ with 
$$H = \{C\in  \mathrm{GL}(2,\integers): c_{21} \in 2\integers\},$$ 
and $\Omega((e_3/2+I)_\star)$ corresponding to the generator of the $\integers/2\integers$ factor 
of $(\integers/2\integers) \times H$, and $\Omega((\mathrm{diag}(C,\pm 1))_\star)$ corresponding to $C$ 
in the $H$ factor of $(\integers/2\integers) \times H$. 
\end{lemma}
\begin{lemma} 
Let $H = \{C\in  \mathrm{GL}(2,\integers): c_{21} \in 2\integers\}$, and let 
$$A = \left(\begin{array}{rr} 0 & -1 \\  1 & 0  \end{array}\right), \  B = \left(\begin{array}{rr} 0 & -1 \\  1 & 1  \end{array}\right), \ 
C = \left(\begin{array}{rr} 0 & 1 \\  1 & 0  \end{array}\right).$$
Then $H$ is the free product of the dihedral group $\langle  -I, AC\rangle$ of order $4$ with the dihedral group $B\langle A, C\rangle B^{-1}$ of order $8$ 
amalgamated over the cyclic group $\langle -I\rangle$ of order $2$. 
Every finite subgroup of $H$ is conjugate in $H$ to a subgroup of either $\langle  -I, AC\rangle$ or $B\langle A, C\rangle B^{-1}$. 
The group $H$ has $7$ conjugacy classes of inverse pairs of elements of finite order 
represented by $I, AC, CA, -I, BCB^{-1}, BACB^{-1}, BAB^{-1}$, respectively, of orders  
$1,2,2,2,2,2,4$, respectively. 
Moreover $AC = \mathrm{diag}(-1,1)$, $CA = -AC$ and 
$$BCB^{-1} = \left(\begin{array}{rr} -1 & -1 \\  0 & 1  \end{array}\right), \  BACB^{-1} = \left(\begin{array}{rr} 1 & 0 \\  -2 & -1  \end{array}\right), \ 
BAB^{-1} = \left(\begin{array}{rr} -1 & -1 \\  2 & 1  \end{array}\right).$$
\end{lemma}
\begin{proof}
By mapping $H$ into $\mathrm{GL}(2,\integers/2\integers)$, we see that $H$ has index 3 in $\mathrm{GL}(2,\integers)$ 
with coset representatives the elements of the cyclic group $\langle B\rangle$ of order 3. 
The group $\mathrm{GL}(2,\integers)$ is the free product of the dihedral group $\langle A, C\rangle$ of order 8 
with the dihedral group $\langle B, C\rangle$ of order 12 amalgamated over the dihedral group $\langle -I, C\rangle$ of order 4 by Lemma 50 of \cite{R-TII}. 
By the subgroup theorem for free products with amalgamation as explained in Theorem 3.14 of \cite{S-W}, 
we deduce that $H$ is the free product of the dihedral group $\langle  I, AC\rangle$ of order $4$ with the dihedral group $B\langle A, C\rangle B^{-1}$ of order $8$ 
amalgamated over the cyclic group $\langle -I\rangle$ of order $2$. 

From the free product with amalgamation structure of $H$, 
we deduce that every finite subgroup of $H$ is conjugate in $H$ to a subgroup of either $\langle  I, AC\rangle$ or $B\langle A, C\rangle B^{-1}$,  
and $H$ has $7$ conjugacy classes of inverse pairs of elements of finite order 
represented by $I, AC, CA, -I, BCB^{-1}, BACB^{-1}, BAB^{-1}$, respectively, of orders  $1,2,2,2,2,2,4$, respectively.  
\end{proof}

It follows from Lemmas 18 and 19 that $\mathrm{Out}(\Mu)$ has 14 conjugacy classes of inverse pairs of elements of finite order. 
Hence, there are 14 affine equivalence classes of flat $N^3_1\, |\, S^1$ fiberings. 
These fiberings are represented by the generalized Calabi constructions described in Table 25 as explained in \S 8. 

\begin{table} 

\begin{tabular}{c|c|c|c|c}	
no. & mfd. & cS-fbr.  &  grp. & representative      \\ \hline	
 1	&$N^4_1$ & $N^3_1\,|\,S^1$  & $C_1$ & 
	  	$ \left(\begin{array}{rrr} 
		1 &  0 & 0 \\ 0 & 1 & 0 \\ 0 & 0 & 1 \end{array}\right)$   \\ \hline
 2	&$N^4_2$ & $N^3_1\,|\,S^1$ &  $C_2$ &
		$ \left(\begin{array}{rrrr} 
		 0 & 1 &  0 &  0 \\ 0 & 0 & 1 & 0 \\  \frac{1}{2} & 0 & 0 & 1\end{array}\right)$    \\ \hline
3	&$N^4_3$ & $N^3_1\,|\,S^1$ &  $C_2$ &
		$ \left(\begin{array}{rrr} 
	         1 &  0 &  0 \\  0 & -1 & 0 \\  0 & 0 & 1\end{array}\right)$    \\ \hline
4	&$N^4_4$ & $N^3_1\,|\,S^1$ &  $C_2$ &
		$ \left(\begin{array}{rrrr} 
		 0 & 1 &  0 &  0 \\ 0 & 0 & -1 & 0 \\  \frac{1}{2} & 0 & 0 & 1\end{array}\right)$    \\ \hline
5	&$N^4_4$ & $N^3_1\,|\,S^1$ &  $C_2$ &
		$ \left(\begin{array}{rrr} 
		  1 &  0 &  0 \\  -2 & -1 & 0 \\  0 & 0 & 1 \end{array}\right)$    \\ \hline
6	&$N^4_5$ & $N^3_1\,|\,S^1$ &  $C_2$ &
		$ \left(\begin{array}{rrrr} 
		 0 & 1 &  0 &  0 \\ 0 & -2 & -1 & 0 \\  \frac{1}{2} & 0 & 0 & 1\end{array}\right)$    \\ \hline
7	&$N^4_8$ & $N^3_1\,|\,S^1$ &  $C_2$ &
		$ \left(\begin{array}{rrr} 
		  -1 &  0 &  0 \\  0 & 1 & 0 \\   0 & 0 & 1\end{array}\right)$    \\ \hline
8	&$N^4_9$ & $N^3_1\,|\,S^1$ &  $C_2$ &
		$ \left(\begin{array}{rrrr} 
		 0  & -1 &  0 &  0 \\ 0 & 0 & 1 & 0 \\  \frac{1}{2} & 0 & 0 & 1\end{array}\right)$    \\ \hline 
9	&$N^4_{10}$ & $N^3_1\,|\,S^1$  & $C_2$ & 
	  	$ \left(\begin{array}{rrr} 
		-1 &  -1 & 0 \\ 0 & 1 & 0 \\ 0 & 0 & 1 \end{array}\right)$   \\ \hline
10	&$N^4_{12}$ & $N^3_1\,|\,S^1$  & $C_2$ & 
	  	$ \left(\begin{array}{rrrr} 
		0 & -1 &  -1 & 0 \\ 0 & 0 & 1 & 0 \\ \frac{1}{2} & 0 & 0 & 1 \end{array}\right)$   \\ \hline
11	&$N^4_{24}$ & $N^3_1\,|\,S^1$  & $C_2$ & 
	  	$ \left(\begin{array}{rrr} 
		 -1 &  0 & 0 \\ 0 & -1 & 0 \\  0 & 0 & 1 \end{array}\right)$   \\ \hline
12	&$N^4_{25}$ & $N^3_1\,|\,S^1$  & $C_2$ & 
	  	$ \left(\begin{array}{rrrr} 
		0 & -1 &  0 & 0 \\ 0 & 0 & -1 & 0 \\ \frac{1}{2} & 0 & 0 & 1 \end{array}\right)$   \\ \hline
13	&$N^4_{27}$ & $N^3_1\,|\,S^1$ &  $C_4$ &
		$ \left(\begin{array}{rrrr} 
		  -1 &  -1 &  0 \\  2 & 1 & 0 \\   0 & 0 & 1\end{array}\right)$    \\ \hline
14	&$N^4_{28}$ & $N^3_1\,|\,S^1$ &  $C_4$ &
		$ \left(\begin{array}{rrrr} 
		 0 & -1 &  -1 &  0 \\  0 & 2 & 1 & 0 \\ \frac{1}{2} &  0 & 0 & 1\end{array}\right)$    \\ \hline			
 \end{tabular}

\vspace{.2in}
\caption{The flat $N^3_1$ fiberings over $S^1$}
\end{table}

Next we consider $N^3_1\, | \, I$ fiberings. 

\begin{lemma}  
Let $\Mu = \langle t_1, t_2, t_3, \alpha\rangle$ with $\alpha = e_1/2 + \mathrm{diag}(1,1,-1)$,   
and let $\beta = b + \mathrm{diag}(C,\pm 1)$ be an affinity of $E^3$ that normalizes $\Mu$ 
such that $\beta_\star$ has order $2$ and $\beta_\star$ does not fix a point of $E^3/\Mu$. 
Then $C$ has order $1$ or  $2$. 
If $C$ has order $1$, then $\beta_\star$ is conjugate to either $(e_2/2+I)_\star$, $(e_3/2+I)_\star$, or $(e_2/2+e_3/2+I)_\star$. 
The quotient of $E^3/\Mu$ by the action of  $\beta_\star$ is of type $N^3_1$, $N^3_1$, or $N^3_2$, respectively. 

If $C$ has order $2$, then $\beta_\star$ is conjugate to either $(e_3/2 +\mathrm{diag}(1,-1,1))_\star$, 
or $(e_2/2+\mathrm{diag}(-1,1,1))_\star$, or $(e_2/2+e_3/2+ \mathrm{diag}(-1,1,1))_\star$, or $(e_3/2+B)_\star$, 
with $Be_1 = e_1-2e_2$, $Be_2 = -e_2$, and $Be_3 = e_3$. 
The quotient of $E^3/\Mu$ by the action of $\beta_\star$ is of type $N^3_3$, $N^3_3$, $N^3_4$, or $N^3_4$, respectively. 
\end{lemma} 
\begin{proof}
If $C$ has order $1$, then $\beta_\star$ is conjugate to either $(e_2/2+I)_\star$, $(e_3/2+I)_\star$, or $(e_2/2+e_3/2+I)_\star$ by Lemmas 3 and 8. 
The quotient of $E^3/\Mu$ by the action of  $(e_2/2+I)_\star$ or $(e_3/2+I)_\star$ is clearly of type $N^3_1$. 
The quotient of $E^3/\Mu$ by the action of  $(e_2/2+e_3/2 + I)_\star$ is of type $N^3_2$, 
since the 3-space group $\langle t_2, \alpha, e_2/2+e_3/2 + I\rangle$ is the same as 
the 3-space group $\langle e_2/2-e_3/2 + I, e_2/2+e_3/2 + I, \alpha\rangle$, which has IT number 9 according to Table 1B of \cite{B-Z}. 

If $C$ has order $2$, then $\beta_\star$ is conjugate to either $(e_3/2 +\mathrm{diag}(1,-1,1))_\star$, 
or $(e_2/2+\mathrm{diag}(-1,1,1))_\star$, or $(e_2/2+e_3/2+ \mathrm{diag}(-1,1,1))_\star$, or $(e_3/2+B)_\star$, 
with $Be_1 = e_1-2e_2$, $Be_2 = -e_2$, and $Be_3 = e_3$ 
by Lemmas 3, 8, and 19. 

The quotient of $E^3/\Mu$ by the action of $(e_3/2 +\mathrm{diag}(1,-1,1))_\star$ is of type $N^3_3$, 
since the 3-space group $\langle t_2, \alpha, e_3/2 +\mathrm{diag}(1,-1,1)\rangle$ has IT number 29 according to Table 1B of \cite{B-Z}. 
The quotient of $E^3/\Mu$ by the action of  $(e_2/2+\mathrm{diag}(-1,1,1))_\star$ is of type $N^3_3$, 
since the 3-space group $\langle t_3, \alpha, e_2/2+\mathrm{diag}(-1,1,1)\rangle$ has IT number 29 according to Table 1B of \cite{B-Z}. 

The quotient of $E^3/\Mu$ by the action of $(e_2/2+e_3/2+ \mathrm{diag}(-1,1,1))_\star$ is of type $N^3_4$, 
since the 3-space group $\langle t_2, \alpha, e_2/2+e_3/2+ \mathrm{diag}(-1,1,1)\rangle$ is of type $N^3_4$ according to CARAT. 
The quotient of $E^3/\Mu$ by the action of $(e_3/2+B)_\star$ is of type $N^3_4$, 
since the 3-space group $\langle t_2, \alpha, e_3/2+B\rangle$ is of type $N^3_4$ according to CARAT. 
\end{proof}

Let $\beta = b+\mathrm{diag}(B, \pm 1)$ and $\gamma = c +\mathrm{diag}(C,\pm 1)$ be affinities of $E^3$ that normalize $\Mu$ and such that 
$\beta_\star$ and $\gamma_\star$ are order 2 affinities of $N^3_1$ which do not fix a point of $N^3_1$, and $\Omega(\beta_\star\gamma_\star)$ has finite order. 
If $\langle B, C\rangle$ has order 2, then the pair $\{B,  C\}$ is conjugate to the pair $\{I, D\}$ 
where $D=\mathrm{diag}(-1,1)$ or  $\mathrm{diag}(1,-1)$ or $De_1 = e_1- 2e_2$ and $De_2 = -e_1$ by Lemma 20. 
If $\langle B, C\rangle$ has order greater than 2, then $\langle  B,  C\rangle$ has order 4 
and the pair $\{B, C\}$ is conjugate to either the pair $\{\mathrm{diag}(-1,1),  \mathrm{diag}(1,-1)\}$, 
or the pair $\{D, -D\}$, where $De_1 = e_1- 2e_2$ and $De_2 = -e_1$, by Lemmas 18, 19, and 20.  
By considering all the possibilities for the conjugacy classes of $\beta_\star$ and $\gamma_\star$, 
we find that there are 37 equivalence classes of pairs $\{\beta_\star,\gamma_\star\}$, and so there are 37 
affine equivalence classes of flat $N^3_1\,|\, \Iota$ fiberings by Theorem 10.  
These fiberings are described in Tables 26, 27, and 28 as explained in \S 8. 

\vspace{.15in}
\noindent{\bf Example 5.}
Consider the torsion-free 4-space group 6/2/1/27 in \cite{B-Z}. 
After the permutation in coordinates $(143)$, the group is defined by 
$$\Gamma = \langle t_1, t_2, t_3, t_4, \alpha, \beta, \gamma\rangle$$
with $\alpha = e_1/2 + \mathrm{diag}(1,1,-1,1)$, and $\beta = e_2/2+\mathrm{diag}(-1,1,-1,-1)$, and $\gamma = e_3/2+e_4/2 + \mathrm{diag}(1,-1,1,-1)$. 

Let $\Nu = \langle t_1, t_2, t_3, \alpha\rangle$.   Then $\Nu$ is a complete normal subgroup of $\Gamma$ of type $N^3_1$. 
Let $V = \mathrm{Span}\{e_1, e_2, e_3\}$.  Then $V/\Nu$ is a flat 3-manifold of type $N^3_1$. 
The group $\Gamma/\Nu$ is infinite dihedral, since $\Nu\beta$ and $\Nu\gamma$ are order 2 generators of $\Gamma/\Nu$, 
and $\Nu\beta$ acts on the line $V^\perp$ by the reflection $(-I)_\star$ and $\Nu\gamma$ acts on the line $V^\perp$ by the reflection $(e_4/2-I)_\star$.  

Now $\Nu\beta$ acts on $V/\Nu$ by $(e_2/2+\mathrm{diag}(-1,1,-1))_\star$ in the first three coordinates, 
which is the same as $(e_1/2+e_2/2+\mathrm{diag}(-1,1,1))_\star$ after multiplying by $\alpha_\star$, 
and which is conjugate to $(e_2/2+\mathrm{diag}(-1,1,1))_\star$ after conjugating by $(-e_1/4+I)_\star$. 
Moreover $\Nu\gamma$ acts on $V/\Nu$ by $(e_3/2+\mathrm{diag}(1,-1,1))_\star$.  
Conjugating by $(-e_1/4+I)_\star$ does not affect the action of $\Nu\gamma$ on $V/\Nu$. 
Therefore the geometric fibration of $E^4/\Gamma$ determined by $\Nu$ is affinely equivalent to the $N^3_1\,|\,\Iota$ fibering described in Row 36 of Table 28 
by Theorem 10.  Hence $E^4/\Gamma$ is of type $N^4_{46}$. 

The first Betti number of $E^4/\Gamma$ is zero, since the first Betti number of $E^4/\Gamma$ is the dimension of the fixed space of the point group of $\Gamma$ 
which is zero. This example is a counterexample to Hillman's claim at the bottom of page 37 of \cite{H} that there are no closed flat 4-manifolds 
with first Betti number zero that have a $N^3_1\,|\,\Iota$ fibering. 
Hillman's argument for the classification of all the closed flat 4-manifolds, with zero first Betti number, in \cite{H} has a gap, 
since he did not consider all possible fibrations of these manifolds over $\Iota$.  
This gap is filled by our classification of all the fibrations of closed flat 4-manifolds over $\Iota$.

\begin{table} 

\begin{tabular}{c|c|c|c|c|c}	
no. & mfd. & cS-fbr.  &  grp. & pair representatives  & s-fbrs.  \\ \hline	
1	&$N^4_3$ & $N^3_1\,|\,\Iota$ & $D_1$ &
 		$ \left(\begin{array}{rrrr} 
		0 & 1 & 0 &  0 \\ \frac{1}{2} & 0 & 1 & 0 \\ 0 & 0 & 0 & 1\end{array}\right)$, 
		$ \left(\begin{array}{rrrr} 
		0 & 1 & 0 &  0 \\ \frac{1}{2} & 0 & 1 & 0 \\ 0 & 0 & 0 & 1\end{array}\right)$ & $N^3_1, N^3_1$ \\ \hline
2	&$N^4_4$ & $N^3_1\,|\,\Iota$ & $D_1$ &
 		$ \left(\begin{array}{rrrr} 
		0 & 1 & 0 &  0 \\ \frac{1}{2} & 0 & 1 & 0 \\ \frac{1}{2} & 0 & 0 & 1\end{array}\right)$, 
		$ \left(\begin{array}{rrrr} 
		0 & 1 & 0 &  0 \\ \frac{1}{2} & 0 & 1 & 0 \\ \frac{1}{2} & 0 & 0 & 1\end{array}\right)$ & $N^3_2, N^3_2$ \\ \hline
3	&$N^4_6$ & $N^3_1\,|\,\Iota$ & $D_2$ &
 		$ \left(\begin{array}{rrrr} 
		0 & 1 & 0 &  0 \\ \frac{1}{2} & 0 & 1 & 0 \\ 0 & 0 & 0 & 1\end{array}\right)$, 
		$ \left(\begin{array}{rrrr} 
		0 & 1 & 0 &  0 \\ \frac{1}{2} & 0 & 1 & 0 \\ \frac{1}{2} & 0 & 0 & 1\end{array}\right)$ & $N^3_1, N^3_2$ \\ \hline
4	&$N^4_8$ & $N^3_1\,|\,\Iota$ & $D_1$ &
 		$ \left(\begin{array}{rrrr} 
		0 & 1 & 0 &  0 \\ 0 & 0 & 1 & 0 \\ \frac{1}{2} & 0 & 0 & 1\end{array}\right)$, 
		$ \left(\begin{array}{rrrr} 
		0 & 1 & 0 &  0 \\ 0 & 0 & 1 & 0 \\ \frac{1}{2} & 0 & 0 & 1\end{array}\right)$ & $N^3_1, N^3_1$ \\ \hline
5	&$N^4_{10}$ & $N^3_1\,|\,\Iota$ & $D_2$ &
 		$ \left(\begin{array}{rrrr} 
		0 & 1 & 0 &  0 \\ \frac{1}{2} & 0 & 1 & 0 \\ 0 & 0 & 0 & 1\end{array}\right)$, 
		$ \left(\begin{array}{rrrr} 
		\frac{1}{2} & 1 & 0 &  0 \\ \frac{1}{2} & 0 & 1 & 0 \\ 0 & 0 & 0 & 1\end{array}\right)$ & $N^3_1, N^3_1$ \\ \hline
6	&$N^4_{11}$ & $N^3_1\,|\,\Iota$ & $D_2$ &
 		$ \left(\begin{array}{rrrr} 
		0 & 1 & 0 &  0 \\ 0 & 0 & 1 & 0 \\ \frac{1}{2} & 0 & 0 & 1\end{array}\right)$, 
		$ \left(\begin{array}{rrrr} 
		0 & 1 & 0 &  0 \\ \frac{1}{2} & 0 & 1 & 0 \\ \frac{1}{2} & 0 & 0 & 1\end{array}\right)$ & $N^3_1, N^3_2$ \\ \hline
7	&$N^4_{12}$ & $N^3_1\,|\,\Iota$ & $D_2$ &
 		$ \left(\begin{array}{rrrr} 
		0 & 1 & 0 &  0 \\ \frac{1}{2} & 0 & 1 & 0 \\ \frac{1}{2} & 0 & 0 & 1\end{array}\right)$, 
		$ \left(\begin{array}{rrrr} 
		\frac{1}{2} & 1 & 0 &  0 \\ \frac{1}{2} & 0 & 1 & 0 \\ \frac{1}{2} & 0 & 0 & 1\end{array}\right)$ & $N^3_2, N^3_2$ \\ \hline
8	&$N^4_{13}$ & $N^3_1\,|\,\Iota$ & $D_2$ &
 		$ \left(\begin{array}{rrrr} 
		0 & 1 & 0 &  0 \\ \frac{1}{2} & 0 & 1 & 0 \\ 0 & 0 & 0 & 1\end{array}\right)$, 
		$ \left(\begin{array}{rrrr} 
		0 & 1 & 0 &  0 \\ 0 & 0 & 1 & 0 \\ \frac{1}{2} & 0 & 0 & 1\end{array}\right)$ & $N^3_1, N^3_1$ \\ \hline
9	&$N^4_{13}$ & $N^3_1\,|\,\Iota$ & $D_2$ &
 		$ \left(\begin{array}{rrrr} 
		0 & 1 & 0 &  0 \\ \frac{1}{2} & 0 & 1 & 0 \\ 0 & 0 & 0 & 1\end{array}\right)$, 
		$ \left(\begin{array}{rrrr} 
		\frac{1}{2}  & 1 & 0 &  0 \\ \frac{1}{2} & 0 & 1 & 0 \\ \frac{1}{2} & 0 & 0 & 1\end{array}\right)$ & $N^3_1, N^3_2$ \\ \hline
10	&$N^4_{22}$ & $N^3_1\,|\,\Iota$ & $D_1$ &
 		$ \left(\begin{array}{rrrr} 
		0 & 1 & 0 &  0 \\ 0 & 0 & -1 & 0 \\ \frac{1}{2} & 0 & 0 & 1\end{array}\right)$, 
		$ \left(\begin{array}{rrrr} 
		0 & 1 & 0 &  0 \\ 0 & 0 & -1 & 0 \\ \frac{1}{2} & 0 & 0 & 1\end{array}\right)$ & $N^3_3, N^3_3$ \\ \hline
11	&$N^4_{23}$ & $N^3_1\,|\,\Iota$ & $D_1$ &
 		$ \left(\begin{array}{rrrr} 
		0 & 1 & 0 &  0 \\ 0 & -2 & -1 & 0 \\ \frac{1}{2} & 0 & 0 & 1\end{array}\right)$, 
		$ \left(\begin{array}{rrrr} 
		0 & 1 & 0 &  0 \\ 0 & -2 & -1 & 0 \\ \frac{1}{2} & 0 & 0 & 1\end{array}\right)$ & $N^3_4, N^3_4$ \\ \hline
12	&$N^4_{24}$ & $N^3_1\,|\,\Iota$ & $D_1$ &
 		$ \left(\begin{array}{rrrr} 
		0 & -1 & 0 &  0 \\ \frac{1}{2} & 0 & 1 & 0 \\ 0 & 0 & 0 & 1\end{array}\right)$, 
		$ \left(\begin{array}{rrrr} 
		0 & -1 & 0 &  0 \\ \frac{1}{2} & 0 & 1 & 0 \\ 0 & 0 & 0 & 1\end{array}\right)$ & $N^3_3, N^3_3$ \\ \hline
13	&$N^4_{25}$ & $N^3_1\,|\,\Iota$ & $D_1$ &
 		$ \left(\begin{array}{rrrr} 
		0 & -1 & 0 &  0 \\ \frac{1}{2} & 0 & 1 & 0 \\ \frac{1}{2} & 0 & 0 & 1\end{array}\right)$, 
		$ \left(\begin{array}{rrrr} 
		0 & -1 & 0 &  0 \\ \frac{1}{2} & 0 & 1 & 0 \\ \frac{1}{2} & 0 & 0 & 1\end{array}\right)$ & $N^3_4, N^3_4$ \\ \hline
14	&$N^4_{26}$ & $N^3_1\,|\,\Iota$ & $D_2$ &
 		$ \left(\begin{array}{rrrr} 
		0 & -1 & 0 &  0 \\ \frac{1}{2} & 0 & 1 & 0 \\ 0 & 0 & 0 & 1\end{array}\right)$, 
		$ \left(\begin{array}{rrrr} 
		0 & -1 & 0 &  0 \\ \frac{1}{2} & 0 & 1 & 0 \\ \frac{1}{2} & 0 & 0 & 1\end{array}\right)$ & $N^3_3, N^3_4$ \\ \hline \end{tabular}

\vspace{.2in}
\caption{The flat $N^3_1$ fiberings over $\Iota$, Part I}
\end{table}

\begin{table} 

\begin{tabular}{c|c|c|c|c|c}	
no. & mfd. & cS-fbr.  &  grp. & pair representatives  & s-fbrs.  \\ \hline	
15	&$N^4_{30}$ & $N^3_1\,|\,\Iota$ & $D_2$ &
 		$ \left(\begin{array}{rrrr} 
		 \frac{1}{2} & 1 & 0 &  0 \\ \frac{1}{2} & 0 & 1 & 0 \\ 0 & 0 & 0 & 1\end{array}\right)$, 
		$ \left(\begin{array}{rrrr} 
		0 & -1& 0 &  0 \\ \frac{1}{2} & 0 & 1 & 0 \\ 0 & 0 & 0 & 1\end{array}\right)$ & $N^3_1, N^3_3$ \\ \hline
16	&$N^4_{31}$ & $N^3_1\,|\,\Iota$ & $D_2$ &
 		$ \left(\begin{array}{rrrr} 
		0 & 1 & 0 &  0 \\ \frac{1}{2} & 0 & 1 & 0 \\ 0 & 0 & 0 & 1\end{array}\right)$, 
		$ \left(\begin{array}{rrrr} 
		0 & -1& 0 &  0 \\ \frac{1}{2} & 0 & 1 & 0 \\ 0 & 0 & 0 & 1\end{array}\right)$ & $N^3_1, N^3_3$ \\ \hline
17	&$N^4_{31}$ & $N^3_1\,|\,\Iota$ & $D_2$ &
 		$ \left(\begin{array}{rrrr} 
		0 & 1 & 0 &  0 \\ 0 & 0 & 1 & 0 \\ \frac{1}{2} & 0 & 0 & 1\end{array}\right)$, 
		$ \left(\begin{array}{rrrr} 
		0 & 1& 0 &  0 \\ 0 & 0 & -1 & 0 \\  \frac{1}{2} & 0 & 0 & 1\end{array}\right)$ & $N^3_1, N^3_3$ \\ \hline
18	&$N^4_{32}$ & $N^3_1\,|\,\Iota$ & $D_2$ &
 		$ \left(\begin{array}{rrrr} 
		0 & 1 & 0 &  0 \\ 0 & 0 & 1 & 0 \\ \frac{1}{2} & 0 & 0 & 1\end{array}\right)$, 
		$ \left(\begin{array}{rrrr} 
		0 & -1& 0 &  0 \\  \frac{1}{2}  & 0 & 1 & 0 \\ 0 & 0 & 0 & 1\end{array}\right)$ & $N^3_1, N^3_3$ \\ \hline
19	&$N^4_{32}$ & $N^3_1\,|\,\Iota$ & $D_2$ &
 		$ \left(\begin{array}{rrrr} 
		\frac{1}{2} & 1 & 0 &  0 \\ \frac{1}{2} & 0 & 1 & 0 \\ \frac{1}{2} & 0 & 0 & 1\end{array}\right)$, 
		$ \left(\begin{array}{rrrr} 
		0 & -1& 0 &  0 \\  \frac{1}{2}  & 0 & 1 & 0 \\ \frac{1}{2} & 0 & 0 & 1\end{array}\right)$ & $N^3_2, N^3_4$ \\ \hline
20	&$N^4_{33}$ & $N^3_1\,|\,\Iota$ & $D_2$ &
 		$ \left(\begin{array}{rrrr} 
		0 & 1 & 0 &  0 \\ 0 & 0 & 1 & 0 \\ \frac{1}{2} & 0 & 0 & 1\end{array}\right)$, 
		$ \left(\begin{array}{rrrr} 
		0 & -1& 0 &  0 \\  \frac{1}{2}  & 0 & 1 & 0 \\ \frac{1}{2} & 0 & 0 & 1\end{array}\right)$ & $N^3_1, N^3_4$ \\ \hline
21	&$N^4_{33}$ & $N^3_1\,|\,\Iota$ & $D_2$ &
 		$ \left(\begin{array}{rrrr} 
		0 & 1 & 0 &  0 \\  \frac{1}{2} & 0 & 1 & 0 \\ \frac{1}{2} & 0 & 0 & 1\end{array}\right)$, 
		$ \left(\begin{array}{rrrr} 
		0 & -1& 0 &  0 \\  \frac{1}{2}  & 0 & 1 & 0 \\ \frac{1}{2} & 0 & 0 & 1\end{array}\right)$ & $N^3_2, N^3_4$ \\ \hline
22	&$N^4_{33}$ & $N^3_1\,|\,\Iota$ & $D_2$ &
 		$ \left(\begin{array}{rrrr} 
		0 & 1 & 0 &  0 \\  0 & 0 & 1 & 0 \\ \frac{1}{2} & 0 & 0 & 1\end{array}\right)$, 
		$ \left(\begin{array}{rrrr} 
		0 & 1& 0 &  0 \\  0 & -2 & -1 & 0 \\ \frac{1}{2} & 0 & 0 & 1\end{array}\right)$ & $N^3_1, N^3_4$ \\ \hline
23	&$N^4_{34}$ & $N^3_1\,|\,\Iota$ & $D_2$ &
 		$ \left(\begin{array}{rrrr} 
		\frac{1}{2} & 1 & 0 &  0 \\ \frac{1}{2} & 0 & 1 & 0 \\ \frac{1}{2} & 0 & 0 & 1\end{array}\right)$, 
		$ \left(\begin{array}{rrrr} 
		0 & -1& 0 &  0 \\  \frac{1}{2}  & 0 & 1 & 0 \\ 0 & 0 & 0 & 1\end{array}\right)$ & $N^3_2, N^3_3$ \\ \hline
24	&$N^4_{35}$ & $N^3_1\,|\,\Iota$ & $D_2$ &
 		$ \left(\begin{array}{rrrr} 
		0 & 1 & 0 &  0 \\ \frac{1}{2} & 0 & 1 & 0 \\ \frac{1}{2} & 0 & 0 & 1\end{array}\right)$, 
		$ \left(\begin{array}{rrrr} 
		0 & -1& 0 &  0 \\  \frac{1}{2}  & 0 & 1 & 0 \\ 0 & 0 & 0 & 1\end{array}\right)$ & $N^3_2, N^3_3$ \\ \hline
25	&$N^4_{35}$ & $N^3_1\,|\,\Iota$ & $D_2$ &
 		$ \left(\begin{array}{rrrr} 
		0 & 1 & 0 &  0 \\   \frac{1}{2} & 0 & 1 & 0 \\ \frac{1}{2} & 0 & 0 & 1\end{array}\right)$, 
		$ \left(\begin{array}{rrrr} 
		0 & 1& 0 &  0 \\  0 & -2 & -1 & 0 \\ \frac{1}{2} & 0 & 0 & 1\end{array}\right)$ & $N^3_2, N^3_4$ \\ \hline
26	&$N^4_{36}$ & $N^3_1\,|\,\Iota$ & $D_2$ &
 		$ \left(\begin{array}{rrrr} 
		 \frac{1}{2} & 1 & 0 &  0 \\  \frac{1}{2} & 0 & 1 & 0 \\ 0 & 0 & 0 & 1\end{array}\right)$, 
		$ \left(\begin{array}{rrrr} 
		0 & -1& 0 &  0 \\  \frac{1}{2}  & 0 & 1 & 0 \\ \frac{1}{2} & 0 & 0 & 1\end{array}\right)$ & $N^3_1, N^3_4$ \\ \hline
27	&$N^4_{36}$ & $N^3_1\,|\,\Iota$ & $D_2$ &
 		$ \left(\begin{array}{rrrr} 
		0 & 1 & 0 &  0 \\  \frac{1}{2} & 0 & 1 & 0 \\ 0 & 0 & 0 & 1\end{array}\right)$, 
		$ \left(\begin{array}{rrrr} 
		0 & 1& 0 &  0 \\  0 & -2 & -1 & 0 \\ \frac{1}{2} & 0 & 0 & 1\end{array}\right)$ & $N^3_1, N^3_4$ \\ \hline
28	&$N^4_{37}$ & $N^3_1\,|\,\Iota$ & $D_2$ &
 		$ \left(\begin{array}{rrrr} 
		0 & 1 & 0 &  0 \\  \frac{1}{2} & 0 & 1 & 0 \\ 0 & 0 & 0 & 1\end{array}\right)$, 
		$ \left(\begin{array}{rrrr} 
		0 & 1 & 0 &  0 \\ 0 & 0 & -1 & 0 \\  \frac{1}{2} & 0 & 0 & 1\end{array}\right)$ & $N^3_1, N^3_3$ \\ \hline \end{tabular}

\vspace{.2in}
\caption{The flat $N^3_1$ fiberings over $\Iota$, Part II}
\end{table}

\begin{table} 

\begin{tabular}{c|c|c|c|c|c}	
no. & mfd. & cS-fbr.  &  grp. & pair representatives  & s-fbrs.  \\ \hline	
29	&$N^4_{37}$ & $N^3_1\,|\,\Iota$ & $D_2$ &
 		$ \left(\begin{array}{rrrr} 
		0 & 1 & 0 &  0 \\  \frac{1}{2} & 0 & 1 & 0 \\ \frac{1}{2}  & 0 & 0 & 1\end{array}\right)$, 
		$ \left(\begin{array}{rrrr} 
		0 & 1 & 0 &  0 \\ 0 & 0 & -1 & 0 \\  \frac{1}{2} & 0 & 0 & 1\end{array}\right)$ & $N^3_2, N^3_3$ \\ \hline
30	&$N^4_{37}$ & $N^3_1\,|\,\Iota$ & $D_2$ &
 		$ \left(\begin{array}{rrrr} 
		0 & 1 & 0 &  0 \\  \frac{1}{2} & 0 & 1 & 0 \\ 0 & 0 & 0 & 1\end{array}\right)$, 
		$ \left(\begin{array}{rrrr} 
		0 & -1 & 0 &  0 \\ \frac{1}{2} & 0 & 1 & 0 \\  \frac{1}{2} & 0 & 0 & 1\end{array}\right)$ & $N^3_1, N^3_4$ \\ \hline
31	&$N^4_{38}$ & $N^3_1\,|\,\Iota$ & $D_2$ &
 		$ \left(\begin{array}{rrrr} 
		\frac{1}{2} & 1 & 0 &  0 \\  \frac{1}{2} & 0 & 1 & 0 \\ 0 & 0 & 0 & 1\end{array}\right)$, 
		$ \left(\begin{array}{rrrr} 
		0 & 1 & 0 &  0 \\ 0 & 0 & -1 & 0 \\  \frac{1}{2} & 0 & 0 & 1\end{array}\right)$ & $N^3_1, N^3_3$ \\ \hline
32	&$N^4_{38}$ & $N^3_1\,|\,\Iota$ & $D_2$ &
 		$ \left(\begin{array}{rrrr} 
		\frac{1}{2} & 1 & 0 &  0 \\  \frac{1}{2} & 0 & 1 & 0 \\ \frac{1}{2} & 0 & 0 & 1\end{array}\right)$, 
		$ \left(\begin{array}{rrrr} 
		0 & 1 & 0 &  0 \\ 0 & 0 & -1 & 0 \\  \frac{1}{2} & 0 & 0 & 1\end{array}\right)$ & $N^3_2, N^3_3$ \\ \hline
33	&$N^4_{38}$ & $N^3_1\,|\,\Iota$ & $D_2$ &
 		$ \left(\begin{array}{rrrr} 
		\frac{1}{2} & 1 & 0 &  0 \\  \frac{1}{2} & 0 & 1 & 0 \\ 0 & 0 & 0 & 1\end{array}\right)$, 
		$ \left(\begin{array}{rrrr} 
		0 & 1 & 0 &  0 \\ 0 & -2 & -1 & 0 \\  \frac{1}{2} & 0 & 0 & 1\end{array}\right)$ & $N^3_1, N^3_4$ \\ \hline
34	&$N^4_{39}$ & $N^3_1\,|\,\Iota$ & $D_2$ &
 		$ \left(\begin{array}{rrrr} 
		\frac{1}{2} & 1 & 0 &  0 \\  \frac{1}{2} & 0 & 1 & 0 \\  \frac{1}{2} & 0 & 0 & 1\end{array}\right)$, 
		$ \left(\begin{array}{rrrr} 
		0 & 1 & 0 &  0 \\ 0 & -2 & -1 & 0 \\  \frac{1}{2} & 0 & 0 & 1\end{array}\right)$ & $N^3_2, N^3_4$ \\ \hline
35	&$N^4_{45}$ & $N^3_1\,|\,\Iota$ & $D_2$ &
 		$ \left(\begin{array}{rrrr} 
		0 & 1 & 0 &  0 \\ 0 & -2 & -1 & 0 \\  \frac{1}{2} & 0 & 0 & 1\end{array}\right)$, 
		$ \left(\begin{array}{rrrr} 
		0 & -1 & 0 &  0 \\ 0 & 2 & 1 & 0 \\  \frac{1}{2} & 0 & 0 & 1\end{array}\right)$ & $N^3_4, N^3_4$ \\ \hline
36	&$N^4_{46}$ & $N^3_1\,|\,\Iota$ & $D_2$ &
 		$ \left(\begin{array}{rrrr} 
		0 & 1 & 0 &  0 \\ 0 & 0 & -1 & 0 \\ \frac{1}{2} & 0 & 0 & 1\end{array}\right)$, 
		$ \left(\begin{array}{rrrr} 
		0 & -1 & 0 &  0 \\ \frac{1}{2} & 0 & 1 & 0 \\ 0 & 0 & 0 & 1\end{array}\right)$ & $N^3_3, N^3_3$ \\ \hline
37	&$N^4_{46}$ & $N^3_1\,|\,\Iota$ & $D_2$ &
 		$ \left(\begin{array}{rrrr} 
		0 & 1 & 0 &  0 \\ 0 & 0 & -1 & 0 \\ \frac{1}{2} & 0 & 0 & 1\end{array}\right)$, 
		$ \left(\begin{array}{rrrr} 
		0 & -1 & 0 &  0 \\ \frac{1}{2} & 0 & 1 & 0 \\   \frac{1}{2} & 0 & 0 & 1\end{array}\right)$ & $N^3_3, N^3_4$ \\ \hline
 \end{tabular}

\vspace{.2in}
\caption{The flat $N^3_1$ fiberings over $\Iota$, Part III}
\end{table}

\newpage

\section{Fibrations of Closed Flat 4-Manifolds with Generic Fiber $N^3_2$} 

In this section, we describe the affine classification of the geometric co-Seifert fibrations of closed flat 4-manifolds with generic fiber 
the flat 3-manifold $N^3_2$.  Here $N^3_2 = E^3/\Mu$ with $\Mu = \langle t_1,t_2,t_3,\alpha\rangle$ and $\alpha = e_3/2  + A$ and
$$A = \left(\begin{array}{rrr}
0 & 1 & 0 \\ 1 & 0 & 0 \\ 0 & 0 & 1
\end{array}\right).$$  

\begin{lemma} 
Let $\Mu = \langle t_1, t_2, t_3, \alpha\rangle$ with $\alpha = e_3/2 + A$ and $A$ as above.  
Then $N_A(\Mu)$ is the set of all $s+S$ such that 
$$S =  \left(\begin{array}{rrr}
\frac{d+e}{2} & \frac{d-e}{2} & \frac{c}{2} \\ \frac{d-e}{2} & \frac{d+e}{2} & \frac{c}{2} \\ b & b & a
\end{array}\right)
\ with\  
\left(\begin{array}{rrr}
a & b &0 \\ c & d & 0 \\ 0 & 0 & e
\end{array}\right) \in \mathrm{GL}(3,\integers)$$
and $c \in 2\integers$ and $s_1-s_2 + c/4 \in \integers$.
Moreover
$$\mathrm{Out}(\Mu) \cong  \{C\in  \mathrm{GL}(2,\integers): c_{21} \in 2\integers\}$$ 
with $\Omega((s+S)_\star)$ corresponding to the $2\times 2$ block of the above block diagonal matrix. 
\end{lemma}

By Lemmas 19 and 21, we have that $\mathrm{Out}(\Mu)$ has 7 conjugacy classes of inverse pairs of elements of finite order. 
Representatives in $N_A(\Mu)$ for these conjugacy classes are given in Table 29. 
Hence, there are 7 affine equivalence classes of flat $N_2^3\,|\,S^1$ fiberings. 
These fiberings are described in Table 29 as explained in \S 8. 

\begin{table} 

\begin{tabular}{c|c|c|c|c}	
no. & mfd. & cS-fbr.  &  grp. & representative      \\ \hline	
 1	&$N^4_2$ & $N^3_2\,|\,S^1$  & $C_1$ & 
	  	$ \left(\begin{array}{rrr} 
		1 &  0 & 0 \\ 0 & 1 & 0 \\ 0 & 0 & 1 \end{array}\right)$   \\ \hline
 2	&$N^4_6$ & $N^3_2\,|\,S^1$ &  $C_2$ &
		$ \left(\begin{array}{rrr} 
		 -1 &  0 &  0 \\  0 & -1 & 0 \\  0 & 0 & 1\end{array}\right)$    \\ \hline
3	&$N^4_7$ & $N^3_2\,|\,S^1$ &  $C_2$ &
		$ \left(\begin{array}{rrrr} 
	         0 & -1 &  0 &  -1  \\  \frac{1}{2} & 0 & -1 & -1 \\  0 & 0 & 0 & 1\end{array}\right)$    \\ \hline
4	&$N^4_{11}$ & $N^3_2\,|\,S^1$ &  $C_2$ &
		$ \left(\begin{array}{rrr} 
		  1 &  0 &  0 \\  0 & 1 & 0 \\   0 & 0 & - 1\end{array}\right)$    \\ \hline
5	&$N^4_{13}$ & $N^3_2\,|\,S^1$ &  $C_2$ &
		$ \left(\begin{array}{rrr} 
		  -1 &  0 &  0 \\  0 & -1 & 0 \\  1 & 1 & 1 \end{array}\right)$    \\ \hline
6	&$N^4_{26}$ & $N^3_2\,|\,S^1$ &  $C_2$ &
		$ \left(\begin{array}{rrr} 
		 -1 &  0 &  0 \\  0 & -1 & 0 \\  0 & 0 & -1\end{array}\right)$    \\ \hline
7	&$N^4_{29}$ & $N^3_2\,|\,S^1$ &  $C_4$ &
		$ \left(\begin{array}{rrrr} 
		  0 & -1 &  0 &  -1 \\  \frac{1}{2} & 0 & -1 & -1 \\  0 & 1 & 1 & 1\end{array}\right)$    \\ \hline			
 \end{tabular}

\vspace{.2in}
\caption{The flat $N^3_2$ fiberings over $S^1$}
\end{table}

Next we consider $N^3_2\, | \, I$ fiberings. 

\begin{lemma}  
Let $\Mu = \langle t_1, t_2, t_3, \alpha\rangle$ with $\alpha = e_3/2 + A$ and $A$ defined as above,   
and let $\beta$ be an affinity of $E^3$ that normalizes $\Mu$ 
such that $\beta_\star$ has order $2$ and $\beta_\star$ does not fix a point of $E^3/\Mu$. 
Then $\beta_\star = (e_1/2 + e_2/2 + I)_\star$. 
The quotient of $E^3/\Mu$ by the action of  $\beta_\star$ is of type $N^3_1$. 
\end{lemma} 
\begin{proof}
If $\Omega(\beta_\star)$ has order 1, then $\beta_\star = (e_1/2 + e_2/2 + I)_\star$ by Lemmas 3 and 8. 
The quotient of $E^3/\Mu$ by the action of  $(e_1/2+e_2/2 + I)_\star$ is of type $N^3_1$, 
since the 3-space group $\langle t_1, \alpha, e_1/2+e_2/2 + I\rangle$ is the same as 
the 3-space group $\langle e_1/2+e_2/2 + I, e_1/2-e_2/2 + I, \alpha\rangle$, which has IT number 7 according to Table 1B of \cite{B-Z}. 

Now suppose $\Omega(\beta_\star)$ has order 2.  
Let $\beta = b+ B$ with $b+B \in N_A(\Mu)$. 
By replacing $\beta$ with $\alpha\beta$, if necessary, 
we may assume that $B$ is conjugate, in the point group of $N_A(\Mu)$, to the matrix part of one of representatives 2-6 in Table 29. 
It follows from Lemmas 3 and 8, that $\beta_\star$ fixes a point of $E^3/\Mu$, which is a contradiction. 
This is easy to prove for the diagonal matrices. 
The two nondiagonal matrices are conjugate in $\mathrm{SL}(3,\integers)$ to $D$ where $D$ transposes $e_1$ and $e_2$ and $De_3=-e_3$. 
which implies that $\beta_\star$ fixes a point of $E^3/\Mu$ by Lemmas 3 and 8. 
\end{proof}

By Lemma 22, the isometry $(e_1/2+e_2/2 + I)_\star$ is the unique order 2 affinity of $N^3_2$ that does not fix a point of $N^3_2$. 
By Theorem 10, there is exactly 1 affine equivalence class of flat $N^3_2\,|\, \Iota$ fiberings, which is described in Table 30 as explained in \S 8. 

\begin{table} 

\begin{tabular}{c|c|c|c|c|c}	
no. & mfd. & cS-fbr.  &  grp. & pair representatives  & s-fbrs.  \\ \hline	
1	&$N^4_{10}$ & $N^3_2\,|\,\Iota$ & $D_1$ &
 		$ \left(\begin{array}{rrrr} 
		\frac{1}{2} & 1 & 0 &  0 \\ \frac{1}{2} & 0 & 1 & 0 \\ 0 & 0 & 0 & 1\end{array}\right)$, 
		$ \left(\begin{array}{rrrr} 
		\frac{1}{2} & 1 & 0 &  0 \\ \frac{1}{2} & 0 & 1 & 0 \\ 0 & 0 & 0 & 1\end{array}\right)$ & $N^3_1, N^3_1$ \\ \hline
 \end{tabular}

\vspace{.2in}
\caption{The flat $N^3_2$ fiberings over $\Iota$}
\end{table}

\section{Fibrations of Closed Flat 4-Manifolds with Generic Fiber $N^3_3$} 

In this section, we describe the affine classification of the geometric co-Seifert fibrations of closed flat 4-manifolds with generic fiber 
the flat 3-manifold $N^3_3$.  Here $N^3_3 = E^3/\Mu$ with $\Mu = \langle t_1,t_2,t_3,\alpha, \beta\rangle$ with $\alpha = e_2/2 + e_3/2 + 
\mathrm{diag}(-1,-1,1)$, and $\beta = e_3/2 + \mathrm{diag}(1,-1,1)$. 

\begin{lemma} 
Let $\Mu = \langle t_1, t_2, t_3, \alpha, \beta\rangle$ with $\alpha $ and $\beta$ defined as above.  
Then we have
$$N_A(\Mu) = \{c+C: 2c_1, 2c_2 \in \integers\ and\ C = \mathrm{diag}(\pm 1, \pm 1,\pm 1)\},$$
and
$$\mathrm{Out}(\Mu) \cong (\integers/2\integers)^2 \times \{\pm I\}$$
with $\Omega((e_1/2+I)_\star)$ corresponding to $(1,0,I)$, and $\Omega((e_2/2+I)_\star)$ corresponding to $(0,1,I)$, 
and $\Omega((-I)_\star)$ corresponding to $(0,0,-I)$.  
\end{lemma}

It follows from Lemma 23 that $\mathrm{Out}(\Mu)$ is an elementary 2-group of order 8,   
and so $\mathrm{Out}(\Mu)$ has 8 conjugacy classes of inverse pairs.  
Hence, there are 8 affine equivalence classes of flat $N_3^3\,|\,S^1$ fiberings. 
These fiberings are described in Table 31 as explained in \S 8. 

\begin{table} 

\begin{tabular}{c|c|c|c|c}	
no. & mfd. & cS-fbr.  &  grp. & representative      \\ \hline	
1	&$N^4_8$ & $N^3_3\,|\,S^1$  & $C_1$ & 
	  	$ \left(\begin{array}{rrr} 
		1 &  0 & 0 \\ 0 & 1 & 0 \\ 0 & 0 & 1 \end{array}\right)$   \\ \hline
2	&$N^4_{10}$ & $N^3_3\,|\,S^1$ &  $C_2$ &
		$ \left(\begin{array}{rrrr} 
		0 & 1 &  0 &  0 \\ \frac{1}{2} & 0 & 1 & 0 \\  0 & 0 & 0 & 1\end{array}\right)$    \\ \hline
3	&$N^4_{11}$ & $N^3_3\,|\,S^1$ &  $C_2$ &
		$ \left(\begin{array}{rrrr} 
		\frac{1}{2} & 1 &  0 &  0 \\ 0 & 0 & 1 & 0 \\  0 & 0 & 0 & 1\end{array}\right)$    \\ \hline
4	&$N^4_{13}$ & $N^3_3\,|\,S^1$ &  $C_2$ &
		$ \left(\begin{array}{rrrr} 
		\frac{1}{2} & 1 &  0 &  0 \\ \frac{1}{2} & 0 & 1 & 0 \\  0 & 0 & 0 & 1\end{array}\right)$    \\ \hline
5	&$N^4_{30}$ & $N^3_3\,|\,S^1$  & $C_2$ & 
	  	$ \left(\begin{array}{rrr} 
		-1 &  0 & 0 \\ 0 & -1 & 0 \\ 0 & 0 & -1 \end{array}\right)$   \\ \hline
6	&$N^4_{31}$ & $N^3_3\,|\,S^1$ &  $C_2$ &
		$ \left(\begin{array}{rrrr} 
		 0 & -1 &  0 &  0 \\  \frac{1}{2} & 0 & -1 & 0 \\ 0 & 0 & 0 & -1\end{array}\right)$    \\ \hline
7	&$N^4_{32}$ & $N^3_3\,|\,S^1$ &  $C_2$ &
		$ \left(\begin{array}{rrrr} 
		  \frac{1}{2} & -1 &  0 &  0 \\ 0 & 0 & -1 & 0 \\ 0 & 0 & 0 & -1\end{array}\right)$    \\ \hline
8	&$N^4_{33}$ & $N^3_3\,|\,S^1$ &  $C_2$ &
		$ \left(\begin{array}{rrrr} 
		  \frac{1}{2} & -1 &  0 &  0 \\ \frac{1}{2} & 0 & -1 & 0 \\ 0 & 0 & 0 & -1\end{array}\right)$    \\ \hline			
 \end{tabular}

\vspace{.2in}
\caption{The flat $N^3_3$ fiberings over $S^1$}
\end{table}

Next we consider $N^3_3\, | \, I$ fiberings. 

\begin{lemma}  
Let $\Mu = \langle t_1, t_2, t_3, \alpha, \beta\rangle$ with $\alpha$ and $\beta$ defined as above,   
and let $\gamma$ be an affinity of $E^3$ that normalizes $\Mu$ 
such that $\gamma_\star$ has order $2$ and $\gamma_\star$ does not fix a point of $E^3/\Mu$. 
Then $\gamma_\star = (e_1/2 + I)_\star$. 
The quotient of $E^3/\Mu$ by the action of  $\gamma_\star$ is of type $N^3_3$. 
\end{lemma} 
\begin{proof}
By Lemma 23, we have that $\gamma_\star = (c\pm I)_\star$. 
If $\gamma_\star = (c-I)_\star$,   
then $\gamma_\star$ fixes a point of $E^3/\Mu$ by Lemmas 3 and 8, which is a contradiction. 
If $\gamma_\star = (c+I)_\star$, then $\gamma_\star = (e_1/2 + I)_\star$ by Lemmas 3 and 8. 
The quotient of $E^3/\Mu$ by the action of  $(e_1/2 + I)_\star$ is of type $N^3_3$, 
since the 3-space group $\langle t_2, \alpha, \beta, e_1/2+ I\rangle$ is of type $N^3_3$ according to CARAT. 
\end{proof}

By Lemma 24, the isometry $(e_1/2 + I)_\star$ is the unique order 2 affinity of $N^3_3$ that does not fix a point of $N^3_3$. 
By Theorem 10, there is exactly 1 affine equivalence class of flat $N^3_3\,|\, \Iota$ fiberings, which is described in Table 32 as explained in \S 8. 

\begin{table} 

\begin{tabular}{c|c|c|c|c|c}	
no. & mfd. & cS-fbr.  &  grp. & pair representatives  & s-fbrs.  \\ \hline	
1	&$N^4_{31}$ & $N^3_3\,|\,\Iota$ & $D_1$ &
 		$ \left(\begin{array}{rrrr} 
		\frac{1}{2} & 1 & 0 &  0 \\ 0 & 0 & 1 & 0 \\ 0 & 0 & 0 & 1\end{array}\right)$, 
		$ \left(\begin{array}{rrrr} 
		\frac{1}{2} & 1 & 0 &  0 \\ 0 & 0 & 1 & 0 \\ 0 & 0 & 0 & 1\end{array}\right)$ & $N^3_3, N^3_3$ \\ \hline
 \end{tabular}

\vspace{.2in}
\caption{The flat $N^3_3$ fiberings over $\Iota$}
\end{table}

\section{Fibrations of Closed Flat 4-Manifolds with Generic Fiber $N^3_4$} 

In this section, we describe the affine classification of the geometric co-Seifert fibrations of closed flat 4-manifolds with generic fiber 
the flat 3-manifold $N^3_4$.  Here $N^3_4 = E^3/\Mu$ with $\Mu = \langle t_1,t_2,t_3,\alpha, \beta\rangle$ with $\alpha = e_2/2 + e_3/2 + 
\mathrm{diag}(-1,-1,1)$, and $\beta = e_1/2 + \mathrm{diag}(1,-1,1)$. 

\begin{lemma}  
Let $\Mu = \langle t_1, t_2, t_3, \alpha, \beta\rangle$ with $\alpha $ and $\beta$ defined as above.  
Then we have
$$N_A(\Mu) = \{c+C: 2c_1, 2c_2 \in \integers\ and\ C = \mathrm{diag}(\pm 1, \pm 1,\pm 1)\},$$
and
$$\mathrm{Out}(\Mu) \cong (\integers/2\integers)^2 \times \{\pm I\}$$
with $\Omega((e_1/2+I)_\star)$ corresponding to $(1,0,I)$, and $\Omega((e_2/2+I)_\star)$ corresponding to $(0,1,I)$, 
and $\Omega((-I)_\star)$ corresponding to $(0,0,-I)$.  
\end{lemma}

It follows from Lemma 25 that $\mathrm{Out}(\Mu)$ is an elementary 2-group of order 8,   
and so $\mathrm{Out}(\Mu)$ has 8 conjugacy classes of inverse pairs.  
Hence, there are 8 affine equivalence classes of flat $N^3_4\,|\,S^1$ fiberings. 
These fiberings are described in Table 33 as explained in \S 8. 

\begin{table} 

\begin{tabular}{c|c|c|c|c}	
no. & mfd. & cS-fbr.  &  grp. & representative      \\ \hline	
1	&$N^4_9$ & $N^3_4\,|\,S^1$  & $C_1$ & 
	  	$ \left(\begin{array}{rrr} 
		1 &  0 & 0 \\ 0 & 1 & 0 \\ 0 & 0 & 1 \end{array}\right)$   \\ \hline
2	&$N^4_{11}$ & $N^3_4\,|\,S^1$ &  $C_2$ &
		$ \left(\begin{array}{rrrr} 
		0 & 1 &  0 &  0 \\ \frac{1}{2} & 0 & 1 & 0 \\  0 & 0 & 0 & 1\end{array}\right)$    \\ \hline
3	&$N^4_{12}$ & $N^3_4\,|\,S^1$ &  $C_2$ &
		$ \left(\begin{array}{rrrr} 
		\frac{1}{2} & 1 &  0 &  0 \\ 0 & 0 & 1 & 0 \\  0 & 0 & 0 & 1\end{array}\right)$    \\ \hline
4	&$N^4_{13}$ & $N^3_4\,|\,S^1$ &  $C_2$ &
		$ \left(\begin{array}{rrrr} 
		\frac{1}{2} & 1 &  0 &  0 \\ \frac{1}{2} & 0 & 1 & 0 \\  0 & 0 & 0 & 1\end{array}\right)$    \\ \hline
5	&$N^4_{34}$ & $N^3_4\,|\,S^1$  & $C_2$ & 
	  	$ \left(\begin{array}{rrrr} 
		  \frac{1}{2} & -1 &  0 &  0 \\ 0 & 0 & -1 & 0 \\ 0 & 0 & 0 & -1\end{array}\right)$   \\ \hline
6	&$N^4_{35}$ & $N^3_4\,|\,S^1$ &  $C_2$ &
		$ \left(\begin{array}{rrr} 
		-1 &  0 & 0 \\ 0 & -1 & 0 \\ 0 & 0 & -1 \end{array}\right)$   \\ \hline
7	&$N^4_{36}$ & $N^3_4\,|\,S^1$ &  $C_2$ &
		$ \left(\begin{array}{rrrr} 
		  \frac{1}{2} & -1 &  0 &  0 \\ \frac{1}{2} & 0 & -1 & 0 \\ 0 & 0 & 0 & -1\end{array}\right)$    \\ \hline
8	&$N^4_{37}$ & $N^3_4\,|\,S^1$ &  $C_2$ &
		$ \left(\begin{array}{rrrr} 
		 0 & -1 &  0 &  0 \\  \frac{1}{2} & 0 & -1 & 0 \\ 0 & 0 & 0 & -1\end{array}\right)$    \\ \hline			
 \end{tabular}

\vspace{.2in}
\caption{The flat $N^3_4$ fiberings over $S^1$}
\end{table}

Next we consider $N^3_4\, | \, I$ fiberings. 

\begin{lemma}  
Let $\Mu = \langle t_1, t_2, t_3, \alpha, \beta\rangle$ with $\alpha$ and $\beta$ defined as above,   
and let $\gamma$ be an affinity of $E^3$ that normalizes $\Mu$ 
such that $\gamma_\star$ has order $2$ and $\gamma_\star$ does not fix a point of $E^3/\Mu$. 
Then $\gamma_\star = (e_2/2 + I)_\star$. 
The quotient of $E^3/\Mu$ by the action of  $\gamma_\star$ is of type $N^3_3$. 
\end{lemma} 
\begin{proof}
By Lemma 25, we have that $\gamma_\star = (c\pm I)_\star$. 
If $\gamma_\star = (c-I)_\star$,   
then $\gamma_\star$ fixes a point of $E^3/\Mu$ by Lemmas 3 and 8, which is a contradiction. 
If $\gamma_\star = (c+I)_\star$, then $\gamma_\star = (e_2/2 + I)_\star$ by Lemmas 3 and 8. 
The quotient of $E^3/\Mu$ by the action of  $(e_2/2 + I)_\star$ is of type $N^3_3$, 
since the 3-space group $\langle t_1, \alpha, \beta, e_2/2+ I\rangle$ is of type $N^3_3$ according to CARAT. 
\end{proof}

By Lemma 26, the isometry $(e_2/2 + I)_\star$ is the unique order 2 affinity of $N^3_4$ that does not fix a point of $N^3_4$. 
By Theorem 10, there is exactly 1 affine equivalence class of flat $N^3_4\,|\, \Iota$ fiberings, which is described in Table 34 as explained in \S 8. 

\begin{table} 

\begin{tabular}{c|c|c|c|c|c}	
no. & mfd. & cS-fbr.  &  grp. & pair representatives  & s-fbrs.  \\ \hline	
1	&$N^4_{30}$ & $N^3_4\,|\,\Iota$ & $D_1$ &
 		$ \left(\begin{array}{rrrr} 
		0 & 1 & 0 &  0 \\ \frac{1}{2} & 0 & 1 & 0 \\ 0 & 0 & 0 & 1\end{array}\right)$, 
		$ \left(\begin{array}{rrrr} 
		0 & 1 & 0 &  0 \\ \frac{1}{2} & 0 & 1 & 0 \\ 0 & 0 & 0 & 1\end{array}\right)$ & $N^3_3, N^3_3$ \\ \hline
 \end{tabular}

\vspace{.2in}
\caption{The flat $N^3_4$ fiberings over $\Iota$}
\end{table}

\end{document}